\documentclass[11pt,leqno]{amsart}

\makeatletter
\renewcommand\subsection{\@startsection{subsection}{2}%
  \z@{.5\linespacing\@plus.7\linespacing}{-.5em}%
  {\normalfont\scshape}}
\makeatother

\usepackage[a4paper,lmargin=2cm,rmargin=2cm,tmargin=4cm,bmargin=4cm]{geometry}

\usepackage[centertags]{amsmath}
\usepackage{amsfonts}
\usepackage{amssymb}
\usepackage{amsthm}
\usepackage{dsfont}

\numberwithin{equation}{section}

\usepackage[pdftex]{color}
\usepackage[bookmarks=true,hyperindex,pdftex,
colorlinks, citecolor=blue,linkcolor=blue, urlcolor=blue
]{hyperref}

\usepackage{tikz}
\usepackage{tikz-cd}
\usetikzlibrary{decorations.markings}

\usepackage[normalem]{ulem}
\usepackage{bbm}

\newcommand{\R}{\mathbb{R}}
\newcommand{\C}{\mathbb{C}}
\newcommand{\K}{\mathbb{K}}
\newcommand{\T}{\mathbb{T}}

\newcommand{\N}{\mathbb{N}}
\newcommand{\natu}{\mathbb{N}}
\newcommand{\F}{\mathcal{F}}
\newcommand{\V}{\mathcal V}

\DeclareMathOperator{\diam}{diam\,}
\DeclareMathOperator{\co}{co}

\DeclareMathOperator{\conv}{co}
 \DeclareMathOperator{\cconv}{\overline{\conv}}
 \DeclareMathOperator{\cco}{\overline{\conv}}
\DeclareMathOperator{\re}{Re}
\DeclareMathOperator{\supp}{supp}

 \newcommand{\Id}{\mathrm{Id}}

\newcommand{\ext}[1]{\operatorname{ext}\left(#1\right)}

\newcommand{\preext}[1]{\operatorname{pre-ext}\left(#1\right)}
\newcommand{\dent}[1]{\operatorname{dent}\left(#1\right)}
\newcommand{\norm}[1]{\left\Vert#1\right\Vert}
\newcommand{\abs}[1]{\left\vert#1\right\vert}
\newcommand{\eps}{\varepsilon}

\DeclareMathOperator{\spn}{span}
\DeclareMathOperator{\cspan}{\overline{\spn}}

\renewcommand{\geq}{\geqslant}
\renewcommand{\leq}{\leqslant}

\theoremstyle{plain}
\newtheorem{thm}{Theorem}[section]
\newtheorem{cor}[thm]{Corollary}
\newtheorem{lem}[thm]{Lemma}
\newtheorem{prop}[thm]{Proposition}
\newtheorem{theorem}[thm]{Theorem}
\newtheorem{lemma}[thm]{Lemma}
\newtheorem{proposition}[thm]{Proposition}
\newtheorem{corollary}[thm]{Corollary}

\theoremstyle{definition}
\newtheorem{defn}[thm]{Definition}
\newtheorem{remark}[thm]{Remark}
\newtheorem{rem}[thm]{Remark}
\newtheorem{expl}[thm]{Example}
\newtheorem{example}[thm]{Example}

\newtheorem{quest}[thm]{Question}
\newtheorem{obs}[thm]{Observation}

\begin{document}

\title{Diametral notions for elements of the unit ball of a Banach space}

\date{January 11th, 2023}

\author[Mart\'{\i}n]{Miguel Mart\'{\i}n}
\address[Mart\'{\i}n]{Universidad de Granada, Facultad de Ciencias.
Departamento de An\'{a}lisis Matem\'{a}tico, 18071 Granada, Spain \newline
	\href{http://orcid.org/0000-0003-4502-798X}{ORCID: \texttt{0000-0003-4502-798X} }}
\email{mmartins@ugr.es}
\urladdr{\url{https://www.ugr.es/local/mmartins}}

\author[Perreau]{Yo\"{e}l Perreau}
\address[Perreau]{University of Tartu, Institute of Mathematics and Statistics, Narva mnt 18, 51009 Tartu linn, Estonia \newline
	\href{https://orcid.org/0000-0002-2609-5509}{ORCID: \texttt{0000-0002-2609-5509}}}
 \email{\texttt{yoel.perreau@ut.ee}}

\author[Rueda Zoca]{Abraham Rueda Zoca}
\address[Rueda Zoca]{Universidad de Granada, Facultad de Ciencias.
Departamento de An\'{a}lisis Matem\'{a}tico, 18071 Granada, Spain
 \newline
\href{https://orcid.org/0000-0003-0718-1353}{ORCID: \texttt{0000-0003-0718-1353} }}
\email{\texttt{abrahamrueda@ugr.es}}
\urladdr{\url{https://arzenglish.wordpress.com}}

\thanks{The first and third named authors were supported by grant PID2021-122126NB-C31 funded by MCIN/AEI/ 10.13039/501100011033 and “ERDF A way of making Europe”, by Junta de Andaluc\'ia I+D+i grants P20\_00255 and FQM-185, and by ``Maria de Maeztu'' Excellence Unit IMAG, reference CEX2020-001105-M funded by MCIN/AEI/10.13039/501100011033. The second named author was supported by the Estonian Research Council grant SJD58.}

\begin{abstract}
We introduce extensions of $\Delta$-points and Daugavet points in which slices are replaced by relative weakly open subsets (super $\Delta$-points and super Daugavet points) or by convex combinations of slices (ccs $\Delta$-points and ccs Daugavet points). These notions represent the extreme opposite to denting points, points of continuity, and strongly regular points. We first give a general overview on these new concepts and provide some isometric consequences on the spaces. As examples: if a Banach space contains a super $\Delta$-point, then it does not admit an unconditional FDD (in particular, unconditional basis) with suppression constant smaller than two; if a real Banach space contains a ccs $\Delta$-point, then it does not admit a one-unconditional basis; if a Banach space contains a ccs Daugavet point, then every convex combination of slices of its unit ball has diameter two. We next characterize the notions in some classes of Banach spaces showing, for instance, that all the notions coincide in $L_1$-predual spaces and that all the notions but ccs Daugavet points coincide in $L_1$-spaces. We next remark on some examples which have previously appeared in the literature and provide some new intriguing examples: examples of super $\Delta$-points which are as closed as desired to strongly exposed points (hence failing to be Daugavet points in an extreme way); an example of a super $\Delta$-point which is strongly regular (hence failing to be a ccs $\Delta$-point in the strongest way); a super Daugavet point which fails to be a ccs $\Delta$-point. The extensions of the diametral notions to point in the open unit ball and the consequences on the spaces are also studied. Last, we investigate the Kuratowski measure of relative weakly open subsets and of convex combinations of slices in the presence of super $\Delta$-points or ccs $\Delta$-points, as well as for spaces enjoying diameter 2 properties. We conclude the paper with a section on open problems.
\end{abstract}

\maketitle



\hypersetup{linkcolor=black}

\makeatletter \def\l@subsection{\@tocline{2}{0pt}{1pc}{5pc}{}} \def\l@subsection{\@tocline{2}{0pt}{3pc}{6pc}{}} \makeatother

\tableofcontents

\hypersetup{linkcolor=blue}

\parskip=1ex

\section{Introduction}\label{sec:introduction}
It is fair to say that one of the most studied properties of Banach spaces is the \textit{Radon-Nikod\'{y}m property (RNP)} because it has shown to be very useful; due to the large amount of its geometric, analytic, and measure theoretic characterisations; in several fields of Banach space theory such as representation of bounded linear operators, representation of dual spaces or representation of certain tensor product spaces (see \cite{bour,dis}).

A famous geometric characterization of the Radon-Nikod\'{y}m property is related to the size of slices. Recall that a \emph{slice} of a bounded non-empty subset $C$ of a Banach space $X$ is simply the (nonempty) intersection of $C$ with a half-space. A Banach space $X$ has the RNP if and only if every non-empty closed and bounded subset of $X$ admits slices of arbitrarily small diameter (see e.g.\ \cite{bour}).

A closely related and equally important geometric property of Banach spaces is the point of continuity property. Recall that a Banach space $X$ has the \textit{point of continuity property (PCP)} if every non-empty closed and bounded subset of $X$ admits non-empty relatively weakly open subsets of arbitrarily small diameter. Let us emphasize here as an example the striking equivalence between the Radon-Nikod\'{y}m property and the weak$^*$ version of the point of continuity property for dual spaces, and the related characterization of Asplund spaces as preduals of RNP spaces (see e.g.\ \cite{DGZ}).
In his proof of the determination of the Radon-Nikod\'{y}m property by subspaces with a finite dimensional decomposition (FDD) in \cite{Bourgain80}, Bourgain also introduced an important weakening of the point of continuity property, that he called property ``$(*)$'', and that is nowadays referred to as the convex point of continuity property. Recall that a Banach space $X$ has the \textit{convex point of continuity property (CPCP)} if every non-empty closed, convex and bounded subset of $X$ admits non-empty relatively weakly open subsets of arbitrarily small diameter.

In fact, Bourgain implicitly used in his work the notion of strong regularity which, as he showed, is implied by the CPCP. Recall that a Banach space $X$ is \textit{strongly regular (SR)} if every non-empty closed, convex and bounded subset of $X$ contains \emph{convex combinations of slices} of arbitrarily small diameter. Observe that the convexity of the subset is required in this definition in order to guarantee that it contains all the convex combinations of its slices. It later turned out that strong regularity had important applications to the famous (still open) question of the equivalence between the Radon-Nikod\'{y}m property and the Krein-Milman property. Recall that a Banach space $X$ has the \emph{Krein-Milman property (KMP)} if every non-empty closed, convex and bounded subset $C$ of $X$ admits an extreme point. The RNP implies the KMP (see e.g.\ \cite[Theorem~3.3.6]{bour}), and it follows from \cite{Schachermayer87} that every strongly regular space with the KMP has the RNP. Also recall that it was proved in \cite{HM} the RNP and the KMP are equivalent in dual spaces.

From the definitions it follows that RNP$\Rightarrow$PCP$\Rightarrow$CPCP and it is also known that CPCP$\Rightarrow$SR. None of the above implications reverse (see e.g \cite{Schachermayer88} and references therein). In order to show that strong regularity is implied by the CPCP, Bourgain made an important geometric observation, namely that in every non-empty bounded and convex subset of a Banach space $X$, every non-empty relatively weakly open subset contains a convex combination of slices. We will discuss this ``Bougain Lemma'' and its applications to the subject of the present paper in more details in Section \ref{sec:notation_and_prelimiary_results}.

Another classical refinement of the above characterization of the Radon-Nikod\'{y}m property is related to the notion of denting points. Recall that a point $x_0$ of a bounded subset $C$ of $X$ is a \textit{denting point} of $C$ if there are slices of $C$ containing $x_0$ of arbitrarily small diameter. A Banach space $X$ has the RNP if and only if every closed, convex and bounded subset contains a denting point. Actually, every nonempty closed, convex and bounded subset $C$ of a Banach space $X$ with the RNP is equal to the closure of the convex hull of the set of its denting points (see e.g.\ \cite[Corollary 3.5.7]{bour}).

For the PCP and the CPCP, a similar role is played by points of weak-to-norm continuity. Given a bounded subset $C$ of $X$, we say that a point $x_0\in C$ is a \textit{point of weak-to-norm continuity} (\emph{point of continuity} in short) if the identity mapping $i\colon (C,w)\longrightarrow (C,\tau)$ is continuous at the point $x_0$ or, equivalently, if $x_0$ belongs to relatively weakly open subsets of $C$ of arbitrarily small diameter. Note that a classical result by Lin-Lin-Troyanski \cite{llt88} establishes that a point $x_0\in C$ is a denting point if, and only if, $x_0$ is simultaneously a point of continuity and a extreme point of $C$. In a space with the PCP every non-empty closed and bounded subset contains a point of continuity; and the set of all points of continuity of a given closed, convex and bounded subset $C$ of a Banach space $X$ with the CPCP is weakly dense in $C$ (see e.g.\ \cite[Theorem 1.13]{ew}).

In relation to strong regularity, a point $x_0$ of a bounded, convex subset $C$ of $X$ is a \textit{point of strong regularity} if there are convex combinations of slices of $C$ containing $x_0$ of arbitrarily small diameter. Then the set of all points of strong regularity of a given closed, convex and bounded subset $C$ of a strongly regular Banach space $X$ is norm dense in $C$ (see \cite[Theorem 3.6]{ggms}). Let us observe that points of strong regularity may be in the interior of a set, while denting points (and points of continuity in the infinite-dimensional case) belong always to the border of the set.

In \cite{ahlp20} the extreme opposite notion to denting point of the unit ball was introduced in the following sense: an element $x$ in the unit sphere of a Banach space $X$ is a \emph{$\Delta$-point} if we can find in every slice of $B_X$ containing $x$ points which are at distance from $x$ as close as we wish to the maximal possible distance in the ball (distance $2$). A similar yet stronger notion appeared simultaneously in relation to another quite famous property of Banach spaces, the Daugavet property. Recall that a Banach space $X$ has the \textit{Daugavet property (DPr)} if the Daugavet equation
\begin{equation}\label{DE}\tag{DE}
\Vert \Id+T\Vert=1+\Vert T\Vert
\end{equation}
holds for every rank-one operator $T\colon X\longrightarrow X$,
where $\Id$ denotes the identity operator. In this case, all weakly compact operators also satisfy \eqref{DE}. We refer the reader to the seminal paper \cite{KSSW} for background. Recent results can be found in \cite{mr22} and references therein. The Daugavet property admits a beautiful geometric characterization involving slices related to the notion of Daugavet points: an element $x$ on the unit sphere of a Banach space $X$ is a \emph{Daugavet point} if in every slice of $B_X$ (not necessarily containing the point $x$) there are points which are at distance from $x$ as close as we wish to $2$. With this definition in mind, \cite[Lemma 2.1]{KSSW} states that $X$ has the DPr if and only if all elements in $S_X$ are Daugavet points. Let us comment that the Daugavet property imposes severe restriction on the Banach space: if $X$ is a Banach space with the DPr, then it fails the RNP and it has no unconditional basis (actually, it cannot be embedded into a Banach space with unconditional basis).

On the other hand, $\Delta$- and Daugavet points have proved to be far more flexible than the global properties that they define. For example, there exists a Banach space with the RNP and a Daugavet point \cite{veefree} (see paragraph~\ref{subsubsec:extreme_molecules}), there exists a Banach space with a one-unconditional basis and a large subset of Daugavet points \cite{almt21} (see paragraph~\ref{subsubsec:oneunconditionalbasis}), and there is an MLUR Banach space for which all elements in its unit sphere are $\Delta$-points, which contains convex combinations of slices of arbitrarily small diameter, but satisfying that every convex combination of slices intersecting its unit sphere has diameter two \cite{ahntt16} (see paragraph~\ref{subsubsect:MLUR}). Nonetheless, it has been recently proved that $\Delta$-points have some influence on the isometric structure of the space. For example, it is shown in \cite{almp22} that \emph{uniformly non-square} spaces do not contain $\Delta$-points; actually, it has been very recently proved in \cite{KaLeeTag2022} that a $\Delta$-point cannot be a \emph{locally uniformly non-square point}. Also, combining the results from  \cite{almp22} and \cite{veelipfunc}, \emph{asymptotic uniformly smooth} spaces and their duals do not contain $\Delta$-points. However, it is still an important open problem to understand whether $\Delta$- or Daugavet point have any influence on the isomorphic structure of the space.

In this paper, our main aim is to study natural strengthening of the notions of Daugavet- and $\Delta$-points obtained by replacing slices by non-empty relatively weakly open subsets (``super points'') or convex combination of slices (``ccs points'') in order to provide new diametral notions which are extreme opposites to points of continuity and to strongly regular points, respectively. See Definitions \ref{def:defiDauganot} and \ref{def:defiDeltanot} for details. Our main goal will be to understand the influence, for a given Banach space, of the existence of such points on its geometry, and to study the different diametral notions in several families of Banach spaces. A particular emphasis will be put on trying to distinguish between all the various formally different notions.

Let us end this section by giving a brief description about the organization of the paper and the main results obtained. Section~\ref{sec:notation_and_prelimiary_results} contains the necessary notation (which is standard, anyway), needed definitions, and some preliminary results. We include in Section~\ref{sec:characterizations_and_geometric_consequences} some characterizations of the newer diametral point notions and some necessary conditions on the existence of such points. In particular, we study the existence of super $\Delta$-points and ccs $\Delta$-points in spaces with a one-unconditional basis. We first give an analogue for ccs $\Delta$-points to a result from \cite{almt21} which implicitly states that such spaces contain no super $\Delta$-points. Second, we provide sharper and improved versions of this super $\Delta$ result in the context of unconditional FDDs with a small unconditional constant, and more generally in the context of spaces in which special families of operators are available. The section finishes with the study of the behaviour of super $\Delta$-points and super Daugavet points with respect to absolute sums somehow analogous to the known one for $\Delta$-points and Daugavet points; however, not all the results extend to ccs $\Delta$-points and ccs Daugavet points, but we also give some partial results. Section~\ref{sec:examples_and_counterexamples} is devoted to examples and counterexamples. We first characterize the diametral notions in some families of classical Banach spaces: we show that all notions are equivalent in $L_1$-preduals and M\"{u}ntz spaces (Subsection~\ref{subsec:L_1_preduals}); all notions but ccs Daugavet points also coincide in $L_1$-spaces (Subsection~\ref{subsec:L_1_spaces}). We next give in Subsection~\ref{subsect:fromliterature} some remarks on examples which have previously appeared in the journal literature, discussing the new  diametral notions on them, and showing that they may help to distinguish between the diametral notions. The most complicated and tricky examples are produced in the last three subsection of this section: super $\Delta$-points which are as closed as desired to strongly exposed points (hence failing to be Daugavet points in an extreme way) in Subsection~\ref{subsect:superdeltanotdauga}; a super $\Delta$-point which is strongly regular (hence failing to be ccs $\Delta$-point in an extreme way) in Subsection~\ref{subsect:superDeltastronglyregular}; super Daugavet points which belong to convex combinations of slices of diameter as small as desired (hence failing to be ccs $\Delta$-points in an extreme way). We finish this subsection with a summary of relations between all the diametral notions. The idea in Section~\ref{sec:inside_of_the_unit_ball} is to generalize the diametral notions to elements of the open unit ball, and use these notions to characterize some geometric properties. In particular, we properly localize the result by Kadets that the DSD2P is equivalent to the Daugavet property.  Section~\ref{section:Kuratowski_measure} deals with Kuratowki index of non-compactness of slices, relative weakly open subsets, and convex combinations of slices. We get that every relative weakly open subset (respectively, every convex combination of slices) in a space with the diameter 2 property (respectively, with the strong diameter 2 property) has Kuratowski measure 2; these results extends the analogous result for slices and the the local diameter 2 property proved in \cite[Proposition~3.1]{DGKR}. Also, we show that every relative weakly open subset that contains a super $\Delta$-point has Kuratowski measure 2, and a similar result is obtained with convex combinations of relative weakly open subsets containing a ccw $\Delta$-point; these results extend  \cite[Corollary~2.2]{veelipfunc}.
Finally, Section~\ref{section:conclusionopenquestions} is devoted to collect some interesting open questions and some remarks on them.

\section{Notation and preliminary results} \label{sec:notation_and_prelimiary_results}
We will use standard notation as in the books \cite{alka}, \cite{checos}, and \cite{fab}, for instance. Given a Banach space $X$, $B_X$ (respectively, $S_X$) stands for the closed unit ball (respectively, the unit sphere) of $X$. We denote by $X^*$ the topological dual of $X$ and we write $J_X\colon X\longrightarrow X^{**}$ for the canonical injection. We denote by $\dent{B_X}$ and $\ext{B_X}$ the sets
of all denting points of $B_X$ and of all extreme points of $B_X$, respectively. The set of preserved extreme points of $B_X$ (i.e.\ those $x\in B_X$ such that $J_X(x)\in \ext{B_{X^{**}}}$) is denoted by $\preext{B_X}$. For Banach spaces $X$ and $Y$, $\mathcal{L}(X,Y)$, $\mathcal{F}(X,Y)$, $\mathcal{K}(X,Y)$ denote, respectively, the set of all (bounded linear) operator, the finite-rank operators, and the compact operators.
The properties in which we are interested only deal with the real structure of the involved Banach spaces, but we do not restrict the study to real spaces in order to consider real or complex examples. We will use the notation $\K$ to denote either $\R$ or $\C$, $\re(z)$ to denote the real part of $z$ (which is just the identity when dealing with a real space), and $\T$ to represent the set of scalars of modulus one.

Given a non-empty subset $C$ of $X$, we will denote by $\conv(C)$ the convex hull of $C$ and by $\spn(C)$ the linear hull of $C$. Also we denote by $\cconv(C)$ (respectively, $\cspan(C)$) the norm closure of the convex hull (respectively, of the linear hull) of $C$. By a slice of $C$ we will mean any subset of $C$ of the form
$$
S(x^*,\delta;C):=\left\{x\in C\colon \re x^*(x)>M-\delta\right\}
$$
where $x^*\in X^*$ is a continuous linear functional on $X$, $\delta>0$ is a positive real number, and $M:=\sup_{x\in C} \re x^*(x)$. For slices of the unit ball we will simply write $S(x^*,\delta):=S(x^*,\delta;B_X)$. By a \emph{relatively weakly open subset} of $C$ we mean as usual any subset of $C$ obtained as the (non-empty) intersection of $C$ with an open set of $X$ in the weak topology.

If $C$ is assumed to be convex we will mean by a \textit{convex combination of slices}\index{slice!convex combination of slices} of $C$ (\emph{ccs} of $C$ in short) any subset of $C$ of the form
$$\sum\nolimits_{i=1}^n \lambda_i S_i,$$
where $\lambda_1,\ldots, \lambda_n\in (0,1]$ are such that $\sum_{i=1}^n \lambda_i=1$ and $S_i$ is a slice of $C$ for every $i\in\{1,\ldots, n\}$. Observe that convex combinations of slices are convex sets. We define in the same way \emph{convex combinations of relatively weakly open subsets of $C$} (\emph{ccw} of $C$ in short).

The following lemma from \cite{ik} is a very useful tool when working with $\Delta$-points.

\begin{lemma}[\textrm{\cite[Lemma 2.1]{ik}}]\label{lemma:diminition} Let $X$ be a Banach space, and let $x^*\in S_{X^*}$ and $\alpha>0$. For every $x\in S(x^*,\alpha)$ and every $0<\beta<\alpha$ there exists $y^*\in S_{X^*}$ such that
$$x\in S(y^*,\beta)\subseteq S(x^*,\alpha).$$
\end{lemma}

We also often rely on the following result, due to Bourgain, and that we already mentioned in the introduction. We provide a proof below, following the one from \cite[Lemma II.1]{ggms}, for the sake of completeness and for further discussions.

\begin{lem}[\textrm{Bourgain}]\label{lem:bourgain_lemma}
Let $X$ be a Banach space and let $C$ be a bounded convex closed subset of $X$. Then, every non-empty relatively weakly open subset $W$ of $C$ contains a convex combination of slices of $C$.
\end{lem}

\begin{proof}
Assume with no loss of generality that $W:=\bigcap\limits_{i=1}^m S(f_i,\alpha_i,C)$, write $\widetilde{C}=\overline{J_X(C)}^{w^*}\subset X^{**}$, and
$$
W^{**}:=\bigcap\limits_{i=1}^m S\bigl(J_{X^*}(f_i),\alpha_i;\widetilde{C}\bigr),
$$
which is a non-empty relatively weak$^*$ open subset of $\widetilde{C}$. By the Krein-Milman theorem (see e.g.\ \cite[Theorem 3.37]{fab}), it follows that $\widetilde{C}=\overline{\co(\ext{B_{X^{**}}})}^{w^*}$, so $\co(\ext{B_{X^{**}}})\cap W^{**}\neq \emptyset$. Pick a convex combination of extreme points $\sum_{i=1}^n \lambda_i e_i^{**}$ contained in $W^{**}$. By the continuity of the sum we can find, for every $1\leq i\leq n$, a weak-star open subset $W_i^{**}$ with $e_i^{**}\in W_i^{**}$ and such that $\sum_{i=1}^n\lambda_{i}W_i^{**}\subset W^{**}$.

Now, since each $e_i^{**}$ is an extreme point of $\widetilde{C}$, we have by Choquet's lemma (see \cite[Lemma~3.40]{fab}, for instance) that there are weak-star slices $S\bigl(J_{X^*}(g_i),\beta_i;\widetilde{C}\bigr)$ with $e_i^{**}\in S\bigl(J_{X^*}(g_i),\beta_i;\widetilde{C}\bigr)\subseteq W_i^{**}$ for every $i\in\{1,\dots,n\}$. Henceforth, $\sum_{i=1}^n \lambda_i S\bigl(J_{X^*}(g_i),\beta_i;\widetilde{C}\bigr)\subseteq \sum_{i=1}^n \lambda_i W_i^{**}\subseteq W^{**}$. Now, if we take
$$U:=\sum_{i=1}^n \lambda_i S(g_i,\beta_i,C)$$
it is not difficult to prove that $U\subseteq W$, as desired.
\end{proof}

\begin{remark}\label{rem:bourgain}
Observe that, in general, it is unclear from the above proof whether or not, if we fix $x\in W$, we can guarantee that there exists a convex combination of slices $U$ of $C$ such that $x\in U\subseteq W$.\parskip=0ex

On the other hand, the result holds true if $x\in W\cap \co(\preext{C})$ in view of the above proof. Indeed, if such situation, if we write $x=\sum_{i=1}^n \lambda_i x_i\in W$ satisfying that $x_1,\ldots, x_n\in \preext{C}$ and $\lambda_1,\ldots, \lambda_n\in (0,1]$ with $\sum_{i=1}^n \lambda_i=1$, by the weak continuity of the sum, we can find, for every $1\leq i\leq n$, a non-empty relatively weakly open subset $V_i$ with $x_i\in V_i$ for every $i$ and such that $x=\sum_{i=1}^n \lambda_i x_i\in \sum_{i=1}^n \lambda_i V_i\subseteq W$. Now, observe that, since each $x_i$ is a preserved extreme point of $C$, slices of $C$ containing $x_i$ are a neighbourhood basis for $x_i$ in the weak topology. Hence, we can find, for $1\leq i\leq n$, a slice $S_i$ of $C$ with $x_i\in S_i\subseteq V_i$, and so $x=\sum_{i=1}^n \lambda_i x_i\in \sum_{i=1}^n\lambda_i S_i\subseteq \sum_{i=1}^n \lambda_i V_i\subseteq W$, so $U:=\sum_{i=1}^n\lambda_i S_i$ is the desired convex combination of slices.
\end{remark}

Throughout the text, we will often be discussing various ``diameter 2 properties''. We use the notation introduced in \cite{aln13}. A Banach space $X$ has the \emph{local or slice diameter 2 property} (\emph{LD2P}) if every slice of $B_X$ has diameter 2; $X$ has the \emph{diameter two property} (\emph{D2P}) if every non-empty relatively weakly open subset of $B_X$ has diameter 2; finally, $X$ has the \emph{strong diameter 2 property} (\emph{SD2P}) whenever every ccs of $B_X$ has diameter 2 (and then, every ccw has diameter 2 due to Lemma~\ref{lem:bourgain_lemma}). For definitions and for examples concerning those properties, we refer to \cite{ahntt16, blradv, blr14, pirkthesis}. In particular, let us comment that the three properties are different, a result which was not easy to show, see \cite{blradv}. Our paper is closely related to the diametral versions of those properties which have been implicitly studied for a long time in the literature, but whose formal definitions and names where fixed in \cite{blr18}. A Banach space $X$ has the \emph{diametral local diameter $2$ property} (\emph{DLD2P}) if for every slice $S$ of $B_X$ and every $x\in S\cap S_X$, $\sup_{y\in S}\|x-y\|=2$; if slices are replaced by non-empty relatively weakly open subsets of $B_X$, we obtain the \emph{diametral diameter $2$ property} (\emph{DD2P}). It is immediate that these properties are not satisfied by any finite-dimensional space. Clearly, DLD2P implies LD2P, DD2P implies D2P (and none of these implications reverses, e.g.\ $X=c_0$), and DD2P implies DLD2P. It is unknown whether the DLD2P and the DD2P are equivalent; in fact it is even unknown whether the DLD2P implies the D2P. For the analogous definition using ccs, we have to discuss a little bit. Even for an infinite-dimensional space $X$, it is not true that every ccs of $B_X$ intersects $S_X$; actually, this happens if and only if $X$ has a property stronger than the SD2P (see \cite[Theorem 3.4]{lmr19}). Thus, the definition of the \emph{diametral strong diameter 2 property} (\emph{DSD2P}) given in \cite{blr18} deals with all points in $B_X$ as follows: for every ccs $C$ and every $x\in C$, $\sup_{y\in C}\|x-y\|=\|x\|+1$. This definition allows to show that DSD2P implies the SD2P. But, actually, it has been recently shown by V.\ Kadets \cite{kadets20} that the DSD2P is equivalent to the Daugavet property. We will discuss this in detail in Section \ref{sec:inside_of_the_unit_ball}. On the other hand, we will use the following property which is weaker than the DSD2P: a Banach space $X$ has the \emph{restricted DSD2P} if for every ccs $C$ and every $x\in C\cap S_X$, $\sup_{y\in C}\|x-y\|=2$.  This property is strictly weaker than the DSD2P, see Paragraph~\ref{subsubsect:MLUR}.

Let us now introduce all the notions of diametral points that we will consider in the text. Let us start with the more closely related ones to the definitions above.

\begin{defn}\label{def:defiDeltanot}
Let $X$ be a Banach space and let $x\in S_X$. We say that
\begin{enumerate}
\item \cite{ahlp20} $x$ is a \emph{$\Delta$-point}  if  $\sup_{y\in S}\norm{x-y}=2$ for every slice $S$ of $B_X$ containing $x$,
\item$x$ is a \emph{super $\Delta$-point} if  $\sup_{y\in V}\norm{x-y}=2$ for every non-empty relatively weakly open subset $V$ of $B_X$ containing $x$,
\item $x$ is a \emph{ccs $\Delta$-point} if  $\sup_{y\in C}\norm{x-y}=2$ for every slice ccs $C$ of $B_X$ containing $x$.
\end{enumerate}
\end{defn}

$\Delta$-points were introduced in \cite{ahlp20} as a natural localization of the DLD2P (i.e.\ $X$ has the DLD2P if and only if every element of $S_X$ is a $\Delta$-point). The other two definitions are new. Clearly, super $\Delta$-points are the natural localization of the DD2P: $X$ has the DD2P if and only if every element of $S_X$ is a super $\Delta$-point. Besides, ccs $\Delta$-points are the localization of the restricted DSD2P: $X$ has the restricted DSD2P if and only if every element of $S_X$ is a ccs $\Delta$-point.

In relation with the Daugavet property, we have the following notions for points.

\begin{defn}\label{def:defiDauganot} Let $X$ be a Banach space and let $x\in S_X$. We say that
\begin{enumerate}
\item \cite{ahlp20} $x$ is a \emph{Daugavet point} if $\sup_{y\in S}\norm{x-y}=2$ for every slice $S$ of $B_X$,
\item $x$ is a \emph{super Daugavet point} if $\sup_{y\in V}\norm{x-y}=2$ for every non-empty relatively weakly open subset $V$ of $B_X$,
\item $x$ is a \emph{ccs Daugavet point} if $\sup_{y\in C}\norm{x-y}=2$ for every ccs $C$ of $B_X$.
\end{enumerate}
\end{defn}

Let us recall that Daugavet points were introduced in \cite{ahlp20} as a natural localization of the \emph{Daugavet property} in the sense that a Banach space $X$ has the Daugavet property if and only if every point in $S_X$ is a Daugavet point (\cite[Lemma~2.1]{KSSW}). From the geometric characterization given in \cite[Lemma~3]{shv00} and the implicit result contained in its proof, it follows that super Daugavet points as well as ccs Daugavet points are also natural localizations of the Daugavet property.

Since every slice of $B_X$ is relatively weakly open, and since by Bourgain's lemma (see Lemma~\ref{lem:bourgain_lemma}) every non-empty relatively weakly open subset of $B_X$ contains a ccs of $B_X$, we clearly have the diagram of Figure~\ref{figure01}.

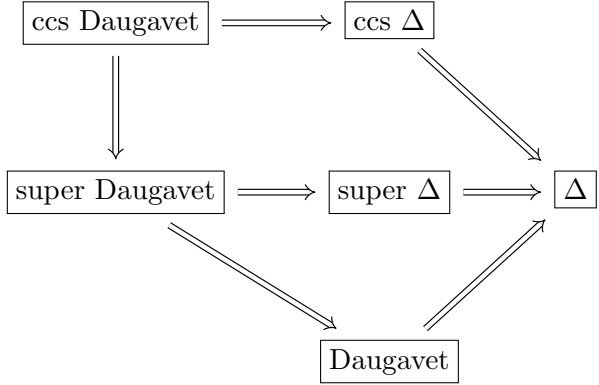
\begin{figure}\centering
    \begin{tikzcd}\centering
\fbox{\text{ccs Daugavet}} \arrow[dd, Rightarrow] \arrow[r, Rightarrow] & \fbox{\text{ccs $\Delta$}} \arrow[ddr, Rightarrow]&  & & \\
& & & & \\
\fbox{\text{super Daugavet}}  \arrow[r, Rightarrow] \arrow[ddr, Rightarrow]& \fbox{\text{super $\Delta$}} \arrow[r, Rightarrow] & \fbox{\text{$\Delta$}} \\
& & & & \\
& \fbox{\text{Daugavet}} \arrow[uur, Rightarrow] & &  &
\end{tikzcd}
\caption{Relations between the diametral notions}
\label{figure01}
\end{figure}

We will show throughout the text that none of the above implications reverses, see Subsection~\ref{subsection:schemeofcounterexamples} for a description of all the relations and the counterexamples. However, let us point out right away that we do not know whether there exists ccs $\Delta$-points which are not super $\Delta$. In view of Remark~\ref{rem:bourgain} such examples may exist since Bourgain's lemma is not localizable. Also let us point out that it follows again from Bourgain's lemma that a ccs Daugavet point $x\in S_X$ also satisfies $\sup_{y\in D}\norm{x-y}=2$ for every ccw $D$ of $B_X$. Again this is not clear for ccs $\Delta$-points and we could thus naturally distinguish between ccs $\Delta$-points and ``ccw $\Delta$-points''. Since we do not have concrete examples at hand, we will focus on convex combination of slices and specifically point out any available ccw behavior throughout the text.

Let us also comment that it is clear that if every ccs of the unit ball of a given Banach space is weakly open (respectively, has non-empty relative weak interior),  then every super $\Delta$-point (respectively, every super Daugavet point) in this space is a ccs $\Delta$-point (respectively, a ccs Daugavet point). Several properties of this kind where introduced and studied in \cite{abhlp}, \cite{al18}, and \cite{lmr19}. We refer to those papers for some background and for examples.

\begin{rem}\label{rem:weak*_diametral_points}\parskip=0ex
There are natural weak$^*$ versions in dual spaces of all the notions of diametral-points introduced in the present section where slices and relatively weakly open subsets are respectively replaced with weak$^*$ slices (i.e.\ slices defined by elements of the predual) and relatively weak$^*$ open subsets. With obvious terminology, it then follows from \cite[Lemma~2.1]{KSSW} and from \cite[Lemma~3]{shv00} that a Banach space $X$ has the Daugavet property if and only if every element in $S_{X^*}$ is a weak$^*$ Daugavet point if and only if every element in $S_{X^*}$ is a weak$^*$ ccs Daugavet point. It also follows from \cite[Theorem~3.6]{ahntt16} that $X$ has the DLD2P if and only if every point in $S_{X^*}$ is a weak$^*$ $\Delta$-point. However, the relationship between the DD2P in $X$ and weak$^*$ super $\Delta$-points in $S_{X^*}$ is currently unknown.

Observe that a direct consequence of those results is that weak$^*$ diametral points and their weak counterparts might differ in a very strong way since, for instance, the unit ball of the space $C[0,1]^*$ admits denting points. Yet clearly all the results from the following sections concerning the different notions of diametral-points admit obvious analogues for their weak$^*$ counterparts. We leave the details to the reader to avoid unnecessary repetitions, but let us still point out that it follows from Goldstine's theorem and from the lower weak$^*$ semicontinuity of the norm in dual spaces that there is a natural correspondence between diametral-properties of points in $S_X$ and weak$^*$ properties of their image in the bidual under the canonical embedding $J_X$. Namely:
\begin{enumerate}
    \item $x\in S_X$ is a Daugavet point (respectively, a ccs Daugavet point) if and only if $J_X(x)$ is a weak$^*$ Daugavet point (respectively, a weak$^*$ ccs Daugavet point).
    \item $x\in S_X$ is a super Daugavet point if and only if $J_X(x)$ is a weak$^*$ super Daugavet point if and only if for every $y\in B_X$ there exists a net $(y_s^{**})$ in $B_{X^{**}}$ which converges to $J_X(y)$ in the weak$^*$ topology and such that $\norm{\pi_X(x)-y_s^{**}}\longrightarrow 2$ (see Section \ref{sec:characterizations_and_geometric_consequences}).
    \item $x\in S_X$ is a $\Delta$-point (respectively, a super $\Delta$-point) if and only if $J_X(x)$ is a $\Delta$-point (respectively, a super $\Delta$-point).
    \item $x\in S_X$ is a ccs $\Delta$-point if and only if $J_X(x)$ is a weak$^*$ ccs $\Delta$-point.
\end{enumerate}
Let us point out that $(3)$ essentially follows from the obvious fact that $\Delta$-points and super $\Delta$-points naturally pass to superspaces, that is if $Y$ is a subspace of $X$ and if $x\in S_Y$ if a $\Delta$-point (respectively, a super $\Delta$-point) in $Y$, then $x$ is a $\Delta$-point (respectively, a super $\Delta$-point) in $X$ . This property is unclear for ccs $\Delta$-points, so the assertion (4) is not analogous to assertion (3). 
\end{rem}

\section{Characterisations of diametral-notions and implications on the geometry of the ambient space}\label{sec:characterizations_and_geometric_consequences}

In view of the definitions of diametral-points, it is natural to expect that the presence of any kind of Daugavet- or $\Delta$-element in a given Banach space will affect, by the severe restrictions it inflicts on the nature of the considered point, its global isometric geometry or even its topological structure. However, previous studies in the context have shown that the situation is much more complicated than one could expect at first sight. For example, let us comment that a Banach space $X$ with the RNP and admitting a Daugavet point, and a Banach space with a one-unconditional basis and admitting a weakly dense subset of Daugavet points, were respectively constructed in \cite{veefree} and in \cite{almt21}. In this section, we provide useful characterizations of the new diametral notions, and investigate the immediate effect of the presence of such points on the geometry of the considered space.


We start by an intuitive but not completely trivial observation.

\begin{obs}\label{obs:existeninfidim}
By definition, it is clear that super $\Delta$-points do not exist in finite dimensional spaces because the weak and norm topology coincide in this context. Also, it was proved in \cite[Theorem 4.4]{almp22} that finite dimensional spaces do also fail to contain $\Delta$-points (hence ccs $\Delta$-points). In fact they fail to contain them in a stronger way, see \cite[Corollary 6.10]{almp22}. Consequently, the study of diametral-notions only makes sense in infinite dimension, and from now on we will assume unless otherwise stated that all the Banach spaces we consider are infinite dimensional.
\end{obs}

Let us next prove a bunch of characterisations for super Daugavet- and super $\Delta$-points.

Let $X$ be a Banach space. For every $x\in S_X$ and for every $\eps>0$, let us define $$\Delta_\eps(x):=\{y\in B_X\colon \norm{x-y}> 2-\eps\}.$$ We recall the following characterization of Daugavet- and $\Delta$-points from \cite{ahlp20}.

\begin{lem}[\mbox{\textrm{\cite[Lemma 2.1 and 2.2]{ahlp20}}}]\label{lem:Delta_set_characterizations}
Let $X$ be a Banach space.
\begin{enumerate}
\item An element $x\in S_X$ is a Daugavet point if and only if $B_X=\cconv\Delta_\eps(x)$ for every $\eps>0$.
\item An element $x\in S_X$ is a $\Delta$-point if and only if $x\in \cconv\Delta_\eps(x)$ for every $\eps>0$.
\end{enumerate}
\end{lem}

We have similar characterisations for super points.

\begin{lem}\label{lem:super_Delta_set_characterizations}
Let $X$ be a Banach space.
\begin{enumerate}
\item An element $x\in S_X$ is a super Daugavet point if and only if $B_X=\overline{\Delta_\eps(x)}^w$ for every $\eps>0$.
\item An element $x\in S_X$ is a super $\Delta$-point if and only if $x\in \overline{\Delta_\eps(x)}^w$ for every $\eps>0$.
\end{enumerate}
\end{lem}

\begin{proof}
Observe that for given $x\in S_X$, $y\in B_X$, and $\eps>0$, we have that $y$ belongs to the weak closure of the set $\Delta_\eps(x)$ if and only if $\Delta_\eps(x)$ has non-empty intersection with any neighborhood of $y$ in the relative weak topology of $B_X$. Thus $y$ belongs  to $\overline{\Delta_\eps(x)}^w$ for every $\eps>0$ if and only if $\sup_{z\in V} \norm{x-z}=2$ for every relatively weakly open subset $V$ of $B_X$ containing $y$. The conclusion easily follows.
\end{proof}

For any given $x\in B_X$, we denote by $\V(x)$ the set of all neighborhoods of $x$ for the relative weak topology of $B_X$. We can provide characterizations of super points using nets which is just a localization of \cite[Proposition 2.5]{blr18}.

\begin{prop}\label{prop:net_characterizations}
Let $X$ be an infinite-dimensional Banach space.
\begin{enumerate}
\item An element $x\in S_X$ is a super Daugavet point if and only if for every $y\in B_X$ there exists a net $(y_s)$ in $B_X$ which converges weakly to $y$ and such that $\norm{x-y_s}\longrightarrow 2$.
\item An element $x\in S_X$ is a super $\Delta$-point if and only if there exists a net $(x_s)$ in $B_X$ which converges weakly to $x$ and such that $\norm{x-x_s}\longrightarrow 2$.
\end{enumerate}
In both cases we can moreover force the nets to be in $S_X$.
\end{prop}

\begin{proof}
Let us fix $x\in S_X$. Given any $y\in B_X$, it is clear that if there exists a net $(y_s)$ in $B_X$ which converges weakly to $y$ and such that $\norm{x-y_s}\longrightarrow 2$, then $y$ belongs to the weak closure of $\Delta_\eps(x)$ for every $\eps>0$.  Conversely let us pick $y\in B_X$ satisfying this property. We turn $S:=\V(y)\times (0,\infty)$ into a directed set by $(V,\eps)\leq (V',\eps')$ if and only if $V'\subset V$ and $\eps'\leq \eps$. By the assumptions we have that $V\cap \Delta_\eps(x)$ is a non-empty subset of $B_X$ for every couple $s:=(V,\eps)$ in $S$. Picking any $y_s$ in this set will then provide the desired net.

Finally observe that for $x\in B_X$ and $\eps>0$, we have that $B_X\backslash \Delta_\eps(x)=\{y\in B_X\colon \norm{x-y}\leq 2-\eps\}$ is weakly closed by the lower semi-continuity of the norm, so that $\Delta_\eps(x)$ is a relatively weakly open subset of $B_X$. Thus we have that $V\cap \Delta_\eps(x)$ is a non-empty relatively weakly open subset of $B_X$ for every couple $s:=(V,\eps)$ in $S$. Since $X$ is infinite dimensional, this set has to intersect $S_X$, and we can actually pick $y_s$ in  $V\cap \Delta_\eps(x)\cap S_X$. \end{proof}

\begin{rem} In \cite{kw04} an example of a Banach space satisfying simultaneously the Daugavet property and the Schur property was provided. Such example shows that there is no hope to get a version of the above result involving sequences.
\end{rem}

Observe that the following result, similar to \cite[Lemma 2.1 and 2.2]{jr22}, is included in the preceding proof.

\begin{prop}
Let $X$ be a Banach space and let $x \in S_X$.
\begin{enumerate}
\item If $x$ is a super Daugavet point, then for every $\eps>0$ and every non-empty relatively weakly open subset $V$ of $B_X$  we can find a non-empty relatively weakly open subset $U$ of $B_X$ which is contained in $V$ and such that $\norm{x-y}>2-\eps$ for every $y\in U$.
\item If $x$ is a super $\Delta$-point, then for every $\eps>0$ and every non-empty relatively weakly open subset $V$ of $B_X$ containing $x$  we can find a non-empty relatively weakly open subset $U$ of $B_X$ which is contained in $V$ and such that $\norm{x-y}>2-\eps$ for every $y\in U$.
\end{enumerate}
\end{prop}

\begin{proof}
Fix any  $x\in S_X$ and any $y\in B_X$ which belongs to the weak closure of $\Delta_\eps(x)$ for every $\eps>0$. Then, for every $V\in \V(y)$ and every $\eps>0$, we have that $U:=V\cap \Delta_\eps(x)$ is a non-empty relatively weakly open subset of $B_X$. \end{proof}

It is clear from the definition that denting points of $B_X$ cannot be $\Delta$-points. Also it was first observed in \cite[Proposition 3.1]{jr22} that every Daugavet point in a Banach space $X$ has to be at distance $2$ from every denting point of the unit ball of $X$.  This elementary observation turned out to play an important role in the study of Daugavet points in Lipschitz-free spaces in \cite{jr22} and \cite{veefree}. We have similar observations for super points.

\begin{lem}
Let $X$ be a Banach space and let $x\in S_X$. If $x$ is a super $\Delta$-point, then $x$ cannot be a point of continuity. If, moreover, $x$ is a super Daugavet point, then $x$ has to be at distance $2$ from every point of continuity of $B_X$.
\end{lem}

\begin{proof}
If and element $y$ of $B_X$ is a point of continuity, then it is contained in relatively weakly open subsets of $B_X$ of arbitrarily small diameter. Clearly no super $\Delta$-point can have this property, and any super Daugavet point has to be at distance $2$ from any such points.
\end{proof}

This lemma provides quite a few examples of Banach spaces which fail to contain super points. Following \cite{ht14} let us recall that $X$ has the \emph{Kadets property} if the norm topology and the weak topology coincide on $S_X$, and that $X$ has the \emph{Kadets-Klee} property if weakly convergent sequences in $S_X$ are norm convergent. Let us also recall that any LUR space has the Kadets-Klee property, and that any space with the Kadets-Klee property which fails to contain $\ell_1$ has the Kadets property.  By proposition \ref{prop:net_characterizations} we clearly have the following result.

\begin{prop}\label{prop:kadetsproperty}
If $X$ has the Kadets property, then  $X$ fails to contain super $\Delta$-points.
\end{prop}

As a corollary we obtain the following. Recall that a Banach space is \emph{asymptotic uniformly convex} (\emph{AUC} in short) \cite{jlps} if its modulus of asymptotic uniform convexity $$\overline{\delta}_X(t):=\inf_{x\in S_X}\sup_{\dim X/Y <\infty}\inf_{y\in S_Y} \norm{x+ty}-1$$ is strictly positive for every $t>0$.

\begin{cor}
Let $X$ be AUC. Then, $X$ fails to contain super $\Delta$-points.
\end{cor}

\begin{proof}
In an AUC space, every element of the unit sphere is a point of continuity of $B_X$, see \cite[Proposition~2.6]{jlps}.
\end{proof}

\begin{rem}
It was proved in \cite[Theorem 3.4]{almp22} that any \emph{reflexive} AUC space fails to contain $\Delta$-points. Also, combining the observations from \cite[End of Section 4]{almp22} about weak$^*$ quasi-denting points in the unit ball of AUC$^*$ duals and \cite[Corollary 2.4]{veelipfunc} about the maximality of the Kuratowski index of weak$^*$ slices containing weak$^*$ $\Delta$-points, we have that every AUC$^*$ dual space fails to contain weak$^*$ $\Delta$-points. However, note that it is currently unknown whether non-reflexive AUC spaces (and, in particular, whether the dual of the James tree spaces JT$^*$) may contain Daugavet- or $\Delta$-points.
\end{rem}

It turns out that Daugavet points are characterized by this distance to denting points in RNP spaces (because the unit ball of an RNP space $X$ can be written as the closed convex hull of the set of its denting points) as well as in Lipschitz-free spaces (\cite[Theorem~3.2]{jr22} for compact metric spaces and \cite[Theorem~2.1]{veefree} for a general statement). In the same way we can characterize super Daugavet points in terms of this distance to points of continuity of $B_X$ is spaces with the CPCP.

\begin{prop}
If a Banach space $X$ has the CPCP, then a point $x\in S_X$ is a super Daugavet point if and only if it is at distance $2$ from any point of continuity of $B_X$.
\end{prop}

\begin{proof} If $X$ has the CPCP, then the set of all points of continuity of $B_X$ is weakly dense in $B_X$ (see for example \cite[Proposition 3.9]{ew}), that is, every non-empty relatively weakly open subset of $B_X$ contains a point of continuity. The conclusion follows easily.
\end{proof}

For ccs points, the situation is quite different. Indeed, although ccs $\Delta$-points can clearly not be points of strong regularity, we have by
\cite[Theorem 3.1]{lmr19} that $X$ has the SD2P if and only if every convex combination of slices of $B_X$ contains elements of norm arbitrarily close to $1$. It readily follows that any space $X$ which contains a ccs Daugavet point satisfies the SD2P, so it is very far from being strongly regular. We will provide more details on this topic in Section \ref{sec:inside_of_the_unit_ball}, but for later reference let us state the following.

\begin{prop}\label{prop:ccs_Daugavet_SR}
Let $X$ be a Banach space. If $X$ contains a ccs Daugavet point, then it has the SD2P (it fails to be strongly regular).
\end{prop}

Next, we show that extreme points have a nice behaviour with respect to diametral notions.

\begin{proposition}\label{prop:extreme_diametral_delta}
Let $X$ be a Banach space and let $x\in S_X$.
\begin{enumerate}
    \item If $x\in \preext{B_X}$ and it is a $\Delta$-point, then $x$ is a super $\Delta$-point.
    \item If $x\in \ext{B_X}$ and it is a super $\Delta$-point, then $x$ is a ccs $\Delta$-point.
    \item In particular, if $x\in \preext{B_X}$ is a $\Delta$-point, then $x$ is a super $\Delta$-point as well as a ccs $\Delta$-point.
\end{enumerate}
\end{proposition}

\begin{proof}
It follows from Choquet's lemma (see for example \cite[Lemma 3.69]{checos}) that slices form neighborhood bases in the relative weak topology of the unit ball of a Banach space for its preserved extreme points, so (1) immediately follows. For (2), if $x$ is extreme and belongs to a ccs $C:=\sum_{i=1}^n \lambda_iS_i$ of $B_X$ then $x\in \bigcap_{i=1}^n S_i$, which is a relatively weakly open subset of $B_X$.
\end{proof}

\begin{rem}
Observe that, in fact, any extreme super $\Delta$-point is  ``ccw $\Delta$-point'' as we discussed in Section \ref{sec:notation_and_prelimiary_results}. Also, Choquet's lemma implies that every extreme weak$^*$ $\Delta$-point in a dual space is weak$^*$ ccw $\Delta$-point.
\end{rem}

\subsection{Spaces with a one-unconditional basis and beyond}\label{subsec:one-unconditional_bases}

In \cite{almt21}, it was proved that no real Banach space with a subsymmetric basis contains a $\Delta$-point. On the other hand, an example of a Banach space with a one-unconditional basis that contains a $\Delta$-point was provided, and a more involved example of a Banach space with a one-unconditional basis that contains many Daugavet-points was constructed. We will discuss this second example in detail in Subsection~\ref{subsubsec:oneunconditionalbasis}.

In the process, it was also implicitly shown that real Banach spaces with a one-unconditional basis cannot contain super $\Delta$-points. In the present subsection, we prove that the same goes for ccs $\Delta$-points. Also, we provide sharper and more general versions of \cite[Proposition~2.12]{almt21}. In the first part of this section, we follow \cite{almt21} and restrict ourselves to real Banach spaces. 

Let $X$ be a real Banach space with a Schauder basis $(e_i)_{i\geq 1}$. We denote by $(e_i^*)_{i\geq 1}$ the corresponding sequence of biorthogonal functionals. Recall that $(e_i)_{i\geq 1}$ is said to be \emph{unconditional} if the series $\sum_{i\geq 1}e_i^*(x)e_i$ converges unconditionally for every $x\in X$.  Also, recall that an unconditional basis $(e_i)_{i\geq 1}$ is said to be \emph{one-unconditional} if $$\norm{\sum_{i\geq 1}\theta_ie_i^*(x)e_i}=\norm{\sum_{i\geq 1}e_i^*(x)e_i}$$ for every $(\theta_i)_{i\geq 1}\in \{-1,1\}^{\N}$ and for every $x\in X$. Moreover, if $$\norm{\sum_{i\geq 1}\theta_ie_i^*(x)e_{n_i}}=\norm{\sum_{i\geq 1}e_i^*(x)e_i}$$ for every $(\theta_i)_{i\geq 1}\in \{-1,1\}^{\N}$, for every $x\in X$, and for every strictly increasing sequence $(n_i)_{i\geq 1}$ in $\N$, then the basis is called \emph{subsymmetric}.

Observe that for spaces with a one-unconditional basis, it is enough, in order to study the various Daugavet- and $\Delta$-notions, to work in the \emph{positive sphere} $$S_X^+:=\{x\in S_X\colon e_i^*(x)\geq 0\ \forall i\}$$ of the space $X$. Also, the following result is well known. 

\begin{lem}\label{lem:fundamental_lemma_one-unconditional-bases}

Let $X$ be a real Banach space with a one-unconditional basis $(e_i)_{i\geq 1}$, and let $(a_i)_{i\geq 1}$ and $(b_i)_{i\geq 1}$ be sequences of real numbers.  If the series $\sum_{i\geq 1}b_ie_i$ converges, and if $\abs{a_i}\leq \abs{b_i}$ for every $i$, then $\sum_{i\geq 1}a_ie_i$ converges as well, and we have $$\norm{\sum_{i\geq 1}a_ie_i}\leq \norm{\sum_{i\geq 1}b_ie_i}.$$
    
\end{lem}

Let us now recall a few notation and preliminary results from \cite{almt21}. Let $X$ be a real Banach space with a normalized one-unconditional basis $(e_i)_{i\geq 1}$. For every subset $A$ of $\N$, we denote by $P_A$ the projection on $\cspan\{e_i,\ i\in A\}$. Then for every $x\in X$, we define $$M(x):=\{A\subset \N\colon \norm{P_A(x)}=\norm{x},\ \text{and } \norm{P_A(x)-e_j^*(x)e_j}<\norm{x}\ \forall j\in A\}. $$ The set $M(x)$ can be seen as the set of all \emph{minimal norm-giving subsets} of the support of $x$. We denote respectively by $M^{\F}(x)$ and $M^{\infty}(x)$ the subsets of all finite and infinite elements of $M(x)$. It follows from \cite[Lemma~2.7]{almt21} that the set $M(x)$ is never empty, and from \cite[Proposition~2.15]{almt21} that no element $x\in S_X$ satisfying $M^{\infty}(x)=\emptyset$ can be a $\Delta$-point.

For every non-empty ordered subset $A:=\{a_1<a_2<\dots \}$ of $\N$, and for every $n\in \N$ smaller than or equal to $\abs{A}$, we denote by $A(n):=\{a_1,\dots, a_n\}$ the subset consisting of the $n$ first elements of $A$. We will implicitly assume in the following that the elements of $M(x)$ are ordered subsets of $\N$. The next two results were proved in \cite[Lemma~2.8 and Lemma~2.11]{almt21}.

\begin{lem}\label{lem:first_elements_of_norm_giving_subsets}

Let $X$ be a real Banach space with a normalized one-unconditional basis $(e_i)_{i\geq 1}$ and let $x\in S_X$. For every $n\in\N$, the sets 
$$
\bigl\{A\in M(x)\colon \abs{A}\leq n\bigr\}\quad \text{ and } \quad \bigl\{A(n)\colon A\in M(x),\ \text{and } \abs{A}>n\bigr\}
$$
are both finite.
    
\end{lem}

\begin{lem}\label{lem:essential_subset_for_the_norm}

Let $X$ be a real Banach space with a normalized one-unconditional basis and let $x\in S_X$. For every subset $E$ of $\N$ such that $E\cap A\neq \emptyset$ for every $A\in M(x)$, we have $\norm{x-P_E(x)}<1$.
    
\end{lem}

With those tools at hand, we can now prove an analogue to \cite[Proposition~2.13]{almt21} for convex combination of slices.

\begin{prop}\label{prop:one-unconditional-basis_no_ccs-delta}

Let $X$ be a Banach space with a normalized one-unconditional basis and $x\in S_X^+$. Then, there exists $\delta>0$ and a ccs $C$ of $B_X$ containing $x$ such that $\sup_{y\in C}\norm{x-y}\leq 2-\delta$.    
\end{prop}

\begin{proof}

Let $x\in S_X^+$, and define $E=\bigcup_{A\in M(x)}A(1)$. From Lemma~\ref{lem:first_elements_of_norm_giving_subsets} and Lemma~\ref{lem:essential_subset_for_the_norm}, we have that $E$ is a finite subset of $\N$ and that $\norm{x-P_E(x)}<1$. In particular, there exists $\gamma>0$ such that $\norm{x-P_E(x)}\leq 1-\gamma$. For every $i\in E$, we define $$S_i:=S\left(e_i^*,1-\frac{e_i^*(x)}{2}\right).$$ Then we consider the ccs $$C:=\frac{1}{\abs{E}}\sum_{i\in E}S_i.$$ Since $x\in S_X^+$, we clearly have that $x\in \bigcap_{i\in E}S_i$ and, in particular, that $x\in C$. 

So let us pick $y:=\frac{1}{\abs{E}}\sum_{i\in E}y^i$ in $C$. Then we have $e_i^*(y^i)>\frac{e_i^*(x)}{2}$ for every $i$. In particular, $e_i^*(y^i)\geq 0$, and $\abs{e_i^*(y^i)-e_i^*(x)}\leq e_i^*(y^i)$. Indeed, for any given non-negative real numbers $\alpha$ and $\beta$ with $\beta\geq \frac{\alpha}{2}$, we have $$\abs{\beta-\alpha}=\beta-\alpha\leq \beta$$ if $\beta\geq \alpha$, and $$\abs{\beta-\alpha}=\alpha-\beta\leq \alpha-\frac{\alpha}{2}=\frac{\alpha}{2}\leq \beta$$ if $\beta\leq \alpha$. So in either case, $\abs{\beta-\alpha}\leq \beta$ as desired. 

It then follows from Lemma~\ref{lem:fundamental_lemma_one-unconditional-bases} that $\norm{y^i-e_i^*(x)e_i}\leq \norm{y^i}\leq 1$ and, finally, \begin{align*}
    \norm{x-y}&\leq \norm{x-\frac{x}{\abs{E}}}+\norm{\frac{x}{\abs{E}}-\frac{P_E(x)}{\abs{E}}}+\norm{\frac{P_E(x)}{\abs{E}}-y} \\
    &\leq 1-\frac{1}{\abs{E}}+\frac{1-\gamma}{\abs{E}}+\frac{1}{\abs{E}}\sum_{i\in E}\norm{e_i^*(x)e_i-y^i} \leq 2-\frac{\gamma}{\abs{E}}.
    \end{align*} The conclusion follows with $\delta:=\frac{\gamma}{\abs{E}}$. In particular, note that since $x$ belongs to the relative weakly open set $\bigcap_{i\in E}S_i \subset C$, we also get that $x$ is not super $\Delta$, recovering the result from \cite{almt21}. 
\end{proof}

So combining \cite[Proposition~2.13]{almt21} and Proposition~\ref{prop:one-unconditional-basis_no_ccs-delta}, we immediately get that spaces with a normalized one-unconditional basis fail to contain super $\Delta$-points and ccs $\Delta$-points. So let us state the following here for future reference.

\begin{thm}\label{thm:one-unconditional_no_ccs_nor_super_Delta}

Let $X$ be a real Banach space with a normalized one-unconditional basis. Then $X$ does not contain super $\Delta$-points, and $X$ does not contain ccs $\Delta$-points. 

\end{thm}

In the rest of the subsection, we aim at providing sharper and improved versions of  \cite[Proposition~2.13]{almt21}. In particular we will go back to working with either real or complex Banach spaces. The main result of this study is the following proposition.

\begin{prop}\label{prop:operators_and_super_Delta_points}
Let $X$ be a Banach space, and let us assume that there exists a subset $\mathcal{A}\subseteq \mathcal{F}(X,X)$ satisfying that $\sup\bigl\{\norm{\Id-T}\colon T\in \mathcal{A}\bigr\}<2$ and that for every $\eps>0$ and every $x\in X$, there exists $T\in\mathcal{A}$ such that $\norm{x-T x}<\eps$. Then, $X$ contains no super $\Delta$-point.
\end{prop}

Let us provide a lemma which is a localization of the above result from which its proof is immediate.

\begin{lemma}\label{lemma:operators_and_super_Delta_points_localization}
Let $X$ be a Banach space, and let $x\in S_X$. If there exists a finite-rank operator $T$ on $X$ such that $\norm{x-Tx}+\norm{\Id-T}<2$, then $x$ is not a super $\Delta$-point.
\end{lemma}

\begin{proof}
Consider $\eps>0$ such that $K:=\norm{x-Tx}+\norm{\Id-T}+\eps<2$. Since $T$ has finite rank, we can find $N\geq 1$, $w_1,\dots, w_N\in S_X$ and $f_1,\dots,f_N\in X^*$ such that $T(z)=\sum_{n=1}^Nf_n(z)w_n$ for every $z\in X$. Let us consider
$$
W:=\left\{y\in B_X\colon \abs{f_n(x-y)}<\frac{\eps}{2^{n+1}}\ \forall n\in\{1,\dots, N\}\right\}.
$$
$W$ is a neighborhood of $x$ in the relative weak topology of $B_X$, and for every $y\in W$, we have
\begin{align*}
    \norm{x-y}&\leq \norm{x-T x}+\norm{T x- T y}+\norm{y-T y} \\
    &\leq \norm{x-Tx} + \norm{\Id-T} +\sum_{n=1}^N\abs{f_n(x-y)}\norm{w_n} \\
    &\leq \norm{x-Tx} + \norm{\Id-T} + \eps\sum_{n=1}^N\frac{1}{2^{n+1}} \leq K<2.\qedhere
\end{align*}
\end{proof}

N.B. It is unclear whether an analogue to Lemma~\ref{lemma:operators_and_super_Delta_points_localization} can be given for ccs $\Delta$-points. So we do not know whether Proposition~\ref{prop:operators_and_super_Delta_points} extends to this notion.

As particular cases of Proposition~\ref{prop:operators_and_super_Delta_points}, we have the following ones. Recall that a sequence $(E_n)_{n\geq 1}$ of finite dimensional subspaces of a given Banach space $X$ is called a \emph{finite dimensional decomposition} (\emph{FDD}) for $X$ if every element $x\in X$ can be represented in a unique way as a series $x:=\sum_{n\geq 1}x_n$ with $x_n\in E_n$ for every $n\geq 1$. Such an FDD is said to be $\emph{unconditional}$ if the above series converges unconditionally for every $x\in X$. In this case, it is well known that the family $(P_A)_{A\subset \N}$, where $P_A$ is the projection given by $P_A(x):=\sum_{n\in A}x_n$, is uniformly bounded, and the constant $K_S:=\sup_{A\subset \N}\norm{P_A}$ is called the \emph{suppression-unconditional constant} of the FDD. We refer to \cite[Section~1.g]{LT} for the details and to \cite[Section~3.1]{alka} for the particular case of unconditional bases.

\begin{cor}\label{coro:oneunconditionalandmore}
A Banach space $X$ fails to have super $\Delta$-points provided one of the following conditions is satisfied.
\begin{enumerate}
\item There exists a family $\mathcal{A}\subseteq \mathcal{F}(X,X)$ satisfying that $\sup\bigl\{\norm{\Id-T}\colon T\in \mathcal{A}\bigr\}<2$ and that the identity mapping belongs to its strong operator topology (SOT) closure.
\item There exists a family $\{P_\lambda\}_{\lambda_\in\Lambda}$ of finite rank \emph{projections} on $X$ such that $X=\overline{\bigcup_{\lambda\in\Lambda}P_\lambda(X)}$, and such that $\sup_{\lambda\in \Lambda} \norm{\Id-P_\lambda}<2$.
\item The space $X$ admits a FDD with suppression-unconditional constant less than $2$. In particular, if $X$ admits an unconditional basis with suppression-unconditional constant less than $2$.
\end{enumerate}
\end{cor}

Let us observe that the value $2$ in the above results is sharp in several ways.

\begin{rem}
\begin{enumerate}
      \item The space $C[0,1]$ admits a monotone Schauder basis, so there exists a sequence $\{P_n\}_{n\geq 1}$ of norm one finite rank projections on this space which converges to $\Id$ in SOT topology. As $C[0,1]$ has the Daugavet property, all elements in $S_X$ are super Daugavet points. Observe that $\norm{\Id-P_n}=2$ for every $n\geq 1$ by the DPr.
      \item Let $X$ be an arbitrary Banach space. For every $x\in S_X$ choose $f_x\in S_{X^*}$ such that $f_x(x)=1$, and define $P_x(z)=f_x(z)x$ for every $z\in X$. Then  $\{P_x\colon x\in S_X\}$ is a family of norm one rank-one projections on $X$,  $X=\bigcup_{x\in S_X}P_x(X)$, and $\norm{\Id-P_x}\leq 2$ for every $x\in S_X$.
      \item  The space $c$ admits ccs Daugavet points (hence super Daugavet points), see Theorem~\ref{thm:C(K)-deltaimpliesccsDaugavet}, but it is easy to check that its usual basis is $3$-unconditional and $2$-suppression unconditional. 
      \item It is shown in \cite{ik} that a Banach space has the DLD2P if and only if $\|\Id-P\|\geq 2$ for every rank-one projection $P$. It follows that the suppression constant of an unconditional basis on a Banach space with the DLD2P has to be greater than or equal to $2$. Let us mention here that there is no local version of this result, as there are Banach spaces with one-unconditional basis and containing many Daugavet points \cite{almt21} (see Paragraph~\ref{subsubsec:oneunconditionalbasis}).
  \end{enumerate}
\end{rem}

\subsection{Absolute sums}\label{subsec:absolutesums}

In this subsection we look at the transfer of the diametral points through absolute sums of Banach spaces. Let us first recall the following definition.

\begin{defn}
A norm $N$ on $\R^2$ is \emph{absolute} if $N(a,b)=N(\abs{a},\abs{b})$ for every $(a,b)\in \R^2$ and \emph{normalized} if $N(0,1)=N(1,0)=1$.
\end{defn}

If $X$ and $Y$ are Banach spaces, and if $N$ is an absolute normalized norm on $\R^2$, we denote by $X\oplus_NY$ the product space $X\times Y$ endowed with the norm $\norm{(x,y)}=N(\norm{x},\norm{y})$.  It is easy to check that $X\oplus_NY$ is a Banach space, and that its dual can be expressed as $(X\oplus_N Y)^*\equiv X^*\oplus_{N^*}Y^*$ where $N^*$ is the absolute norm given by the formula $N^*(c,d)=\max_{N(a,b)=1}\abs{ac}+\abs{bd}$.  Classical examples of absolute normalized norms on $\R^2$ are the $\ell_p$ norms for $p\in[1,\infty]$.  Information on absolute norms can be found in \cite[\S 21]{BD} and \cite{MenaPayaRodr1989} and references therein, for instance. Let us recall that for every absolute normalized sum $N$, given non-negative $a,b,c,d$ in $\R$ with $a\leq b$ and $c\leq d$ we have $N(a,b)\leq N(c,d)$. In particular, $\norm{\cdot}_\infty\leq N \leq \norm{\cdot}_1$.

Similar to the DD2P (see \cite[Theorem 2.11]{blr18}) and to $\Delta$-points \cite{hpv21}, super $\Delta$-points transfer very well through absolute sums.

\begin{prop}\label{prop:super_Delta_abs_sums}
Let $X$ and $Y$ be Banach spaces, and let $N$ be an absolute normalized norm.
\begin{enumerate}
  \item If $x\in S_X$ and $y\in S_Y$ are super $\Delta$-points, then $(ax,by)$ is a super $\Delta$-point in $X\oplus_NY$ for every $(a,b)\in \R^2$ with $N(a,b)=1$.
  \item If $x\in S_X$ is a super $\Delta$-point, then $(x,0)$ is a super $\Delta$-point in $X\oplus_NY$. If $y\in S_Y$ is a super $\Delta$-point, then $(0,y)$ is a super $\Delta$-point in $X\oplus_NY$.
\end{enumerate}
\end{prop}

\begin{proof}
(1). We can find two nets $(x_s)_{s\in S}$  and $(y_t)_{t\in T}$ respectively in $S_X$ and $S_Y$ such that $x_s\overset{w}{\longrightarrow} x$, $y_t\overset{w}{\longrightarrow} y$, and $\norm{x-x_s},\norm{y-y_t}\longrightarrow 2$. Now, if we take $(a,b)\in \R^2$ with $N(a,b)=1$ we clearly have $(ax_s,by_t)\overset{w}{\underset{(s,t)\in S\times T}{\longrightarrow}} (ax,by)$ and $\norm{(ax,by)-(ax_s,by_t)}=N\left(a\norm{x-x_s},b\norm{y-y_t}\right)\longrightarrow 2N(a,b)=2$, so $(ax,by)$ is a super $\Delta$-point in $X\oplus_N Y$. For (2), we just repeat the previous proof with $a=1$ and $b=0$ or with $a=0$ and $b=1$ and so we only need one of the points to be super $\Delta$-point.
\end{proof}

For super Daugavet points the situation is more complicated and we need to distinguish between different kinds of absolute norms. The following definitions can be found, for instance, in \cite{hpv21}. 

\begin{defn}
Let $N$ be an absolute normalized norm on $\R^2$.
\begin{enumerate}
\item $N$ has \emph{property $(\alpha)$} if for every $a,b\in \R_+$ with $N(a,b)=1$ we can find a neighborhood $W$ of $(a,b)$ in $\R^2$ with $\sup_{(c,d)\in W }c<1$ or $\sup_{(c,d)\in W} d<1$ and such that any couple $(c,d)\in \R_+^2$ satisfying $N(c,d)=1$ and $N\left((a,b)+(c,d)\right)=2$ belongs to $W$.
\item $N$ is \emph{$A$-octahedral} if there exists $a,b\in \R_+$ such that $N(a,b)=1$ and $N\left((a,b)+(c,d)\right)=2$ for $$c=\max\{e\in \R_+\colon N(e,1)=1\}\ \ \text{ and }\ \ d=\max\{f\in\R_+:\ N(1,f)=1\}.$$
\item $N$ is \emph{positively octahedral} if there exists $a,b\in\R_+$ such that $N(a,b)=1$ and $$N\left((a,b)+(0,1)\right)=N\left((a,b)+(1,0)\right)=2.$$
\end{enumerate}
\end{defn}

Positively octahedral norms where introduced in \cite{hln18} in order to characterize the absolute norms for which the corresponding absolute sum is \emph{octahedral}. It is clear that property $(\alpha)$ and $A$-octhaedrality exclude each other and that every positively octahedral absolute normalized norm is $A$-octahedral (while there clearly exists absolute $A$-octahedral norms which are not positively octahedral). Moreover it was proved in \cite[Proposition 2.5]{hpv21} that every absolute normalized norm on $\R^2$ must either satisfy property $(\alpha)$ or be $A$-octahedral. For $\ell_p$-norms, we have that $\norm{\cdot}_1$ and $\norm{\cdot}_{\infty}$ are both positively octahedral, and that $\norm{\cdot}_p$ satisfies property $(\alpha)$ for every $p\in(1,\infty)$.

Observe that if an absolute normalized norm $N$ on $\R^2$ is positively octahedral, and if $(a,b)$ is as in the above definition, then the intersection of the unit sphere of $N$ with the positive quadrant of $\R^2$ is equal to the union of the segments $[(1,0),(a,b)]$ and $[(0,1),(a,b)]$ (see \cite[Section 3.3.1]{pirkthesis} for pictures).  In particular, it follows that $N((a,b)+(c,d))=2$ for every non-negative $c,d$ with $N(c,d)=1$. Similar to the results from \cite[Section 4]{ahlp20} concerning Daugavet points, we have the following.

\begin{prop}\label{prop:absolutesums-Daugavetpoints}
Let $X$ and $Y$ be Banach spaces, and let $N$ be an absolute normalized norm.
\begin{enumerate}
\item \cite[Proposition 4.6]{ahlp20} If $N$ has property $(\alpha)$, then $X\oplus_N Y$ has no Daugavet point (hence, in particular, no super Daugavet points).
\item If $N$ is positively octahedral and if $x\in S_X$ and $y\in S_Y$ are super Daugavet points, then $(ax,by)$ is a super Daugavet point in $X\oplus_N Y$ for every $(a,b)\in\R_+^2$ as in the above definition.
\end{enumerate}
\end{prop}

\begin{proof}
(2). Assume that $N$ is positively octahedral, take $(a,b)\in \R_+^2$ as in the  definition, and let $x\in S_X$ and $y\in S_Y$ be super Daugavet points. For any given $(u,v)\in X\oplus_N Y$ of norm $\norm{(u,v)}=1$ we can find two nets $(u_s)_{s\in S}$ and  $(v_t)_{t\in T}$ respectively in $S_X$ and $S_Y$ such that $\|u\|u_s\overset{w}{\longrightarrow} u$, $\|v\|v_t\overset{w}{\longrightarrow} v$, and $\norm{x-u_s},\norm{y-v_t}\longrightarrow 2$. Then $(\norm{u}u_s,\norm{v}v_t)\overset{w}{\underset{(s,t)\in S\times T}{\longrightarrow}} (u,v)$. Since
\begin{align*}
\norm{ax-\norm{u}u_s} &= \norm{(x-u_s)-[(1-a)x-(1-\norm{u})u_s]} \\
&\geq \norm{x-u_s}-(1-a+1-\norm{u}), \\
&= a+\norm{u}-(2-\norm{x-u_s}),
\end{align*}
and, in the same way, $$\norm{by-\norm{v}v_t}\geq b+\norm{v}-(2-\norm{y-v_t}),$$
 we have \begin{align*}
\norm{\left(ax-\norm{u}u_s,by-\norm{v}v_t\right)} &= N\left(\norm{ax-\norm{u}u_s},\norm{by-\norm{v}v_t}\right) \\
&\geq N\left(a+\norm{u}-(2-\norm{x-u_s}),b+\norm{v}-(2-\norm{y-v_t})\right) \\
&\longrightarrow N\left((a+\norm{u},b+\norm{v}\right)=2.
\end{align*} This shows that $(ax,by)$ is a super Daugavet point in $X\oplus_N Y$.
\end{proof}

\begin{rem}\label{rem:l_1_sums}
Note that if $(a,b)=(1,0)$ (respectively, $(a,b)=(0,1)$) in the previous statement (for example, when $N=\norm{\cdot}_1$), then we only need to assume that $x$ (respectively, $y$) is super Daugavet in order to get that $(x,0)$ (respectively, $(0,y)$) is super Daugavet in $X\oplus_N Y$. Also, if $N=\norm{\cdot}_\infty$, then we only need to assume that $x$ (respectively, $y$) is super Daugavet in order to obtain that $(x,\beta y)$ (respectively, $(\alpha x,y)$) is super Daugavet in $X\oplus_N Y$ for every $\beta \in[0,1]$ (respectively, $\alpha\in[0,1]$).
\end{rem}

In \cite[Theorem 2.2]{hpv21} it is proved that regular Daugavet points do also transfer through A-octahedral sums. We do not know if a similar result can be obtained for super Daugavet points. Indeed, observe that if $N$ is an A-octahedral norm, and if $c$, $d$, and $(a,b)$ are as in the above definition, then the intersection of the unit sphere of $N$ with the positive quadrant of $\R^2$ is equal to the union of the segments $[(1,0),(1,d)]$, $[(1,d),(a,b)]$, $[(0,1),(c,1)]$ and $[(c,1),(a,b)]$. In particular, $N((a,b)+(e,f))=2$ for every couple $(e,f)$ on the segments $[(1,d),(a,b)]$ and $[(c,1),(a,b)]$, but this is no longer true on the segments $[(1,0),(1,d)]$ and $[(0,1),(c,1)]$ and the argument in the above proof does not work anymore.

The situation for ccs $\Delta$-points and ccs Daugavet point is not clear and the proofs of the above results do not seem to admit easy extensions. For instance, it follows from the next result that Remark~\ref{rem:l_1_sums} is not valid for ccs Daugavet points.

\begin{proposition}\label{prop:oplus1stronglyexposedpoint}
Let $X$ be an arbitrary Banach space, let $Y$ be a Banach space containing an strongly exposed point $y_0\in S_Y$, and let $E:=X\oplus_1 Y$. Then, there are convex combinations of slices of $B_E$ around $0$ of arbitrarily small diameter. In particular, $E$ fails to contain ccs Daugavet points and also fails to have the SD2P.
\end{proposition}

\begin{proof}
Let $y_0^*\in S_{Y^*}$ strongly exposes $y_0$. Given $\eps>0$, there is $0<\delta<\eps$ such that $\|y-y_0\|<\eps$ whenever $y\in B_Y$ satisfies $\re y_0^*(y)>1-\delta$. Consider $f=(0,y_0^*)\in S_{E^*}$ and write
$$
C:=\frac{1}{2}\left(S(f,\delta;B_E)+S(-f,\delta,B_E)\right)
$$
Take $u:=\frac{1}{2}(u_1+u_2)\in C$ with $u_1\in S(f,\delta;B_E)$ and $u_2\in S(-f,\delta;B_E)$. So if write $u_1:=(x_1,y_1)$ and $u_2:=(x_2,y_2)$, we have
$$
\re y_0^*(y_1)=\re f(x_1,y_1)>1-\delta \  \text{ and } \ \re y_0^*(y_2)=\re f(x_2,y_2)<-1+\delta.
$$
On the one hand, it follows that $\|y_1-y_0\|<\eps$ and $\|y_2+y_0\|<\eps$. On the other hand, $\|y_1\|,\|y_2\|>1-\delta$, hence $\|x_1\|<\delta<\eps$ and $\|x_2\|<\delta<\eps$. Summarizing, we have
\begin{equation*}
\|u\|=\frac12\bigl(\|x_1+x_2\| + \|y_1+y_2\|\bigr)\leq \frac12(2\eps + 2\eps)=2\eps.\qedhere
\end{equation*}
\end{proof}

\begin{remark}
It is straightforward to adapt the previous proof to $\ell_p$-sums for $1<p<\infty$.
\end{remark}

However, note that the situation is very different for $\ell_\infty$-sums.

\begin{theorem}\label{thm:ccs_Daugavet_infinity_sums}
    Let $X$ and $Y$ be Banach spaces, and let $E:=X\oplus_\infty Y$. If $x\in S_X$ is a ccs Daugavet point, then $(x,y)\in S_E$ is a ccs Daugavet point for every $y\in B_Y$.
\end{theorem}

\begin{proof}
Let $C:=\sum_{i=1}^n \lambda_i S_i$ be a ccs of $B_E$. For every $i\in \{1,\dots,n\}$, we can write $S_i:=S(f_i,\delta_i)$ with $f_i:=(x_i^*,y_i^*)\in S_{E^*}$ satisfying  $1=\norm{f_i}=\norm{x_i^*}+\norm{y_i^*}$. Consider on the one side $$\tilde{S_i}:=\left\{s\in B_X\colon \re x_i^*(s)>\norm{x_i^*}-\frac{\delta_i}{2}\right\},$$ and pick on the other side any $t_i\in B_Y$ such that $\re y_i^*(t_i)>\norm{y_i^*}-\frac{\delta_i}{2}$. Since $\tilde{C}:=\sum_{i=1}^n \lambda_i\tilde{S_i}$ is a ccs of $B_X$, we can find for every $\eps>0$ an element $s:=\sum_{i=1}^n\lambda_is_i$ in $\tilde{C}$ such that $\norm{x-s}>2-\eps$. Then, if we let $t:=\sum_{i=1}^n\lambda_it_i$, we get $(s_i,t_i)\in B_E$ and 
$$\re f_i(s_i,t_i)= \re x_i^*(s_i)+y_i^*(t_i)>\norm{x_i^*}+\norm{y_i^*}-\delta_i=1-\delta_i$$ for every $i$, so that $(s_i,t_i)\in S_i$, and $(s,t)=\sum_{i=1}^n\lambda_i(s_i,t_i)\in C$. Finally, $$\norm{(x,y)-(s,t)}\geq \norm{x-s}>2-\eps,$$ so $(x,y)$ is a ccs Daugavet point as stated.  
\end{proof}

\section{Examples and counterexamples of diametral elements} \label{sec:examples_and_counterexamples}

In this section we aim to include a number of examples and counterexamples of diametral elements on the unit sphere of Banach spaces. We first characterize the notion in some spaces which have natural relations with the Daugavet property, such as $L_1$-preduals spaces, M\"{u}ntz spaces, and $L_1$-spaces. Next, we will remark on some examples which have previously appear in the literature, including some improvements in some cases (as for Lipschitz free spaces). Finally, we will include some complicated examples which will be needed to see that no implication in Figure~\ref{figure01} in page~\pageref{figure01} reverses and also to negate some other possible implications between the notions. A summary of all the relations between properties will be included in Subsection~\ref{subsection:schemeofcounterexamples}.

\subsection{Characterization in $C(K)$-spaces,  $L_1$-preduals, and M\"{u}ntz spaces}\label{subsec:L_1_preduals}
It was shown in \cite[Theorems 3.4 and 3.7]{ahlp20} that the notions of $\Delta$-point and Daugavet point coincide for $L_1$-preduals. The authors first characterize the $\Delta$-points in $C(K)$ spaces and then get the result for $L_1$-preduals by using the principle of local reflexivity. Later on, a characterization of $\Delta$-points (equivalently, Daugavet points) of $L_1$-preduals was provided in \cite[Theorem 3.2]{mr22} which implicitly prove that actually $\Delta$-points,  and super Daugavet points coincides in this setting. Let us state this result here for further reference. Let us observe that the authors of \cite{ahlp20} works with real Banach spaces, but it is immediate that the proof of \cite[Theorems 3.4]{ahlp20} works in the complex case as well; the paper \cite{mr22} works in both the real and the complex case.

\begin{proposition}[\mbox{\textrm{\cite[Theorems 3.4 and 3.7]{ahlp20}, \cite[Theorem 3.2]{mr22}}}]\label{proposition:delta-superDaugavet-L1predual-old}
Let $X$ be an $L_1$-predual and let $x\in S_X$. The following assertions are equivalent.
\begin{enumerate}
\item $x$ is a Daugavet point.
\item $x$ is a $\Delta$-point
\item For every $\delta>0$, the weak$^*$ slice $S(J_X(x),\delta;B_{X^*})$ contain infinitely many pairwise linearly independent extreme points of $B_{X^*}$.
\item For every element $y\in B_X$, there exists a sequence $(x_n^{**})$ in $B_{X^{**}}$ such that $\norm{x-x_n^{**}}\longrightarrow 2$ and $$\norm{\sum_{n\geq 1} a_n(y-x_n^{**})}\leq 2 \norm{a}_\infty$$ for every $a:=(a_n)\in c_{00}$.
\item For every element $y\in B_X$, there exists a sequence $(x_n^{**})$ in $B_{X^{**}}$ which converges weak$^*$ to $y$ and such that $\norm{x-x_n^{**}}\longrightarrow 2$.
\end{enumerate}
In the case that $X=C(K)$ for a Hausdorff topological space $K$, the above is also equivalent to:
\begin{enumerate}
\item[(6)] $x$ attains its norm at an accumulation point of $K$.
\end{enumerate}
\end{proposition}

We will show that, in fact, $\Delta$-points also coincide with the ccs versions for $L_1$ preduals. Our approach will be analogous to the one used in \cite{ahlp20} for $\Delta$-points and Daugavet points: we first prove the result for $C(K)$ spaces and then deduce it for all $L_1$-preduals using that the bidual of an $L_1$-predual is a $C(K)$-space. In the case of $C(K)$ spaces, we first prove a sufficient condition for ccs Daugavet points which, for the same price, can be proved for vector-valued spaces. Recall that given a compact Hausdorff topological space $K$ and a Banach space $X$, $C(K,X)$ denotes the Banach space of those continuous functions from $K$ to $X$ endowed with the supremum norm.

\begin{theorem}\label{thm:C(K)-deltaimpliesccsDaugavet}
Let $K$ be a compact Hausdorff topological space, $X$ a Banach space, and let $t_0$ be an accumulation point of $K$. If a function $f\in S_{C(K,X)}$ satisfies $\norm{f(t_0)}=1$, then $f$ is a ccs Daugavet point.
\end{theorem}

\begin{proof}
Pick $x^*\in S_{X^*}$ such that $\re x^*(f(t_0))=\norm{f(t_0)}=1$. Let $C:=\sum_{i=1}^L\lambda_iS_i$ be a convex combination of slices of $B_{C(K)}$. For every $i\in \{1,\dots, L\}$, pick a function $g_i\in S_i$. Since $K$ is compact and $t_0$ is an accumulation point of $K$ we have the following.

\emph{Claim.} There exists a sequence $(U_n)_{n\geq 0}$ of open neighborhoods of $t_0$ such that:
\begin{enumerate}
\item $U_0=K$,
\item $\overline{U_{n+1}}$ is a proper subset of $U_n$ for every $n\geq 0$,
\item $\re (x^*\circ f)_{\lvert \overline{U_n}}\geq 1-\frac{1}{n}$ and $\norm{{g_i}_{\lvert \overline{U_n}}-g_i(t_0)}\leq \frac{1}{n}$ for every $i\in \{1,\dots,L\}$ and every $n\geq 1$.
\end{enumerate}
Indeed, we construct the sequence inductively. Let $U_0:=K$ and assume that $U_0,\dots, U_n$ are constructed for some $n\geq 0$. Since $K$ is normal, we can find an open subset $U$ of $U_n$ such that $t_0\in U\subset \overline{U} \subset U_n$. Also since $t_0$ is an accumulation point of $K$ and since $K$ is Hausdorff, we can find an open subset $V$ of $U$ such that $\overline{V}$ is a proper subset of $\overline{U}$ (pick any point in $U$ distinct from $t_0$ and separate the two points with open sets). By continuity of $f$ and of the finitely many $g_i's$, we can then find an open subset $W$ of $V$ such that $$\re (x^*\circ f)_{\lvert W}> 1-\frac{1}{n+1} \ \ \text{ and } \norm{{g_i}_{\lvert W}-g_i(t_0)}< \frac{1}{n+1}$$ for every $i\in \{1,\dots,L\}$. The set $U_{n+1}:=W$ does the job.

Now, let us pick $(U_n)_{n\geq 0}$ as in the claim and let us define $F_n:=\overline{U_n}\backslash U_{n+1}$ for every $n\geq 0$. By construction, the $F_n's$ are closed non-empty subsets of $K$ and cover $K\backslash \left( \bigcap_{n\geq 0}U_n\right)$, and each $F_n$ may only intersects its neighbors $F_{n-1}$ and $F_{n+1}$. By Urysohn's lemma, for every $n\geq 1$ we can find a function $p_n\in C(K)$ satisfying: \begin{enumerate}
\item $0\leq p_n\leq 1$,
\item ${p_n}_{\lvert F_{n+1}}=1$,
\item ${p_n}_{\lvert F_0\cup \dots \cup F_{n-1}\cup \overline{U_{n+3}}}=0$.
\end{enumerate}
The sequence $(p_n)$ is normalized and converges pointwise to $0$, so it converges weakly to $0$. Moreover, observe that  $$\norm{g_i-(1+g_i(t_0))p_n}_\infty\leq 1+\frac{1}{n}$$ for every $i\in\{1,\dots, L\}$ since $\norm{{g_i}_{\lvert \overline{U_n}}-g_i(t_0)}\leq \frac{1}{n}$ and ${p_n}_{\lvert (K\backslash\overline{U_n})}=0$ by construction. So all the functions $$g_{i,n}:=\frac{n}{n+1}\left(g_i-(1+g_i(t_0))p_n\right)$$ belong to $B_{C(K)}$ and the sequences $(g_{i,n})_{n\in \N}$ converges weakly to $g_i$ for every $i\in\{1,\dots, L\}$. Since the finitely many $S_i's$ are all weakly open, we may thus find some $N\geq 1$ such that $g_{i,n}\in S_i$ for every $i$ and every $n\geq N$. In particular, the function $$g_n:=\sum_{i=1}^L\lambda_ig_{i,n}$$ belongs to $C$ for every $n\geq N$.  To conclude, fix $t\in F_{n+1} \subset \overline{U_{n+1}}\subset \overline{U_n}$ and observe that $$\re x^* f(t)\geq 1-\frac{1}{n}$$
and that
\begin{align*} \re x^*g_{i,n}(t) &= \frac{n}{n+1}\re x^*\left(g_i(t)-(1+g_i(t_0)\right) \\
&\leq \frac{n}{n+1}\re x^*\left(g_i(t_0)+\frac{1}{n}-(1+g_i(t_0))\right) =-1+\frac{2}{n+1}
\end{align*} for every $i\in\{1,\dots L\}$. Hence,
\begin{align*} \norm{f-g_n}_\infty &\geq \re x^*\left( f(t)-\sum_{i=1}^L\lambda_i\re(g_{i,n}(t)) \right)\geq 2-\frac{1}{n}-\frac{2}{n+1}.\qedhere
\end{align*}
\end{proof}

Combining the previous result with Proposition~\ref{proposition:delta-superDaugavet-L1predual-old}, we get the promised characterization of diametral points in $C(K)$ spaces.

\begin{corollary}\label{corollary:sixnotionsequalsinC(K)}
Let $K$ be a Hausdorff topological compact space. Then the six concepts of diametral points are equivalent in $C(K)$.  
\end{corollary}

For vector-valued spaces, the situation is not that easy, but we may provide with some results. Observe that, clearly, if $t_0$ is an isolated point of a compact Hausdorff topological space $K$ and $X$ is a Banach space, then $C(K,X)=C(K\setminus \{t_0\},X)\oplus_\infty X$.

\begin{remark}\label{remark:C(K)-vector-valued}
Let $K$ be a Hausdorff topological compact space, let $X$ be a Banach space. and let $f\in C(K,X)$ be a function with $\|f\|=1$.
\begin{enumerate}
    \item[(1)] If $f\in C(K,X)$ with $\|f\|=1$ attains its norm at an accumulation point of $K$, then $f$ is a ccs Daugavet point (by Theorem~\ref{thm:C(K)-deltaimpliesccsDaugavet}) and hence, $f$ satisfies the six diametral notions.
    \item[(2)] If $f\in C(K,X)$ with $\|f\|=1$ attains its norm at an isolated point $t_0$ and $f(t_0)$ is a Daugavet (respectively, super Daugavet, ccs Daugavet) point, then $f$ is a Daugavet (respectively, super Daugavet, ccs Daugavet) point (by \cite[Section 4]{ahlp20}, Remark~\ref{rem:l_1_sums}, and Theorem~\ref{thm:ccs_Daugavet_infinity_sums}, respectively).
    \item[(3)] Suppose that $K$ contains an isolated point $t_0$, let $x_0\in S_X$, and let $f\in C(K,X)$ be given by $f(t_0)=x_0$ and $f(t)=0$ for every $t\in K\setminus\{t_0\}$. Then:
    \begin{enumerate}
        \item[(3.1)] If $x_0$ is a $\Delta$- (respectively, super $\Delta$-) point of $X$, then $f$ is a $\Delta$- (respectively, super $\Delta$-) point of $C(K,X)$ (by \cite[Section 4]{ahlp20} and Proposition~\ref{prop:super_Delta_abs_sums}, respectively).
        \item[(3.2)] If $x_0$ is a Daugavet (respectively, super Daugavet, ccs Daugavet) point of $X$, then $f$ is a Daugavet (respectively, super Daugavet, ccs Daugavet) point of $C(K,X)$ (by  \cite[Proposition~3.3.11]{pirkthesis}, Remark~\ref{rem:l_1_sums} Theorem~\ref{thm:ccs_Daugavet_infinity_sums}, respectively).
        \item[(3.3)] If $f$ is a $\Delta$- (respectively, Daugavet) point of $C(K,X)$, then $x_0$ is a $\Delta$- (respectively, Daugavet) point of $X$ (by \cite[Theorem~3.4.4]{pirkthesis}, \cite[Theorem~3.3.13]{pirkthesis}, respectively).
    \end{enumerate}
    \item[(4)] It is now easy to show that the six diametral notions do not coincide in $C(K,X)$ spaces. Indeed, let $K$ be a compact Hausdorff topological space containing an isolated point $t_0$, let $X$ a Banach space containing a $\Delta$-point $x_0$ which is not a Daugavet point (e.g.\ any $x_0$ in the unit sphere of $X=C[0,1]\oplus_2 C[0,1])$, see Propositions \ref{prop:super_Delta_abs_sums} and \ref{prop:absolutesums-Daugavetpoints}), and consider the function $f\in C(K,X)$ given by $f(t_0)=x_0$ and $f(t)=0$ for every $t\in K\setminus\{t_0\}$. Then, $f$ is a $\Delta$-point by (3.1) but it is not a Daugavet point by (3.3).
\end{enumerate}
\end{remark}

We are now ready to extend Corollary~\ref{corollary:sixnotionsequalsinC(K)} to general $L_1$-predual spaces.

\begin{corollary}\label{corollary:L1-preduals-deltaiffccs}
Let $X$ be an $L_1$-predual and let $x\in S_X$ be a $\Delta$-point. Then, $x$ is a ccs Daugavet point. Hence the six diametral notions are equivalent for $L_1$-preduals.
\end{corollary}

\begin{proof}
If $x$ is a $\Delta$-point in $X$, then as mentioned in item (3) of Remark \ref{rem:weak*_diametral_points}, we have that $J_X(x)$ is a $\Delta$-point in $X^{**}$. Now, $X^{**}$ is isometric to a $C(K)$ space so Theorem~\ref{thm:C(K)-deltaimpliesccsDaugavet} gives that $J_X(x)$ is a ccs Daugavet point in $X^{**}$. Then, using now item (4) of Remark \ref{rem:weak*_diametral_points} (or using a straightforward argument based on the principle of local reflexivity as in \cite[Theorem 3.7]{ahlp20}), we get that $x$ is a ccs Daugavet point in $X$.
\end{proof}

Let us observe that the proof of Theorem~\ref{thm:C(K)-deltaimpliesccsDaugavet} also works for M\"untz spaces (by using \cite[Lemma 3.10]{ahlp20} to provide suitable replacements for the functions $p_n$). We recall that given an
an increasing sequence  $\Lambda = (\lambda_n)_{n=0}^\infty$
of non-negative real numbers with $\lambda_0=0$ such that  $\sum_{i=1}^{\infty} \frac{1}{\lambda_i} < \infty$, then the real Banach space
$$
M(\Lambda) := \cspan \{t^{\lambda_n}\colon n\geq 0\} \subseteq C[0,1]
$$
is called the \emph{M\"{u}ntz space} associated with $\Lambda$.
Excluding the constant functions form $M(\Lambda)$, we have the subspace $M_0(\Lambda) := \cspan \{t^{\lambda_n}\colon n\geq 1\}$ of $M(\Lambda)$.

So, adapting the proof of Theorem~\ref{thm:C(K)-deltaimpliesccsDaugavet} to M\"untz spaces (for real scalar-valued functions attaining its norm at $1\in [0,1]$) and also using \cite[Proposition~3.12]{ahlp20}, we get the following result analogous to Corollaries \ref{corollary:sixnotionsequalsinC(K)} and \ref{corollary:L1-preduals-deltaiffccs}.

\begin{corollary}\label{coro:Muntz}
Let $X=M(\Lambda)$ or $X=M_0(\Lambda)$ for an increasing sequence $\Lambda$ of non-negative real numbers with $\lambda_0=0$ such that  $\sum_{i=1}^{\infty} \frac{1}{\lambda_i} < \infty$. Then, every $\Delta$-point of $X$ is a ccs Daugavet point (and hence the six diametral notions are equivalent).
\end{corollary}


\subsection{Characterization in $L_1$-spaces}\label{subsec:L_1_spaces}
In \cite[Theorem~3.1]{ahlp20} the equivalence between the notions of Daugavet point and $\Delta$-point was obtained for elements of $\sigma$-finite $L_1$-spaces in the real case. Actually, it is not complicated to extend the results to arbitrary measures and also to the complex case. 

\begin{proposition}[\mbox{\textrm{\cite[Theorem~3.1]{ahlp20} for the $\sigma$-finite real case}}]
\label{prop:L1-spaces-previous}
Let $(\Omega,\Sigma,\mu)$ be a measure space, and let $f$ be a norm one element in $L_1(\mu)$. Then, the following assertions are equivalent.
\begin{enumerate}
\item $f$ is a Daugavet point.
\item $f$ is a $\Delta$-point.
\item The support of the function $f$ contains no atom.
\end{enumerate}
\end{proposition}

Observe that (1) implies (2) is immediate. For (2) implies (3), suppose that $f$ is a $\Delta$-point and let $A$ be an atom of finite measure (the only ones that can be contained in the support of an integrable function). Then, we clearly have that $L_1(\mu)=L_1(\mu|_{\Omega\setminus A}) \oplus_1 \K$ (as integrable functions are constant on atoms), and we may write $f=(f_1,c)$ for suitable $f_1\in L_1(\mu|_{\Omega\setminus A})$ and $c=f(A)\in \K$. If $c\neq 0$, then $\|f_1\|\neq 1$ and it follows from \cite[Theorem~3.4.4]{pirkthesis} that $1\in \mathbb{K}$ is a $\Delta$-point, a contradiction. This shows that the support of $f$ does not contain any atom.

To get that (3) implies (1), we actually prove the following more general result. Recall that given a measured space $(\Omega,\Sigma,\mu)$ and a Banach space $X$, $L_1(\mu,X)$ denotes the Banach space of all Böchner-integrable functions from $\Omega$ to $X$.

\begin{theorem}\label{thrm:L_1_ccs_Delta}
Let $(\Omega,\Sigma,\mu)$ be a measured space, let $X$ be a Banach space, and let $f$ be a norm one element in $L_1(\mu,X)$. If the support of the function $f$ contains no atom, then $f$ is a super Daugavet point.
\end{theorem}

\begin{proof}
Let us write $S:=\supp f$ which contains no atom by hypothesis.
Let us first prove that $f$ is a super Daugavet point. Since $S$ contains no atoms, we have that $L_1(\mu_{\lvert S},X)$ satisfies the Daugavet property (see e.g.\ \cite[Example in p.~81]{werner}). In particular, $f$ is a super Daugavet point in this space. Since $L_1(\mu,X)=L_1(\mu_{\lvert S},X)\oplus_1 L_1(\mu_{\lvert \Sigma\backslash S},X)$, we get that $f$ is a super Daugavet point in $L_1(\mu,X)$ by the transfer results from Subsection~\ref{subsec:absolutesums} (see Remark \ref{rem:l_1_sums}).
\end{proof}

Our next goal is to discuss the relationship with the ccs diametral notions. For real $L_1(\mu)$-spaces, and using a result from \cite{abhlp}, we may actually get that real-valued integrable functions with atomless support are ccs $\Delta$-points.

\begin{proposition}\label{proposition:realcaseL1mu-ccs-delta}
Let $(\Omega,\Sigma,\mu)$ be a measured space and let $f$ be a norm one element in \emph{the real space} $L_1(\mu)$. If the support of the function $f$ contains no atom, then $f$ is a ccs $\Delta$-point.   
\end{proposition}

\begin{proof}
Take $\eps>0$ and $D:=\sum_{i=1}^n\lambda_i S_i$ a ccs of $B_{L_1(\mu)}$ containing $f$. Write $f:=\sum_{i=1}^n\lambda_i g_i$ with $g_i\in S_i$ for every $i$. Consider the measurable subset $\tilde{S}:=\supp f\cup\bigcup_{i=1}^n \supp g_i$ of $\Omega$ and let $\tilde{\mu}$ be the $\sigma$-finite measure $\tilde{\mu}:=\mu_{\lvert \tilde{S}}$ on $(\tilde{S},\Sigma_{\lvert \tilde{S}})$. Then $D$ induces a ccs $\tilde{D}$ of $B_{L_1(\tilde{\mu})}$ by restriction of the support which contains the function $\tilde{f}$ which is just $f$ viewed as an element of $L_1(\tilde{\mu})$ and hence, the support of $\tilde{f}$ does not contains atoms. Since $\tilde{f}$ belongs to the unit sphere of the real space $L_1(\tilde{\mu})$, we have by \cite[Theorem 5.5]{abhlp} that $\tilde{f}$ is an interior point of $\tilde{D}$ for the relative weak topology of $B_{L_1(\tilde{\mu})}$. As we have already shown that $\tilde{f}$ is a super Daugavet point in Theorem~\ref{thrm:L_1_ccs_Delta} (and hence a super $\Delta$-point), we can find $\tilde{g}\in \tilde{D}$ such that $\bigl\|\tilde{f}-\tilde{g}\bigr\|>2-\eps$. By just considering the extension $g$ of $\tilde{g}$ to the whole $\Omega$ by $0$, we get that $g\in D$ and that $\norm{f-g}=\bigl\|\tilde{f}-\tilde{g}\bigr\|>2-\eps$.
\end{proof}

Let us comment that it is not clear whether ccs $\Delta$-points transfer through absolute sums, but we have used specific geometric properties of $L_1$-spaces in the previous proof.

\begin{rem}
Observe that since \cite[Theorem 5.5]{abhlp} is also valid for convex combination of relative weakly open subsets of $B_{L_1(\mu)}$, we in fact have that every $\Delta$-point in a real $L_1(\mu)$ space is actually a ccw $\Delta$-point.
\end{rem}

Putting together Proposition~\ref{prop:L1-spaces-previous}, Theorem~\ref{thrm:L_1_ccs_Delta}, and Proposition~\ref{proposition:realcaseL1mu-ccs-delta}, we get the following corollary.

\begin{corollary}
Let $(\Omega,\Sigma,\mu)$ be a measured space and let $f$ be a norm one element in $L_1(\mu)$. Then, the following notions are equivalent for $f$: $\Delta$-points, Daugavet point, super $\Delta$-point, and super Daugavet point. Moreover, in the real case, the previous four notions are also equivalent to being ccs $\Delta$-point.
\end{corollary}

We now deal with ccs Daugavet points in $L_1(\mu)$-spaces. Observe that if $\Omega$ admits an atom $A$ of finite measure, then we have $L_1(\mu)\equiv L_1(\mu_{\lvert \Omega\backslash A})\oplus_1 \K$. In particular, in this case $L_1(\mu)$ fails to have ccs Daugavet points by Proposition~\ref{prop:oplus1stronglyexposedpoint}. We then have the following characterization of the presence of a ccs Daugavet point in an $L_1$-space.

\begin{prop}\label{prop:L_1_ccs_Daugavet}
Let $(\Omega,\Sigma,\mu)$ be a measure space. Then, the following assertions are equivalent.
\begin{enumerate}
\item $L_1(\mu)$ has the Daugavet property.
\item $L_1(\mu)$ contains a ccs Daugavet point.
\item $L_1(\mu)$ has the SD2P.
\item $\mu$ admits no atom of finite measure.
\end{enumerate}
\end{prop}

\begin{proof}
(1)$\Leftrightarrow$(4) is well known (see \cite[Section 2, Example (b)]{werner}); (1)$\Rightarrow$(2) is also known; (2)$\Rightarrow$(3) is contained in Proposition~\ref{prop:ccs_Daugavet_SR}. Finally,  (3)$\Rightarrow$(4) follows from Proposition~\ref{prop:oplus1stronglyexposedpoint} and the comment before the statement of this proposition.
\end{proof}

\subsection{Remarks on some examples from the literature}\label{subsect:fromliterature}
\subsubsection{Two examples in Lipschitz-free spaces}\label{subsubsec:extreme_molecules}
In \cite{veefree}, Veeorg constructed a surprising example of a space satisfying the Radon-Nikod\'{y}m property and containing a Daugavet point. We slightly improve this result by showing that this point is also a ccs $\Delta$-point by proving a general fact about extreme $\Delta$-molecules in Lipschitz-free spaces. For the necessary definitions we refer to the cited paper \cite{veefree} and to \cite{AGPP-TAMS2022,ANPP,jr22}; for further background on Lipschitz-free spaces, we refer to the book \cite{wea2}.

For this purpose, we start by recalling the following characterization of molecules which are $\Delta$-points on Lipschitz-free spaces from \cite{jr22}.

\begin{prop}[\mbox{\textrm{\cite[Theorem 4.7]{jr22}}}]
\label{prop:Delta_molecules_free_spaces}
Let $M$ be a pointed metric space and let $x\neq y\in M$. The molecule $m_{x,y}$ is a $\Delta$-point if and only if every slice $S$ of $B_{\F(M)}$ containing $m_{x,y}$ also contains for every $\eps>0$ a molecule $m_{u,v}$ with $u\neq v\in M$ satisfying $d(u,v)<\eps$.
\end{prop}

In the case in which the molecule is an extreme point, we have the following improved result.

\begin{theorem}\label{thrm:extreme_Delta_molecules}
Let $M$ be a pointed metric space, and let $x\neq y\in M$. If the molecule $m_{x,y}$ is an extreme point and a $\Delta$-point, then $m_{x,y}$ is a ccs $\Delta$-point.
\end{theorem}

Observe that this result cannot be obtained from  Proposition~\ref{prop:extreme_diametral_delta}: molecules of Lipschitz-free spaces which are preserved extreme points are denting points, hence very far from being $\Delta$-points.

To give the proof of the theorem, we need a result which is just an  equivalent reformulation of a result in \cite{jr22}.

\begin{lemma}[\mbox{\textrm{\cite[Theorem 2.6]{jr22}}}]
\label{lemma:distance_to_closely_supported_molecules}
Let $M$ be a pointed metric space, and let $\mu\in S_{\F(M)}$. For every $\eps>0$, there exists $\delta>0$ such that given $u\neq v\in M$ with $d(u,v)<\delta$ we have $\norm{\mu\pm m_{u,v}}>2-\eps$.
\end{lemma}

Using this result and a homogeneity argument similar to the one from \cite[Lemma 2.3]{blr14}, we can provide the pending proof.

\begin{proof}[Proof of Theorem~\ref{thrm:extreme_Delta_molecules}]
Let $C:=\sum_{i=1}^n\lambda_iS_i$ be a ccs of $B_{\F(M)}$ containing $m_{x,y}$ and let $\eps>0$. Since $m_{x,y}$ is extreme, we have that $m_{x,y}\in \bigcap_{i=1}^n S_i$, and by Proposition \ref{prop:Delta_molecules_free_spaces} every $S_i$ contains molecules of $\F(M)$ supported at arbitrarily close points. Using Lemma \ref{lemma:distance_to_closely_supported_molecules}, we construct inductively for every $\eta>0$ a finite sequence $(m_{u_i,v_i})_{i=1}^n$ of molecules in $\F(M)$ such that
\begin{enumerate}
\item $m_{u_i,v_i}\in S_i$ for every $i$.
\item  $\norm{m_{x,y}-\sum_{i=1}^k\lambda_im_{u_i,v_i}} >1+\sum_{i=1}^k\lambda_i-\frac{k\eps}{n}$ for every $k\leq n$.
\end{enumerate}
Indeed, since $S_1$ contains molecules of $\F(M)$ supported at arbitrarily close points, we can find by Lemma \ref{lemma:distance_to_closely_supported_molecules} $u_1\neq v_1\in M$ such that $m_{u_1,v_1}\in S_1$ and $\norm{m_{x,y}-m_{u_1,v_1}}>2-\frac{\eps}{n}$. It follows that $\norm{m_{x,y}-\lambda_1m_{u_1,v_1}}\geq \norm{m_{x,y}-m_{u_1,v_1}}-(1-\lambda_1)>1+\lambda_1-\frac{\eps}{n}$. Let us assume that $m_{u_1,v_1},\dots, m_{u_k,v_k}$ are constructed as desired for a given $k\in\{1,\dots, n-1\}$. Since $S_{k+1}$ contains molecules of $\F(M)$ supported at arbitrarily close points, we can find by Lemma \ref{lemma:distance_to_closely_supported_molecules} $u_{k+1}\neq v_{k+1}\in M$ such that $m_{u_{k+1},v_{k+1}}\in S_{k+1}$ and $$\norm{\frac{m_{x,y}-\sum_{i=1}^k\lambda_im_{u_i,v_i}}{\norm{m_{x,y}-\sum_{i=1}^k\lambda_im_{u_i,v_i}}}-m_{u_{k+1},v_{k+1}}}>2-\frac{\eps}{n\norm{m_{x,y}-\sum_{i=1}^k\lambda_im_{u_i,v_i}}}.$$ Then, \begin{align*}\norm{\frac{m_{x,y}-\sum_{i=1}^{k+1}\lambda_im_{u_i,v_i}}{\norm{m_{x,y}-\sum_{i=1}^k\lambda_im_{u_i,v_i}}}} &\geq \norm{\frac{m_{x,y}-\sum_{i=1}^k\lambda_im_{u_i,v_i}}{\norm{m_{x,y}-\sum_{i=1}^k\lambda_im_{u_i,v_i}}}-m_{u_{k+1},v_{k+1}}} \\
& \quad -\left(1-\frac{\lambda_{k+1}}{\norm{m_{x,y}-\sum_{i=1}^k\lambda_im_{u_i,v_i}}}\right)\\
&>1+\frac{\lambda_{k+1}}{\norm{m_{x,y}-\sum_{i=1}^k\lambda_im_{u_i,v_i}}}- \frac{\eps}{n\norm{m_{x,y}-\sum_{i=1}^k\lambda_im_{u_i,v_i}}}.
\end{align*}
By the assumption,
\begin{align*}
\norm{m_{x,y}-\sum_{i=1}^{k+1}\lambda_im_{u_i,v_i}}> \norm{m_{x,y}-\sum_{i=1}^k\lambda_im_{u_i,v_i}}+\lambda_{k+1}-\frac{\eps}{n} >1+\sum_{i=1}^{k+1}\lambda_i-\frac{(k+1)\eps}{n}.
\end{align*}

As a consequence, $\mu:=\sum_{i=1}^n\lambda_im_{u_i,v_i}$ belongs to $C$ and satisfies $\norm{m_{x,y}-\mu}>2-\eps$.
\end{proof}

In particular, we have, as announced, that the molecule $m_{x,y}$ in the example from \cite{veefree} is a ccs $\Delta$-point. Note that it cannot be a ccs Daugavet point by Proposition \ref{prop:ccs_Daugavet_SR} since the space has the RNP, but we do not know whether it is a super $\Delta$-point or even a super Daugavet point. Let us state the result for further reference.

\begin{example}\label{example:veefree}
Let $M$ be the metric space constructed in \cite[Example 3.1]{veefree} and let $x,y$ be the points described there. Then, $\mathcal{F}(M)$ has the RNP, the molecule $m_{x,y}$ is an extreme point of the unit ball of $\mathcal{F}(M)$ which is a Daugavet point. Hence, by our Theorem~\ref{thrm:extreme_Delta_molecules}, $m_{x,y}$ is a ccs-$\Delta$-point.
\end{example}

Another interesting example in the Lipschitz-free space setting is the following one which uses a metric space constructed by Aliaga, No\^{u}s, Petitjean, and Proch\'{a}zka \cite{ANPP}.

\begin{example}\label{example:Lipfree-ANPP}
Let $M$ be the metric space from \cite[Examples 4.2]{ANPP}. Then one can check that the molecule $m_{0,q}$ is an extreme point of $B_{\F(M)}$ and, since the points $0$ and $q$ are discretely connectable, it follows from an easy adjustment of \cite[Proposition 4.2]{jr22} that this molecule is a $\Delta$-point. In particular, it follows from Theorem~\ref{thrm:extreme_Delta_molecules} that this molecule is also a ccs $\Delta$-point. However, it is not difficult to show that there exists denting points in $B_{\F(M)}$ that are at distance strictly less than $2$ to $m_{0,q}$ (take any among the molecules $m_{x_i^n,x_{i+1}^n}$), so this molecule is not a Daugavet point. Also observe that this space has the RNP since the metric space $M$ is countable and complete \cite[Theorem 4.6]{AGPP-TAMS2022}.
\end{example}

Let us finally remark that the spaces $\mathcal{F}(M)$ of Examples \ref{example:veefree} and \ref{example:Lipfree-ANPP} have the RNP, so they are strongly regular and hence strongly regular points are norm dense, but both examples have ccs $\Delta$-points. They cannot contain ccs Daugavet points by Proposition~\ref{prop:ccs_Daugavet_SR}.

Let us also comment that the use of Theorem~\ref{thrm:extreme_Delta_molecules} above cannot be omitted, as the molecule $m_{0,q}$ is not a preserved extreme point, hence Proposition~\ref{prop:extreme_diametral_delta} is again not applicable.

\subsubsection{An example of a Banach space with the DD2P, the restricted DSD2P, but containing ccs of arbitrarily small diameter}\label{subsubsect:MLUR}

In \cite[Theorem~2.12]{ahntt16}, Abrahamsen, H\'ajek, Nygaard, Talponen, and Troyanski constructed a space $X$ which has the DLD2P, which is midpoint locally uniformly rotund (in particular, satisfying that $\preext{B_X}=S_X$), and such that $B_X$ contains convex combinations of slices of arbitrarily small diameter. It then follows from Proposition~\ref{prop:extreme_diametral_delta} that every element of $S_X$ is actually a super $\Delta$-point and a ccs $\Delta$-point (that is, $X$ has the DD2P and the restricted DSD2P). But containing ccs of arbitrarily small diameter, $X$ fails the SD2P. The obvious explanation for the failure of the SD2P and the fact that every element in the unit sphere is a ccs $\Delta$-point is that none of the convex combinations of slices of diameter strictly smaller than 2 intersects the unit sphere.
On the other hand, the space $X$ is constructed as the $\ell_2$-sum of spaces, and so $X$ does not contain Daugavet points by \cite[Proposition 4.6]{ahlp20} (see Proposition~\ref{prop:absolutesums-Daugavetpoints}).

Observe further that $X$ has the restricted DSD2P and the DD2P, but fails the DSD2P (which is equivalent to the Daugavet property by \cite{kadets20}).

\subsubsection{An example in a space with one-unconditional basis}\label{subsubsec:oneunconditionalbasis}

Abrahamsen, Lima, Martiny, and Troyanski constructed in \cite[Section 4]{almt21} a Banach space $X_{\mathfrak{M}}$ with one-unconditional basis which contains a subset $D_B\subseteq S_{X_{\mathfrak{M}}}$ satisfying:
\begin{itemize}
  \item Every element in $D_B$ is both a Daugavet point and a point of continuity;
  \item $B_{X_{\mathfrak{M}}}=\cconv(D_B)$;
  \item $D_B$ is weakly dense in the unit ball.
\end{itemize}
Observe that no element of $D_B$ is a super $\Delta$-point (it is exactly the opposite!). By Theorem~\ref{thm:one-unconditional_no_ccs_nor_super_Delta}, no element of $D_B$ is a ccs $\Delta$-point.

\subsection{A super \texorpdfstring{$\Delta$}{Delta}-point which fails to be a Daugavet point in an extreme way}\label{subsect:superdeltanotdauga}
In order to put into a context the following result, let us recall that Daugavet points are at distance $2$ from any denting point (see \cite[Proposition 3.1]{jr22}). With this in mind, the following result can be interpreted as the existence of super $\Delta$-points which fail to be Daugavet points in an extreme way.

\begin{theorem}\label{theo:renormingDP}
Let $X$ be a Banach space with the Daugavet property. Then, for every $\varepsilon>0$, there exists an equivalent norm $\vert\cdot\vert$ and two points $x,y\in B_{(X,\vert\cdot\vert)}$ such that
\begin{enumerate}
\item $y$ is a super $\Delta$-point.
\item $x$ is strongly exposed.
\item $\vert x-y\vert<\varepsilon$.
\end{enumerate}
\end{theorem}

\begin{proof}
Take a subspace $Y\subseteq X$ with $\dim(X/Y)=1$. Observe that $Y$ has the Daugavet property (see e.g.\ \cite[Theorem 6 (a)]{shv00}). Take $x\in S_X$ with $0<d(x,Y)<\varepsilon$ (this can be settled taking a non-zero element $v\in X/Y$ with quotient norm smaller than $\varepsilon$). Now, we can find an element $y\in S_Y$ such that $\Vert x-y\Vert<\varepsilon$. By the Hahn-Banach theorem, we can take $f\in S_{X^*}$ with $\re f(x)>0$ and $f=0$ on $Y$. This means that $x$ belongs to the slice $T:=\{z\in B_X\colon \re f(z)>\alpha\}$ for some $\alpha>0$. Take $\delta>0$ such that $\frac{\Vert x-y\Vert}{1-\delta}<\varepsilon$. By Lemma \ref{lemma:diminition} we can find $x^*\in S_{X^*}$ such that $x\in S(x^*,\delta;B_X)\subseteq T$. By the above inclusion we conclude that $S(x^*,\delta;B_X)\cap B_Y=\emptyset$ or, in other words, that $\re x^*(z)\leq 1-\delta$ for every $z\in B_Y$. Set
$$B:=\cco(B_Y\cup (1-\delta)B_X\cup \{\pm x\}).$$
$B$ is the unit ball of an equivalent norm $\vert\cdot \vert$ which satisfies, in view of the inclusions $(1-\delta)B_X\subseteq B\subseteq B_X$, that
$$\Vert x\Vert\leq \vert x\vert\leq \frac{1}{1-\delta}\Vert x\Vert$$
for every $x\in X$. Let us prove that $\vert \cdot\vert$, $x$ and $y$ satisfies our requirements. First, observe that
$$\vert x-y\vert\leq \frac{\Vert x-y\Vert}{1-\delta}<\varepsilon.$$
Next, we claim that $y$ is a super $\Delta$ point. Indeed, since $Y$ has the Daugavet property we can find a net $\{y_s\}\subseteq B_Y$ with $\{y_s\}\longrightarrow y$ weakly and $\Vert y-y_s\Vert\longrightarrow 2$. Notice that the weak convergence $\{y_s\}\longrightarrow y$ is still guaranteed on $X$ because $i\colon(Y,\Vert\cdot\Vert)\longrightarrow (X,\vert \cdot\vert)$ is weak to weak continuous as $\Vert\cdot\Vert$ and $\vert\cdot\vert$ are equivalent. Moreover, notice that $y_s\in B_Y\subseteq B$ for every $s$, so $\vert y_s\vert\leq 1$ for every $s$. Finally,
$$\vert y_s-y\vert\geq \Vert y_s-y\Vert\longrightarrow 2,$$
and since $y\in B_Y\subseteq B$, we conclude $\vert y_s-y\vert\longrightarrow 2$. From there, $y$ is clearly a super $\Delta$-point for the norm $\vert\cdot\vert$.

It remains to prove that $x$ is strongly exposed. Indeed, we will prove that $\re x^*$ strongly exposes $B$ at $x$, for which it is enough to prove that $\re x^*$ strongly exposes $\co(B_Y\cup (1-\delta)B_X\cup \{\pm x\})$ at $x$.
Take $z:=\alpha u+\beta (1-\delta)v+(\gamma-\omega)x\in \co(B_Y\cup (1-\delta)B_X\cup \{\pm x\})$ with $\alpha+\beta+\gamma+\omega=1$. Observe that $1-\delta<\re x^*(x)\leq \vert x^*\vert\leq \Vert x^*\Vert$ due to the inclusion $B\subseteq B_X$. Taking into account that $\re x^*(u)\leq 1-\delta$ since $u\in B_Y$ as $B_Y\cap S(x^*,\delta;B_X)=\emptyset$, we conclude
$$
\re x^*(z)\leq (1-\delta)(\alpha+\beta)+(\gamma-\omega)\re x^*(x).
$$
Since $\re x^*(x)>1-\delta$, we get that   $\sup\bigl\{\re x^*(z)\colon z\in \conv(B_Y\cup (1-\delta)B_X\cup \{\pm x\})\bigr\}=\re x^*(x)$. If we take a
sequence
$$
z_n:=\alpha_n u_n+\beta_n (1-\delta)v_n+(\gamma_n-\omega_n)x\in \co(B_Y\cup (1-\delta)B_X\cup \{\pm x\})
$$
with $\alpha_n+\beta_n+\gamma_n+\omega_n=1$ such that $\re x^*(z_n)\longrightarrow \re x^*(x)$, it follows from the previous argument that $\alpha_n\rightarrow 0, \beta_n\rightarrow 0, \omega_n\rightarrow 0$ and $\gamma_n\rightarrow 1$, which means $z_n\rightarrow x$ in norm.
\end{proof}

\begin{remark}\label{remark:super-delta-closed-to-strongly-exposed}
Using the previous theorem and Proposition~\ref{prop:super_Delta_abs_sums} it is easy to construct (considering $\ell_2$-sums, for instance) a Banach space $X$ containing a sequence of super $\Delta$-points $(y_n)$ such that the distance from $y_n$ to the set of strongly exposed points is going to zero.
\end{remark}

\subsection{A super \texorpdfstring{$\Delta$}{Delta}-point which is a strongly regular point}\label{subsect:superDeltastronglyregular}

In the present subsection, as well as in the next, we aim to distinguish the super and ccs notions of $\Delta$- and Daugavet points. The following result shows that there are plenty of examples of spaces containing super $\Delta$-points which are strongly regular points (hence far from being ccs $\Delta$-points). We do the construction in for real spaces for simplicity.

\begin{theorem}\label{theorem_super-Delta-pointofSR}
Every real Banach space with the Daugavet property can be equivalently renormed so that the new unit ball has a point which is simultaneously super-$\Delta$ and a point of strong regularity (hence, far away of being ccs $\Delta$-point).
\end{theorem}

We will use the following immediate result which follows from the fact that a convex combination of ccs is again a ccs.

\begin{lemma}\label{lemma_SRPisconvex}
Let $X$ be a Banach space and let $C$ be a closed, convex, bounded subset of $X$. Then the set of strongly regular points of $C$ is a convex set.
\end{lemma}

\begin{proof}[Proof of Theorem~\ref{theorem_super-Delta-pointofSR}]
Let $X$ be a Banach space with the Daugavet property. Take a $1$-codimensional subspace $Y$ of $X$. Since $Y$ is complemented in $X$ then $X=Y\oplus \mathbb R$, so we will see $X$ in such way. Take $r>0$, $y_0\in S_Y$ and $f\in S_{X^*}$ such that $f(y_0)=1$, and consider on $X=Y\oplus \mathbb R$ the equivalent norm $\vert\cdot\vert$ whose unit ball is $B:=\cco\bigl(B_Y\times\{0\}\cup \{\pm (y_0,r)\}\cup\{\pm (y_0,-r)\}\bigr)$. It readily follows that $\vert\cdot\vert$ agrees with the original norm $\Vert \cdot\Vert$ on the elements of the form $(y,0)$.

We claim that $(y_0,0)$ satisfies our requirements. First of all, let us prove that $(y_0,0)$ is a super-$\Delta$ point. Since $Y$ is one-codimensional, it has the Daugavet property (see e.g.\ \cite[Theorem 6 (a)]{shv00}). Consequently, there exists a net $(y_s)\longrightarrow y_0$ weakly in $B_Y$ such that $\Vert y_0-y_s\Vert\longrightarrow 2$. Then, $(y_s,0)\longrightarrow (y_0,0)$ weakly in $(X,\vert\cdot\vert)$. Moreover, it is clear that $(y_s,0)\in B$ for every $s$. Finally,
$$
\vert (y_s,0)-(y_0,0)\vert=\vert (y_s-y_0,0)\vert=\Vert y_s-y_0\Vert\longrightarrow 2.
$$
Let us now prove that $(y_0,0)$ is a point of strong regularity. To do so, it is enough, in view of Lemma~\ref{lemma_SRPisconvex}, to show that $(y_0,\pm r)$ is a strongly exposed point (we will prove that for $(y_0,r)$, being the other case completely analogous). Let us prove that $\re(f,1)$ strongly exposes $(y_0,r)$ in the set $B_Y\times\{0\}\cup \{\pm (y_0,r)\}\cup\{\pm (y_0,-r)\}$. On the one hand, we have
$$
\re (f,1)(y_0,r)=\re f(y_0)+r=1+r.
$$
On the other hand, given $(y,0)\in B_Y\times 0$ we have $\re(f,1)(y,0)=f(y)\leq 1<1+r$. Moreover, $\re(f,1)(y_0,-r)=1-r$ and $\re (f,1)(-y_0,\pm r)=-1\pm r<1+r$. Consequently,
\begin{align*}
\sup\{\re (f,1)(a,b)\colon  (a,b)\in B_Y\times\{0\}\cup \{\pm (y_0,r)\}\cup\{\pm (y_0,-r)\},\, (a,b)\neq (y_0,r)\} \\  \leq 1<1+r=\re (f,1)(y_0,r).
\end{align*}
This is enough to guarantee that $\re (f,1)$ strongly exposes $(y_0,r)$ in $B$, so we are done.
\end{proof}

\subsection{A super Daugavet point which is not ccs \texorpdfstring{$\Delta$}{Delta}-point}\label{subsect:superdauganotccsdelta}

The previous example shows that we can distinguish the notion of super $\Delta$-point and the one of ccs $\Delta$-point. It seems natural then that we should be able to distinguish the notions of super Daugavet point and the one of ccs $\Delta$-point. In order to do so, we need to consider an involved construction but, as a consequence, we will prove that there are super Daugavet points which are contained in convex combinations of slices of small diameter. The construction will be very similar to that of \cite[Theorem~2.4]{blradv}, with a slight variation which makes the resulting norm with a stronger Daugavet flavour. As in the previous subsection, we will only work with real spaces here.

In order to do so, let us recall a construction from Argyros, Odell, and Rosenthal \cite{aor}. Pick a nonincreasing null sequence $\{\varepsilon_n\}$ in $\mathbb R^+$. We construct an increasing sequence of closed, bounded and convex subsets $\{K_n\}$ in the real space $c_0$ and a sequence $\{g_n\}$ in $c_0$ as follows: First define $K_1=\{e_1\}$, $g_1=e_1$ and $K_2=\conv(e_1, e_1+e_2)$. Choose $l_2>1$ and $g_2,\ldots ,g_{l_2}\in K_2$ an $\varepsilon_2$-net in $K_2$. Assume that $n\geq 2$ and that $m_n,\ l_n
,\ K_n$, and $\{g_1,\ldots ,g_{l_n}\}$ have been constructed, with
$K_n\subseteq B_{\spn\{e_1,\ldots ,e_{m_n}\}}$ and $g_i\in K_n$ for every
$1\leq i\leq l_n$. Define $K_{n+1}$ as $$K_{n+1}=\conv(K_n\cup
\{g_i+e_{m_n+i}\colon 1\leq i\leq l_n\}).$$ Consider $m_{n+1}=m_n+l_n$ and
choose $\{g_{l_{n}+1},\ldots ,g_{l_{n+1}}\}\in K_{n+1}$ so that
$\{g_1 ,\ldots ,g_{l_{n+1}}\}$ is an $\varepsilon_{n+1}$-net in
$K_{n+1}$. Finally, we define $K_0=\overline{\cup_n K_n}$. Then it
follows that $K_0$ is a non-empty closed, bounded and convex subset
of $c_0$ such that $x(n)\geq 0$ for every $n\in \natu$ and $\Vert
x\Vert_{\infty}\leq 1$ for every $x\in K_0$ and so $\diam(K_0)\leq 1$.

Now, for a fixed $i$, we have from the construction that
$\{g_i+e_{m_n+i}\}_n$ is a sequence in $K_0$ (for $n$ large enough) which is weakly convergent to
$g_i$, and $\Vert (g_i-e_{m_n+i})-g_i\Vert=\Vert e_{m_n+i}\Vert=1$
holds for every $n$. Then $\diam(K_0)=1$. We will freely use the set $K_0$ and the above construction throughout the subsection. Observe that, from the
above construction, it follows that
$$K_0=\overline{\{g_i\colon i\in\natu\}}^{w}=\overline{\{g_i\colon i\in\natu\}}.$$
Observe finally that, by the inductive construction, $g_i$ has finite support for every $i\in\mathbb N$.

By \cite[Theorem 1.2]{aor} we have that $K_0$ contains convex combinations of slices of arbitrarily small diameter. However, all the points in $K_0$ are ``super Daugavet points'' in the following sense.

\begin{proposition}\label{prop:examk0superdauga}
For every $x_0\in K_0$, every $\varepsilon>0$, and every non-empty weakly open subset $W$ of $K_0$, there exists $y\in W$ satisfying that $\Vert x_0-y\Vert>1-\varepsilon=\diam(K_0)-\varepsilon$.
\end{proposition}

\begin{proof}
Take $\varepsilon>0$ and a non-empty relatively weakly open subset of $K_0$. By a density argument, we can find $i\in\mathbb N$ satisfying that $\Vert x_0-g_i\Vert<\varepsilon$. Again by a density argument there exists $g_k\in W$ for certain $k\in\mathbb N$.

As we explained above, by the definition of $K_0$ we have that the sequence $g_k+e_{m_n+k}\in K_0$ for every $n\in\mathbb N$. Since $\bigl(g_k+e_{m_n+k}\bigr)_{n\in \N}\longrightarrow g_k$ weakly, we can find $n\in\mathbb N$ large enough so that $g_k+e_{m_n+k}\in W$ and $m_n+k\notin \supp(g_i)\cup \supp(g_k)$ (this is possible because the previous set is finite). So taking $y=g_k+e_{m_n+k}$, we get $y(m_m+k)=1$ and so
$$\Vert g_i-y\Vert\geq y(m_n+k)-g_i(m_n+k)=1-0=1.$$
As a consequence, $\Vert x_0-y\Vert\geq \Vert g_i-y\Vert-\Vert g_i-x_0\Vert>1-\varepsilon$, and the proof is finished.
\end{proof}

It is time to construct the announced renorming of $C[0,1]$. Take a sequence of non-empty pairwise disjoint open subsets $V_n$ of $[0,1]$ satisfying that $0\notin \bigcup\limits_{n\in\mathbb N} V_n$. By Urysohn lemma, we can find, for every $n\in\mathbb N$, a function $h_n\in S_{C[0,1]}$ with $0\leq h_n\leq 1$ and such that $\supp(h_n)\subseteq V_n$. If we consider $Z:=\cspan\{h_n\colon n\in \N\}$, we get that $Z$ is lattice isometrically isomorphic to $c_0$ (indeed, the mapping $e_n\longmapsto h_n$ is an isometric Banach lattice isomorphism). Consequently, we can consider the set $K_0$ constructed in $Z$, obtaining that $K_0\subseteq B_{C[0,1]}$ is a set of positive functions (because the latter linear isometry preserves the lattice structure) which contains convex combination of slices of arbitrarily small diameter but enjoying the property exhibited in Proposition \ref{prop:examk0superdauga}. Moreover, by the construction of the functions $h_n$, $f(0)=0$ for every $f\in Z$ so, in particular, $f(0)=0$ for every $f\in K_0$.

Now, take $0<\varepsilon<1$ and write
$$
B_\varepsilon:=\cco\left(2\left(K_0-\frac{\mathds{1}}{2}\right)\cup 2\left(-K_0+\frac{\mathds{1}}{2}\right) \cup ((1-\varepsilon)B_{C[0,1]}+\varepsilon B_{\ker(\delta_0)}) \right),$$
where $\mathds{1}$ stands for the constant function $1$ in $C[0,1]$.

Consider $\Vert\cdot\Vert_\varepsilon$ the norm on (the real version of) $C[0,1]$ whose unit ball is $B_\varepsilon$. As we have indicated, the renorming technique follows the scheme of the renorming given in \cite[Theorem 2.4]{blradv} with the difference that we use $B_{\ker(\delta_0)}$ instead of $B_{c_0}$ in the last term because $\ker(\delta_0)$ is a Banach space with the Daugavet property.

We have the following result.

\begin{theorem}\label{theo:superdaunoccwdelta}
The space $(X,\Vert\cdot\Vert_\varepsilon)$ satisfies that:
\begin{enumerate}
    \item Every element of $2(K_0-\frac{\mathds{1}}{2})$ is a super Daugavet point.
    \item For every $\eta>0$ there exists a convex combination of slices $D$ of $B_\varepsilon$ with $D\cap 2(K_0-\frac{\mathds{1}}{2})\neq \emptyset$ and such that $\diam(D)<\eta$.
\end{enumerate}
In particular, there are super Daugavet points which are not ccs$-\Delta$ points.
\end{theorem}

\begin{proof}
(1). Take $a\in K_0$, and let us prove that $2a-\mathds{1}$ is a super Daugavet point. In order to do so, pick a non-empty relatively weakly open subset $W$ of $B_\varepsilon$. Write
$$
A:=2(K_0-\tfrac{{\mathds{1}}}{2}) \ \text{ and } \   B:=(1-\varepsilon)B_{C[0,1]}+\varepsilon B_{\ker(\delta_0)}.
$$
Since  $B_\varepsilon=\cco(A\cup-A\cup B)$ we have that $W$ has non-empty intersection with $\conv(A\cup -A\cup B)$. Now observe that $\frac{A-A}{2}=K_0-K_0\subseteq B_{\ker(\delta_0)}\subseteq B$ so that $\conv(A\cup -A\cup B)=\conv(A\cup B)\cup \conv({-A}\cup B)$ by \cite[Lemma 2.4]{blradv}.
Consequently, either $W\cap \co(A\cup B)$ or $W\cap \co(-A\cup B)$ is non-empty. Let us distinguish by cases.

Assume first that $W\cap \co(A\cup B)$ is non-empty, so find $a'\in K_0$, $f\in B_{C[0,1]}$, $g\in B_{\ker(\delta_0)}$, and $\alpha,\beta\in [0,1]$ with $\alpha+\beta=1$ satisfying that
$$\alpha(2a'-\mathds{1})+\beta((1-\varepsilon)f+\varepsilon g)\in W.$$
Take $\eta>0$. By Proposition \ref{prop:examk0superdauga}, there exists a net $(a_s)\longrightarrow a'$ weakly with $a_s\in K_0$ for every $s$ and satisfying that $\Vert a-a_s\Vert\longrightarrow 1$. Since $(2a_s-\mathds{1})\longrightarrow 2a'-\mathds{1}$ weakly, we can find $s$ large enough so that
$$
\alpha(2a_s-\mathds{1})+\beta((1-\varepsilon)f+\varepsilon g)\in W
$$
and
$$\Vert (2a-\mathds{1})-(2a_s-\mathds{1})\Vert=2\Vert a-a_s\Vert>2-\eta.$$
Observe that $2a-\mathds{1}$ and $2a_s-\mathds{1}$ are functions in $B_{C[0,1]}$ since $a,a_s$ are positive functions of norm at most one.
Since $\Vert (2a-\mathds{1})-(2a_s-\mathds{1})\Vert>2-\eta$, there exists $t_0\in [0,1]$ and $\theta\in\{-1,1\}$ such that $\theta(2a-\mathds{1})(t_0)>1-\eta$ and $\theta(2a_s-\mathds{1})(t_0)<-1+\eta$ (observe that $t_0\neq 0$ since $a(t_0)=a_s(t_0)=0$ by construction). Consequently, the set
$$U:=\{t\in [0,1]\colon \theta(2a-\mathds{1})(t)>1-\eta \,\text{ and } \, \theta(2a_s-1)(t)<-1+\eta\}$$
is a non-empty open subset of $[0,1]$, and we can construct a sequence of non-empty pairwise disjoint open sets $W_n\subseteq U$. Observe that $0\notin \bigcup\nolimits_{n\in\mathbb N} W_n$ since $0\notin U$. Take $p_n\in W_n$ for every $n\in\mathbb N$. We can construct, for every $n\in\mathbb N$, two functions $f_n$ and $g_n$ in the unit ball of $C[0,1]$ satisfying $f_n=f$ and $g_n=g$ in $[0,1]\setminus W_n$ and $f_n(p_n)=g_n(p_n)=-\theta$. Observe that the sequence of functions $(f-f_n)$ have pairwise disjoint supports, so $(f-f_n)\longrightarrow 0$ weakly or, in other words, $(f_n)\longrightarrow f$ weakly. A similar argument shows that $(g_n)\longrightarrow g$ weakly. Notice also that, given $n\in\mathbb N$, since $0\notin W_n$ then $g_n(0)=g(0)=0$, so $(g_n)\subseteq \ker(\delta_0)$. Henceforth $\alpha(2a_s-\mathds{1})+\beta((1-\varepsilon)f_n+\varepsilon g_n)$ is a sequence in $B_\varepsilon$ which converges in $n$ weakly to $\alpha(2a_s-\mathds{1})+\beta((1-\varepsilon)f+\varepsilon g)\in W$. Consequently, we can find $n$ large enough such that $\alpha(2a_s-\mathds{1})+\beta((1-\varepsilon)f_n+\varepsilon g_n)\in W$. Finally, observe that the inclusion $B_\varepsilon\subseteq B_{C[0,1]}$ implies that $\Vert z\Vert\leq \Vert z\Vert_\varepsilon$, so
\begin{align*}
\bigl\Vert (2a-\mathds{1})-\alpha(2a_s-\mathds{1})-\beta((1-\varepsilon)f_n+\varepsilon g_n)\bigr\Vert_\varepsilon & \geq \bigl\Vert (2a-\mathds{1})-\alpha(2a_s-\mathds{1})-\beta((1-\varepsilon)f_n+\varepsilon g_n)\bigr\Vert\\
& \geq
\theta((2a-\mathds{1})-\alpha(2a_s-\mathds{1})-\beta((1-\varepsilon)f_n)(p_n)\\
& =\theta(2a-\mathds{1})(p_n)-\theta\alpha(2a_s-\mathds{1})(p_n) \\ &\quad -\theta\beta((1-\varepsilon)f_n(p_n)+\theta\varepsilon g_n(p_n))\\
& >1-\eta-\alpha(-1+\eta)-\beta (-1)\\
& =1+\alpha+\beta -(1+\alpha)\eta=2-2\eta.
\end{align*}
Since $\eta>0$ was arbitrary this finishes the case $W\cap \co(A\cup B)\neq \emptyset$.

For the case $W\cap \co(-A\cup B)\neq \emptyset$, find $a'\in K_0$, $f\in B_{C[0,1]}$, $g\in B_{\ker(\delta_0)}$, and $\alpha,\beta\in [0,1]$ with $\alpha+\beta=1$ satisfying that
$$
\alpha(-2a'+\mathds{1})+\beta((1-\varepsilon)f+\varepsilon g)\in W.
$$
This case is simpler because $\Vert (2a-\mathds{1})-(-2a'+\mathds{1})\Vert\geq (2a-\mathds{1})-(-2a'+\mathds{1})(0)=2$. Now, an approximation argument for $f_n$ and $g_n$ similar to that of the above case (working on a non-empty open subset of $(0,1)$ in order to get $g_n(0)=0$) finishes this case and, consequently, the proof of (1).

(2). The first part of the proof will be a repetition of the argument of \cite[Theorem 2.4]{blradv}. Fix $\gamma >0$. From \cite[Theorem 1.2]{aor} there exist slices $S_1,\cdots ,S_n$ of $K_0$ such that
$$\diam \left(\frac{1}{n}\sum_{i=1}^nS_i \right)<\frac{1}{4}(1-\varepsilon )\gamma.$$ We
can assume that $S_i=\{x\in K_0\colon x_i^*(x)>1-\widetilde{\delta }\}$
where $0< \widetilde{\delta } <1$, $x_i^*\in C[0,1]^*$ and $\sup x_i^*(K_0)=1$ holds for every $i=1,\ldots ,n$. It is
clear that $$\sup x_i^*\bigl(2(K_0-\frac{\mathds{1}}{2})\bigr)=2(1-x_i^*(\frac{\mathds{1}}{2})),
$$ for all $i=1,\cdots ,n$.
We put $\rho ,\delta>0$ such that $\frac{1}{2}\rho \Vert
x_i^*\Vert +\delta <\widetilde{\delta }$, $2\rho <\varepsilon$,
$\rho \Vert x_i^*\Vert < 4\delta$, and $\frac{(7-2\varepsilon
)\rho }{(1-\varepsilon )}< \gamma $, for all $i=1,\ldots ,n$. We
consider the relatively weakly open set of $B_\varepsilon $ given
by
$$ U_i:=\left\{x\in B_\varepsilon \colon x_i^*(x)>2\left(1-\delta -x_i^*\left(\frac{\mathds{1}}{2}\right)\right)+\frac{1}{2}\rho \Vert x_i^*\Vert , \ x(0)=\delta_0(x)<-1+\rho^2\right\}$$ for every $i=1,\ldots ,n$. It is clear that $\Vert x_i^*\Vert_{\varepsilon}\leq  \Vert x_i^*\Vert $ for every $i=1,\ldots , n$ and $\Vert \delta_0\Vert_{\varepsilon}=\Vert\delta_0 \Vert =1$.

Since $\rho \Vert x_i^*\Vert < 4\delta$, we have that
$2(1-x_i^*(\frac{\mathds{1}}{2}))>2(1-\delta -x_i^*(\frac{{\bf
1}}{2}))+\frac{1}{2}\rho \Vert x_i^*\Vert$. Now, we have that
$\sup x_i^*(2(K_0-\frac{\mathds{1}}{2}))=2(1-x_i^*(\frac{{\bf
1}}{2}))$, then there exists $x\in K_0$ such that
$$x_i^*(2(x-\frac{\mathds{1}}{2}))>2(1-\delta -x_i^*(\frac{{\bf
1}}{2}))+\frac{1}{2}\rho \Vert x_i^*\Vert\ \text{ and } \ \delta_0( 2(
x-\frac{1}{2}))=-1<-1+\rho^2.$$ This implies that $U_i  \neq
\emptyset$ for every $i=1,\ldots ,n$. In order to estimate the
diameter of $\frac{1}{n}\sum_{i=1}^nU_i $, it is enough to compute
the diameter of
$$\frac{1}{n}\sum_{i=1}^n U_i \cap \conv\left(2\left(K_0- \frac{\mathds{1}}{2}\right)\cup
{-2}\left(K_0- \frac{\mathds{1}}{2}\right)\cup [(1-\varepsilon) B_X+\varepsilon
B_{\ker(\delta_0)}]\right).$$ Since $2(K_0- \frac{\mathds{1}}{2})$ and $(1-\varepsilon
)B_{C[0,1]}+\varepsilon B_{\ker(\delta_0)}$ are convex subsets of $B_\varepsilon$,
given $x\in B_\varepsilon$, we can assume that $x=\lambda _1 2(a-
\frac{\mathds{1}}{2})+\lambda _2 2 (-b+ \frac{\mathds{1}}{2})+\lambda _3
[(1-\varepsilon )x_0+\varepsilon y_0]$, where $\lambda _i\in
[0,1]$ with $\sum_{i=1}^3\lambda_i=1$ and $a,b\in K_0$, $x_0\in
B_{C[0,1]}$, and $y_0\in B_{\ker(\delta_0)}$.

So given $x,y\in \frac{1}{n}\sum_{i=1}^nU_i$, for $i=1,\cdots , n$,
there exist $a_i,a'_i,b_i,b'_i\in K_0$, $\lambda _{(i,j)},\lambda
'_{(i,j)}\in [0,1]$ with $j=1,2,3$ and, $x_i, x_i'\in B_{C[0,1]}$, and
$y_i,y_i'\in B_{Ker(\delta_0)}$, such that
$$u_i:=2\lambda_{(i,1)} \left(a_i-\frac{{\bf
1}}{2}\right)+2\lambda_{(i,2)}\left(-b_i+\frac{{\bf
1}}{2}\right)+\lambda_{(i,3)}[(1-\varepsilon )x_i+\varepsilon y_i]$$
$$u_i':=2\lambda_{(i,1)}' \left(a_i'-\frac{{\bf
1}}{2}\right)+2\lambda_{(i,2)}'\left(-b_i'+\frac{{\bf
1}}{2}\right)+\lambda_{(i,3)}'[
(1-\varepsilon )x_i'+\varepsilon y_i']$$
belong to $U_i$ for every $i\in\{1,\ldots, n\}$, and such that
$$x=\frac{1}{n}\sum_{i=1}^n u_i\ \text{ and }\ y=\frac{1}{n}\sum_{i=1}^n u_i'.$$
For $i\in\{1,\ldots,n\}$ we have that $u_i\in U_i$ so $$\delta_0(u_i)=\delta_0\left(2\lambda_{(i,1)} \left(a_i-\frac{{\bf
1}}{2}\right)+2\lambda_{(i,2)}\left(-b_i+\frac{{\bf
1}}{2}\right)+\lambda_{(i,3)}[(1-\varepsilon )x_i+\varepsilon
y_i]\right)<-1+\rho^2.$$

Observe that, by construction, $$\delta_0\left(a_i-\frac{{\bf
1}}{2}\right)=-\frac{1}{2}, \delta_0\left(-b_i+\frac{{\bf
1}}{2}\right)=\frac{1}{2}\ \text{ and }\ \delta_0((1-\varepsilon )x_i+\varepsilon
y_i)=\delta_0((1-\varepsilon) x_i)\geq -(1-\varepsilon).$$ This implies that
$$2\lambda _{(i,2)}+\lambda _{(i,3)}\varepsilon -1= -\lambda _{(i,1)}+\lambda _{(i,2)}
-\lambda _{(i,3)}(1-\varepsilon )<-1+\rho ^2.$$ Since
$2\rho<\varepsilon$, we deduce that
$\lambda_{(i,2)}+\lambda_{(i,3)}<\frac{1}{2}\rho$. As a
consequence we get that
\begin{equation}\label{landaeje2}
\lambda_{(i,1)}>1-\frac{1}{2}\rho ,
\end{equation}
and, similarly, we get that
\begin{equation}\label{landaprimaeje2}
\lambda'_{(i,1)}>1-\frac{1}{2}\rho ,
\end{equation}
for every $i=1,\ldots ,n$. Now, the previous inequalities imply that
\begin{align*}
\Vert x-y\Vert_\varepsilon&  \leq \frac{1}{n}\left\Vert \sum_{i=1}^n  2\lambda_{(i,1)} \left(a_i-\frac{\mathds{1}}{2}\right) -2\lambda_{(i,1)}' \left(a_i'-\frac{\mathds{1}}{2}\right)\right\Vert
_\varepsilon \\
& \qquad + \frac{1}{n}\sum_{i=1}^n \left\Vert
2\lambda_{(i,2)}\left(-b_i+\frac{\mathds{1}}{2}\right)\right\Vert _\varepsilon +
\frac{1}{n}\sum_{i=1}^n \left\Vert 2\lambda_{(i,2)}'\left(-b_i'+ \frac{\mathds{1}}{2}\right)\right\Vert _\varepsilon\\
& \qquad \qquad +\frac{1}{n}\sum_{i=1}^n \Vert
\lambda_{(i,3)}[(1-\varepsilon )x_i+\varepsilon
y_i]\Vert_\varepsilon + \frac{1}{n}\sum_{i=1}^n \Vert
\lambda_{(i,3)}'[(1-\varepsilon )x_i'+\varepsilon y_i']\Vert
_\varepsilon\\
& \leq \frac{1}{n}\left\Vert \sum_{i=1}^n
2\lambda_{(i,1)} \left(a_i-\frac{\mathds{1}}{2}\right)-2\lambda_{(i,1)}'
\left(a_i'-\frac{\mathds{1}}{2}\right)\right\Vert _\varepsilon\\
& \qquad +\frac{1}{n}\sum_{i=1}^n
\left(\lambda_{(i,2)}+\lambda_{(i,3)}
\right)+\frac{1}{n}\sum_{i=1}^n
\left(\lambda_{(i,2)}'+
\lambda_{(i,3)}'\right)
\intertext{and, by using \eqref{landaeje2},\eqref{landaprimaeje2},}
& \leq \frac{1}{n}\left\Vert \sum_{i=1}^n
 2\lambda_{(i,1)} \left(a_i-
 \frac{\mathds{1}}{2}\right)-2
 \lambda_{(i,1)}'
\left(a_i'-\frac{\mathds{1}}{2}\right)\right\Vert _\varepsilon +\rho\\
& \leq \frac{2}{n}\left\Vert \sum_{i=1}^n
 \lambda_{(i,1)} a_i-\lambda_{(i,1)}' a_i'\right\Vert _\varepsilon +
\frac{1}{n}\sum_{i=1}^n \vert
\lambda_{(i,1)}-\lambda_{(i,1)}'\vert \Vert {\mathds{1}}\Vert_\varepsilon +\rho\\
& \leq \frac{2}{n}\left\Vert \sum_{i=1}^n
\lambda_{(i,1)} a_i-\lambda_{(i,1)}' a_i'\right\Vert _\varepsilon +
\frac{(3-2\varepsilon)}{2(1-\varepsilon)}\rho .
\end{align*}
Now,
\begin{align*}
\left\Vert \sum_{i=1}^n \lambda_{(i,1)}
a_i-\lambda_{(i,1)}'a'_i\right
\Vert_\varepsilon\\
\leq \left\Vert
\sum_{i=1}^n (\lambda_{(i,1)}-1) a_i\right\Vert_\varepsilon & +\left\Vert
\sum_{i=1}^n a_i-a'_i\right\Vert_\varepsilon +\left\Vert \sum_{i=1}^n
(\lambda'_{(i,1)}-1) a'_i\right\Vert_\varepsilon\\
\leq \frac{1}{1-\varepsilon}\left\Vert \sum_{i=1}^n a_i-a'_i\right\Vert & +\sum_{i=1}^n
\frac{1}{1-\varepsilon}\vert\lambda_{(i,1)}-1\vert \Vert a_i\Vert
+\sum_{i=1}^n \frac{1}{1-\varepsilon}\vert\lambda_{(i,1)}'-1\vert
\Vert a_i'\Vert\\
\leq \frac{1}{1-\varepsilon}\left\Vert \sum_{i=1}^n a_i-a'_i\right\Vert & +\frac{1}{1-\varepsilon}n\rho.
\end{align*}
(In the previous estimate observe that $\Vert a_i\Vert\leq 1$ and $\Vert a_i'\Vert\leq 1$ since $a_i,a_i'\in K_0\subseteq B_{\ker(\delta_0)}\subseteq B_\varepsilon$).
Hence,
\begin{equation}\label{acotado}
\Vert x-y\Vert_\varepsilon \leq \frac{2}{1-\varepsilon}\left\Vert
\frac{1}{n} \sum_{i=1}^n a_i-a'_i\right\Vert
+\frac{(7-2\varepsilon)}{2(1-\varepsilon )}\rho .
\end{equation}
Now, in order to prove that the previous norm is small we will prove that both elements $\frac{1}{n}\sum_{i=1}^n a_i, \frac{1}{n}\sum_{i=1}^n a'_i$ are elements of $\frac{1}{n}\sum_{i=1}^n S_i$, which has small diameter. To this end, note that
\begin{align*}
x_i^*\left(2\lambda_{(i,1)}
\left(a_i-\frac{\mathds{1}}{2}\right) +2\lambda_{(i,2)}\left(-b_i+\frac{\mathds{1}}{2}\right)+\lambda_{(i,3)}[(1-\varepsilon )x_i+\varepsilon
y_i]\right)\\
\qquad >2\left(1-\delta -x_i^*\left(\frac{\mathds{1}}{2}\right)\right)+\frac{\rho}{2}
\Vert x_i^*\Vert ,
\end{align*}
for every $i\in\{1,\ldots,n\}$. Then,
\begin{align*}
x_i^*\left(2\lambda_{(i,1)}
\left(a_i-\frac{\mathds{1}}{2}\right)\right) +\frac{1}{2}\rho \Vert x_i^*\Vert &
\geq x_i^*\left(2\lambda_{(i,1)}
\left(a_i-\frac{\mathds{1}}{2}\right)\right) +\lambda_{(i,2)}\Vert x_i^*\Vert
 +\lambda_{(i,3)}\Vert x_i^*\Vert\\
& \geq x_i^*\left(2\lambda_{(i,1)}
\left(a_i-\frac{\mathds{1}}{2}\right)\right) +\lambda_{(i,2)}\Vert x_i^*\Vert_\varepsilon
 +\lambda_{(i,3)}\Vert x_i^*\Vert_\varepsilon\\
&\geq x_i^*\left(2\lambda_{(i,1)} \left(a_i-\frac{\mathds{1}}{2}\right)+ 2\lambda_{(i,2)}\left(-b_i+\frac{\mathds{1}}{2}\right)+\lambda_{(i,3)}[(1-\varepsilon )x_i+\varepsilon y_i]\right).
\end{align*}
We have that $$x_i^*\left(2\lambda_{(i,1)} \left(a_i-\frac{\mathds{1}}{2}\right)\right) >2\left(1-\delta -x_i^*\left(\frac{\mathds{1}}{2}\right)\right),$$ and hence
$$x_i^*(\lambda_{(i,1)}a_i)>1-\delta -(1-\lambda_{(i,1)})x_i^*\left(\frac{\mathds{1}}{2}\right)\geq  1-\delta -\frac{1}{2}\rho \Vert x_i^*\Vert.
$$
We recall that $\delta +\frac{1}{2}\rho \Vert x_i^*\Vert <
\widetilde{\delta } $, so
$x_i^*(\lambda_{(i,1)}a_i)>1-\widetilde{\delta }$. It follows that
$x_i^*(a_i)>1-\widetilde{\delta }$. Then, $a_i\in K_0\cap S_i$ and,
similarly, we get that $a'_i\in K_0\cap S_i$, for every $i=1,\ldots
,n$. Therefore, $$\frac{1}{n}\sum_{i=1}^n a_i,\ \frac{1}{n}\sum_{i=1}^n
a'_i\in \frac{1}{n}\sum_{i=1}^n S_i.$$ Since the diameter of
$\frac{1}{n}\sum_{i=1}^n S_i$ is less than
$\frac{1}{4}(1-\varepsilon )\gamma $, we deduce that
$\frac{1}{n}\Vert \sum_{i=1}^n a_i-a'_i\Vert
<\frac{1}{4}(1-\varepsilon )\gamma $. Finally, we conclude from
\eqref{acotado} and the above estimate that $\Vert
x-y\Vert_\varepsilon \leq \gamma$. Hence, the set
$C:=\frac{1}{n}\sum_{i=1}^n U_i$ has diameter at most $\gamma$ for
the norm $\Vert\cdot\Vert_{\varepsilon}$.

Now, Bourgain's lemma (see Lemma~\ref{lem:bourgain_lemma}) ensures the existence of a convex combination of slices $\sum_{j=1}^{p_i} \alpha_{ij} T_{ij}\subseteq U_i$ for every $1\leq i\leq n$. Using this fact, we will find a convex combination of slices of $B$ of diameter smaller than $\gamma+\frac{4\rho^2}{(1-\frac{\rho^2}{\varepsilon})\varepsilon}$ and such that every slice contains points of $2(K_0-\frac{\mathds{1}}{2})$. Since $\rho$ and $\gamma$ can be taken as small as we wish, we will be done.
In order to do so, fix $1\leq i\leq n$ and define
$$A_i:=\left\{j\in\{1,\ldots, p_i\}\colon  T_{ij}\cap \left(2K_0-\frac{\mathds{1}}{2}\right)=\emptyset\right\};\quad B_i:=\{1,\ldots, p_i\}\setminus A_i.$$
Given $x_{ij}\in T_{ij}$ we have that, for $j\in A_i$, that $\delta_0(x_{ij})\geq -1+\varepsilon$ by the definition of the unit ball $B_\varepsilon$. Since $\sum_{j=1}^{p_i}\alpha_{ij} x_{ij}\in \sum_{j=1}^{p_i} \alpha_{ij} T_{ij}\subseteq U_i$ we derive $-1+\rho^2> \delta_0\left( \sum_{j=1}^{p_i}\alpha_{ij}x_{ij}\right)$. Hence
$$-1+\rho^2>\sum_{j\in A_i}\alpha_{ij}\delta_0(x_{ij})+\sum_{i\in B_i}\alpha_{ij}\delta_0(x_{ij})\geq (-1+\varepsilon)\sum_{j\in A_i}\alpha_{ij}-\sum_{j\in B_i}\alpha_{ij}=-1+\varepsilon\sum_{j\in A_i}\alpha_{ij}.$$
From the above inequality we infer that $\sum_{j\in A_i}\alpha_{ij}<\frac{\rho^2}{\varepsilon}$ holds for every $1\leq i\leq n$. Now, we set $\Lambda_i:=\sum_{j\in B_i}\lambda_{ij}$, which belongs to the interval $[1-\frac{\rho^2}{\varepsilon},1]$ for $1\leq i\leq n$ and set
$$
D:=\frac{1}{n}\sum_{i=1}^n\sum_{j\in B_i} \frac{\alpha_{ij}}{\Lambda_I}T_{ij}.
$$
Observe that $D$ is a convex combination of slices of $B_\varepsilon$ since every $T_{ij}$ is a slice of $B_\varepsilon$ and since
$$\frac{1}{n}\sum_{i=1}^n\sum_{j\in B_I}\frac{\alpha_{ij}}{\Lambda_i}\alpha_{ij}=1.$$
We claim that $D\subseteq C+\frac{2}{1-\frac{\rho^2}{\varepsilon}}\frac{\rho^2}{\varepsilon} B_\varepsilon$. This is enough to finish the proof because the above condition implies that
$$
\diam(D)\leq \diam(C)+\frac{4}{1-\frac{\rho^2}{\varepsilon}}\frac{\rho^2}{\varepsilon}\leq \gamma+\frac{4}{1-\frac{\rho^2}{\varepsilon}}\frac{\rho^2}{\varepsilon}.
$$
So let us prove the above inclusion. Take $z:=\frac{1}{n}\sum_{i=1}^n \sum_{j\in B_i}\frac{\alpha_{ij}}{\Lambda_i}x_{ij}\in D$ for certain $x_{ij}\in T_{ij}$. Write $z':=\frac{1}{n}\sum_{i=1}^n \sum_{j\in B_i} \alpha_{ij}x_{ij}$. Then
$$\vert z-z'\vert\leq \frac{1}{n}\sum_{i=1}^n \sum_{j\in B_{ij}} \left\vert 1-\frac{1}{\Lambda_i}\right\vert \alpha_{ij}\vert x_{ij}\vert<\frac{1}{1-\frac{\rho^2}{\varepsilon}}\frac{\rho^2}{\varepsilon}.$$
On the other hand, for $1\leq i\leq n$ and $j\in A_i$ take $x_{ij}\in T_{ij}$. Define
$$
z'':=\frac{1}{n}\sum_{i=1}^n \sum_{j=1}^{p_i}\alpha_{ij}x_{ij}\in\frac{1}{n}\sum_{i=1}^n \sum_{j=1}^{p_i}\alpha_{ij}T_{ij}\subseteq \frac{1}{n}\sum_{i=1}^n U_i=C.
$$
Moreover, we have
$$\vert z'-z''\vert\leq \frac{1}{n}\sum_{i=1}^n\sum_{j\in A_i}\alpha_{ij}<\frac{\rho^2}{\varepsilon}.$$
Consequently $z=z''+(z-z'')\in C+\frac{2}{1-\frac{\rho^2}{\varepsilon}}\frac{\rho^2}{\varepsilon} B_\varepsilon$ since
\begin{equation*}
\vert z-z''\vert\leq \vert z-z'\vert+\vert z'-z''\vert<\frac{1}{1-\frac{\rho^2}{\varepsilon}} \frac{\rho^2}{\varepsilon}+\frac{\rho^2}{\varepsilon} <\frac{2}{1-\frac{\rho^2}{\varepsilon}}\frac{\rho^2}{\varepsilon}.\qedhere
\end{equation*}
\end{proof}

\subsection{A summary of relations between the properties}\label{subsection:schemeofcounterexamples}

Figure~\ref{figure02} below is an scheme which complements Figure~\ref{figure01} with the counterexamples following from known results and from the results in this section.

\begin{figure}[hbt!]
    \begin{tikzcd}
\fbox{\text{ccs Daugavet}}
\arrow[dd, Rightarrow, shift left=1ex]
\arrow[r, Rightarrow] &
\fbox{\text{ccs $\Delta$}}
\arrow[ddr, Rightarrow,shift right=-0.5ex]
\arrow[dd, Rightarrow,shift left=1.5ex, "\text{\large ?}" label]
\arrow[dddd, Rightarrow, "/" marking, bend left=90,looseness=2,swap,"\text{(d)}"']
&  & & & \\
& & & & \\
\fbox{\text{super Daugavet}}
\arrow[uur, Rightarrow,"/" marking,"\text{\!\! (a)}"']
\arrow[r, Rightarrow]
\arrow[ddr, Rightarrow]& \fbox{\text{super $\Delta$}}
\arrow[uu, Rightarrow, shift left=1ex, "/" marking,swap,"\text{(b)\  }"']
\arrow[dd, Rightarrow,shift right=1.5ex, "/" marking,"\text{(e) }"']
\arrow[r, Rightarrow,yshift=-0.5ex] & \fbox{\text{$\Delta$}}
\arrow[l, Rightarrow, "/" marking, yshift=1ex,"\text{\!\!\!\!\!\!(c)}"']
 & \\
& & &  &  \\
& \fbox{\text{Daugavet}}
\arrow[uu, Rightarrow, shift right=1ex, "/" marking,"\text{ (f)}"']
\arrow[uur, Rightarrow]
& &  & &
\end{tikzcd}
\caption{Scheme of all relations between the diametral notions}
\label{figure02}
\end{figure}
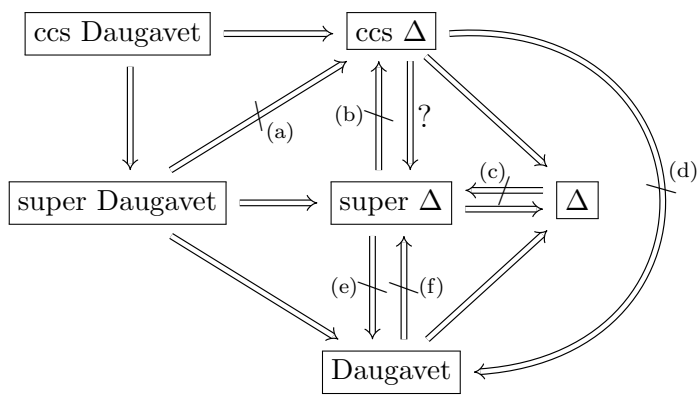

Let us list the corresponding counterexamples.
\begin{enumerate}
  \item[(a)] The example in Subsection~\ref{subsect:superdauganotccsdelta}.
  \item[(b)] The example in Subsection~\ref{subsect:superDeltastronglyregular} negates this implication in the strongest possible way.
  \item[(c)] Any of the elements in $D_B$ in Paragraph~\ref{subsubsec:oneunconditionalbasis}. They also show directly that $\Delta$-points are not necessarily ccs $\Delta$-points. 
  \item[(d)] Every element of the unit sphere of the space $X$ given in Paragraph~\ref{subsubsect:MLUR} is ccs $\Delta$-point but not Daugavet point. Another example is the molecule $m_{0,q}$ of Example~\ref{example:Lipfree-ANPP}.
  \item[(e)] In $X=C[0,1]\oplus_2 C[0,1]$, every element in the unit sphere is super $\Delta$-point (Proposition~\ref{prop:super_Delta_abs_sums}); but $X$ contains no Daugavet point (Proposition~\ref{prop:absolutesums-Daugavetpoints}). Also, every element of the unit sphere of the space $X$ given in Paragraph~\ref{subsubsect:MLUR} is super $\Delta$-point but not Daugavet point.
  \item[(f)] Any of the elements in $D_B$ in Paragraph~\ref{subsubsec:oneunconditionalbasis}. 
\end{enumerate}


\section{Diametral-properties for elements of the open unit ball}\label{sec:inside_of_the_unit_ball}

As mentioned in Section \ref{sec:notation_and_prelimiary_results}, the DSD2P is equivalent to the Daugavet property by \cite{kadets20}, but the ccs $\Delta$-points on the unit sphere of a Banach space do not characterize the DSD2P, but the restricted DSD2P, which is not equivalent to the Daugavet property (see Paragraph~\ref{subsubsect:MLUR}). Actually, the elements in the open unit ball play a decisive role in the proof in \cite{kadets20} of the equivalence between the DSD2P and the Daugavet property. Our objective here is to introduce and study the diametral notions for interior points, providing interesting applications, and to investigate the behavior of Daugavet- and $\Delta$-elements on rays in the unit ball of a given Banach space.

The definition of the Daugavet notions for elements in the open unit ball is the natural extension of the definitions for elements of norm one given in Definition~\ref{def:defiDauganot}.

\begin{defn}\label{defn:unit_ball_Daugavet} Let $X$ be a Banach space and let $x\in B_X$. We say that
\begin{enumerate}
\item $x$ is a \emph{Daugavet point} if $\sup_{y\in S}\norm{x-y}=\norm{x}+1$ for every slice $S$ of $B_X$,
\item $x$ is a \emph{super Daugavet point} if $\sup_{y\in V}\norm{x-y}=\norm{x}+1$ for every non-empty relatively weakly open subset $V$ of $B_X$,
\item $x$ is a \emph{ccs Daugavet point} if $\sup_{y\in C}\norm{x-y}=\norm{x}+1$ for every ccs $C$ of $B_X$.
\end{enumerate}
\end{defn}

It turns out that the existence of a non-zero Daugavet kind element actually forces the whole ray to which it belongs to be composed of similar elements.

\begin{prop}\label{prop:Daugavet_rays}
Let $X$ be a Banach space, and let $x\in S_X$. The following assertions are equivalent.
\begin{enumerate}
\item $x$ is a Daugavet- (resp.\ super Daugavet-, resp.\ ccs Daugavet-) point.
\item $rx$ is a Daugavet- (resp.\ super Daugavet-, resp.\ ccs Daugavet-) point for every $r\in [0,1]$.
\item $rx$ is a Daugavet- (resp.\ super Daugavet- , resp.\ ccs Daugavet-) point for some $r\in (0,1)$.
\end{enumerate}
\end{prop}

Let us recall the following elementary but very useful result from \cite{kadets20} due to Kadets.

\begin{lem}[\mbox{\textrm{\cite[Lemma 2.2]{kadets20}}}] \label{lemma:Kadets-codirected}
Let $X$ be a normed space. If $x,y\in X$ and $\eps>0$ satisfies that $$\norm{x+y}> \norm{x}+\norm{y}-\eps,$$ then for every $a,b>0$, it is satisfied that
$$\norm{ax+by}> a\norm{x}+b\norm{y}-\max\{a,b\}\eps.$$
\end{lem}

\begin{proof}[Proof of Proposition~\ref{prop:Daugavet_rays}]
We will only do the proof for Daugavet points, being the other cases completely analogous. So let us first assume that $x$ is a Daugavet point. Take $r\in [0,1]$, $\eps>0$, and $S$ a slice of $B_X$. Then, there exists $y\in S$ such that $\norm{x-y}>2-\eps$. In particular, $\norm{y}>1-\eps$. As $\|y\|\leq 1$,
$$
\|x-y\|>2-\eps\geq \|x\|+\|y\|-\eps.
$$
It follows from Lemma~\ref{lemma:Kadets-codirected} that
$$
\|rx-y\|>\|rx\|+\|y\|-\eps>\|rx\|+1-2\eps.
$$
Hence, $rx$ is also a Daugavet point.

Now, let us assume that $rx$ is a Daugavet point for some $r\in (0,1)$. Again take $\eps>0$ and $S$ slice of $B_X$, and pick $y\in S$ such that
$\|rx-y\|>\|rx\|+1-r\eps$.
In particular, $\|y\|>1-r\eps$. As $\|y\|\leq 1$, we have
$$
\|rx-y\|>\|rx\|+\|y\|-r\eps.
$$
Hence, by Lemma~\ref{lemma:Kadets-codirected}, we get that
$$
\|x-y\|>\|x\|+\|y\|-\eps>2-(1+r)\eps
$$
and so $x$ is a Daugavet point.
\end{proof}

As mentioned in the discussion preceding Proposition \ref{prop:ccs_Daugavet_SR}, the presence of a ccs Daugavet point in a given Banach space forces the space to satisfy the SD2P. This can now be viewed as a consequence of the previous proposition and the following immediate reformulation of \cite[Theorem 3.1]{lmr19}.

\begin{prop}[\mbox{\textrm{\cite[Theorem 3.1]{lmr19}}}]\label{prop:theorem31}
Let $X$ be a Banach space. Then, $X$ has the SD2P if and only if $0$ is a ccs Daugavet point.
\end{prop}

Note that $c_0$ has the SD2P but has no non-zero Daugavet points (use Proposition~\ref{proposition:delta-superDaugavet-L1predual-old}, for instance).

Compare Proposition~\ref{prop:theorem31} with the following obvious remark.

\begin{remark}\label{remark:0Daugavetpoint}
A Banach space $X$ is infinite-dimensional if and only if $0$ is a super Daugavet point.
\end{remark}

We also mention that $0$ is always a Daugavet point (in finite or infinite dimension), as every slice of the unit ball has to intersect the unit sphere.

Let us also point out that \cite[Theorem 3.1]{lmr19} admits the following scaled version.

\begin{prop}\label{prop:scaledversionthr31lmr19}
Let $X$ be a Banach space, and let $r\in(0,1]$. Then, the following assertions are equivalent.
\begin{enumerate}
    \item Every ccs of $B_X$ has diameter greater than or equal to $2r$.
    \item $\sup\{\|x\|\colon x\in C\}\geq r$ for every ccs $C$ of $B_X$.
    \item $\sup\{\|x\|\colon x\in D\}\geq r$ for every symmetric ccs $D$ of $B_X$  (so containing $0$).
\end{enumerate}
\end{prop}

\begin{proof}
Suppose (1) holds. Then, for any given ccs $C$ of $B_X$, and for any fixed $\eps>0$, there exists $x,y\in C$ such that $\norm{x-y}>2r-2\eps$. In particular, it follows that $\norm{x}>r-\eps$ or $\norm{y}>r-\eps$, giving (2). (2)$\Rightarrow$(3) is immediate. Suppose that (1) fails, that is, that  there exists a ccs $C$ of $B_X$ and $\eps>0$ such that $\diam(C)\leq 2r-2\eps$. We consider the ccs $D$ of $B_X$ given by $D:=\frac{1}{2}(C-C)$. Then $D$ is symmetric, and for every $u:=\frac{x-y}{2}$ and $u':=\frac{x'-y'}{2}$ in $D$ we have $\norm{u-u'}=\norm{\frac{x+y'}{2}-\frac{x'+y}{2}}$. Now, $x,x',y,y'$ belong to $C$, and $C$ is convex, so $\frac{x+y'}{2}$ and $\frac{x'+y}{2}$ do also belong to $C$, so $\norm{u-u'}\leq 2r-2\eps$. Being $D$ symmetric, it implies that $D\subset (r-\eps)B_X$, hence (3) fails.
\end{proof}

The following is a nice consequence of the proposition above outside the diametral notions.

\begin{cor}
Let $X$ be a Banach space. Then, $B_X$ contains ccs of arbitrarily small diameter if and only if $0$ is a strongly regular point of $B_X$.
\end{cor}

Now let us consider the $\Delta$ notions for points of the open unit ball which are just the adaptation of the notions given in Definition~\ref{def:defiDeltanot}.

\begin{defn}\label{defn:unit_ball_Delta}
Let $X$ be a Banach space and let $x\in B_X$. We say that
\begin{enumerate}
\item $x$ is a \emph{$\Delta$-point} if  $\sup_{y\in S}\norm{x-y}=\norm{x}+1$ for every slice $S$ of $B_X$ containing $x$,
\item $x$ is a \emph{super $\Delta$-point} if  $\sup_{y\in V}\norm{x-y}=\norm{x}+1$ for every non-empty relatively weakly open subset $V$ of $B_X$ containing $x$,
\item $x$ is a \emph{ccs $\Delta$-point} if  $\sup_{y\in C}\norm{x-y}=\norm{x}+1$ for every slice ccs $C$ of $B_X$ containing $x$.
\end{enumerate}
\end{defn}

With this definitions in hands, we may get an improvement of  Proposition~\ref{prop:theorem31} from Proposition~\ref{prop:scaledversionthr31lmr19}.

\begin{cor}
Let $X$ be a Banach space. Then, $X$ has the SD2P if and only if $0$ is ccs $\Delta$-point.
\end{cor}

Compare the previous corollary with the following obvious remark which is analogous to Remark~\ref{remark:0Daugavetpoint}.

\begin{remark}
A Banach space $X$ is infinite-dimensional if and only if $0$ is a super $\Delta$-point.
\end{remark}

Observe that the definition of ccs $\Delta$-points for elements in $B_X$ gives a localization of the DSD2P, that is, $X$ has the DSD2P (and hence the Daugavet property \cite{kadets20}) if and only if all the elements of $B_X$ are ccs $\Delta$-points. Recall that the DSD2P is not equivalent to the restricted DSD2P (meaning that all points in $S_X$ are ccs $\Delta$-points), see Paragraph~\ref{subsubsect:MLUR}.

The following result is a localization of Kadets' theorem \cite{kadets20} on the equivalence of the DSD2P and the DPr.

\begin{thm}\label{thm:Daugavet=DSD2P}
Let $X$ be a Banach space and let $x\in S_X$. If $rx$ is a ccs $\Delta$-point for every $r\in (0,1)$, then $x$ is a ccs Daugavet point. Moreover, it is enough that $\inf\{r\in (0,1)\colon rx\text{ is a ccs $\Delta$-point}\}=0$.
\end{thm}

\begin{proof}
Fix a ccs $C$ of $B_X$ and $\eps>0$. Since $\tilde{C}:=\frac{1}{2}(C-C)$ is also a ccs of $B_X$ and since $0\in \tilde{C}$ is a norm interior point of $\tilde{C}$ by \cite[Proposition 2.1]{lmr19}, we have that $rx$ belongs to $\tilde{C}$ for every $r\in (0,\delta)$ for some $\delta>0$. By hypothesis, there is $r>0$ such that $rx$ is a ccs $\Delta$-point and $rx\in \tilde{C}$. So there exists $y\in \tilde{C}$ such that $\norm{rx-y}>r+1-r\eps$.  Then if we write $y:=y_1-y_2$ with $y_1,y_2\in C$, we have $\norm{rx-y_1}>r+1-r\eps$ or $\norm{rx-y_2}>r+1-r\eps$ by the triangle inequality; in particular, $\norm{y_1}>1-r\eps$ and $\norm{y_2}>1-r\eps$. In both cases, we have that there is $y\in C$ such that $\norm{rx-y}>r\|x\|+\|y\|-r\eps$ and it follows from Lemma~\ref{lemma:Kadets-codirected} that
\begin{equation*}
\norm{x-y}>\|x\|+\|y\|-\eps >2-(1+r)\eps>2-2\eps. \qedhere
\end{equation*}
\end{proof}

At this point, it is natural to ask whether an equivalent formulation of Proposition \ref{prop:Daugavet_rays} is valid for some of the various $\Delta$-notions. For ccs $\Delta$-points, the answer is negative as follows from Theorem~\ref{thm:Daugavet=DSD2P} and, for instance, the example in Paragraph~\ref{subsubsect:MLUR}. Another, maybe simpler, example showing that is the following one.

\begin{expl}
Let us assume that a positive measure $\mu$ admits an atom of finite measure and also has a non-empty non-atomic part. Then, the real space $L_1(\mu)$ contains no ccs Daugavet point by Proposition \ref{prop:L_1_ccs_Daugavet}. However, it contains elements in the unit sphere which are ccs $\Delta$-points and super Daugavet points by Theorem~\ref{thrm:L_1_ccs_Delta} and Proposition~\ref{proposition:realcaseL1mu-ccs-delta}. In particular, as a consequence of Theorem~\ref{thm:Daugavet=DSD2P} and of Proposition~\ref{prop:Daugavet_rays}, there must exist $f$ in the unit sphere which is  a ccs $\Delta$-point and $t\in(0,1)$ such that $tf$ is not a ccs $\Delta$-point but it is a super $\Delta$-point.
\end{expl}

For $\Delta$-points, we have the following result.

\begin{prop}
Let $X$ be a Banach space, and let $x\in S_X$. If $x$ is a $\Delta$-point, then $rx$ is a $\Delta$-point for every $r\in(0,1)$.
\end{prop}

\begin{proof}
Let us assume that $x$ is a $\Delta$-point and let us fix $r\in(0,1)$.  Take $\eps>0$ and a slice $S$ of $B_X$ containing $rx$. Now, either $x$ belongs to $S$ or $-x$ belongs to $S$. In the first case, we can find $y\in S$ such that $\norm{x-y}\geq 2-\eps$, and using Lemma~\ref{lemma:Kadets-codirected}, we get $\norm{rx-y}\geq \norm{rx}+1-2\eps$. Else, $\norm{rx-(-x)}=r+1=\norm{rx}+1$ and we are done.
\end{proof}

For super $\Delta$-points, it is currently quite obscure whether they behave like $\Delta$-points up on rays.


\section{Kuratowski measure and large diameters}\label{section:Kuratowski_measure}

Let $M$ be a metric space. The \emph{Kuratowski measure of non-compactness} $\alpha(A)$ of a non-empty bounded subset $A$ of $M$ is defined as the infimum of all real numbers $\eps>0$ such that $A$ can be covered by a finite number of subsets of $M$ of diameter smaller than or equal to $\eps$.

From the definition, we clearly have $\alpha(A)=0$ if and only if $A$ is \emph{totally bounded} (a.k.a.\ \emph{precompact}). It follows that every complete subset $A$ of $M$ with $\alpha$-measure $0$ is compact, and in particular, if $M$ is a complete metric space, that $\alpha(A)=0$ if and only if $\overline{A}$ is compact, where $\overline{A}$ stands for the closure of the set $A$. The $\alpha$-measure can be thus seen as a way to measure how far a given (non-empty) bounded and closed subset of $M$ is from being a compact space. It was introduced by C. Kuratowski in \cite{Kuratowski30} in order to provide a generalization of the famous intersection theorem from Cantor. A general theory on measures of non-compactness was later developed, and it turned out to provide important results in metric fixed point theory,  and in particular to have applications in functional equations or optimal control. We refer e.g. to \cite{noncompactnessbook} for an introduction to the topic and for more precise applications. 

Observe that $A\subseteq B$ implies $\alpha(A)\leq \alpha(B)$, and that $\alpha(\overline{A})=\alpha(A)$. Also note that $\alpha(A\cup B)=\max\{\alpha(A),\alpha(B)\}$ for every non-empty bounded subsets $A,B$ of $M$. Furthermore, if $M=X$ is a normed space, then $\alpha$ is known to enjoy additional useful properties: it is symmetric, translation invariant, positively homogeneous, sub-additive, and satisfies $\alpha(\conv{A})=\alpha(A)$. The $\alpha$-measure has proved to be a powerful tool for the study of the geometry of Banach spaces and we refer e.g. to the works \cite{rolewicz86}, \cite{rolewicz87} and \cite{montesinos87} in connection with \emph{property $(\alpha)$}, with \emph{drop property}, and with an isomorphic characterization of reflexive Banach spaces. 

From the definition it is clear that the Kuratowski measure of $A$ is smaller than or equal to its diameter.  Obviously, equality does not always hold, but a fruitful relationship between the notion of $\Delta$-points and the Kuratowski measure of slices was discovered in \cite{almp22} and completed in \cite{veelipfunc}. In particular, the following result was obtained (see \cite[Corollary~2.2]{veelipfunc}).

\begin{thm}\label{thm:alpha_size_delta_points}
Let $X$ be a Banach space and let $x\in S_X$. If $x$ is a $\Delta$-point, then $\alpha(S)=2$ for every slice $S$ of $B_X$ containing $x$. Besides, $\alpha(S(x,\delta;B_{X^*}))=2$ in $B_{X^*}$ for every $\delta>0$.
\end{thm}

Observe that the converse does not hold in general, as the following example shows.

\begin{example}\label{exam:counteralphadelta}
Consider $X:=L_1([0,1])\oplus_\infty \ell_1$. It follows that both $X$ and $X^*$ enjoy the SD2P \cite[Remark 2.6]{blr14}, so Theorem \ref{thm:alpha_size_slices} below implies that given any slice $S=S(x^*,\delta;B_X)$ we have $\alpha(S)=2$ and the same holds for the slices in the dual. However, there are points which are not $\Delta$-points because $X$ fails the DLD2P \cite[Theorem 3.2]{ik} since $\ell_1$ fails it, so it remains to take any point $x\in S_X$ which is not a $\Delta$-point to get the desired counterexample.
\end{example}

Observe that the connection between having big slices in diameter and having big slices in Kuratowski index goes beyond Theorem \ref{thm:alpha_size_delta_points}. The following result was first pointed out in \cite{DGKR}.

\begin{thm}[\mbox{\textrm{\cite[Proposition~3.1]{DGKR}}}]
\label{thm:alpha_size_slices}
Let $X$ be a Banach space and let $\beta\in(0,2]$. The following assertions are equivalent.
\begin{enumerate}
    \item Every slice of $B_X$ has diameter greater than or equal to $\beta$. 
    \item Every slice of $B_X$ has Kuratowski measure greater than or equal to $\beta$.   
\end{enumerate}
\end{thm}

In this section, we aim to prove analogues to this result for relative weakly open subsets, as well as for convex combinations of slices or of weakly open sets, and to extend Veeorg's result to super $\Delta$-points and ccw $\Delta$-points.

\subsection{Kuratowski measure and diameter two properties.}
The analogue to Theorem \ref{thm:alpha_size_slices} for non-empty weakly open subsets is the following.

\begin{thm}\label{thm:alpha_size_weakly_open_sets}
Let $X$ be a Banach space and let $\beta\in(0,2]$. The following assertions are equivalent.
\begin{enumerate}
    \item Every non-empty relatively weakly open subset of $B_X$ has diameter greater than or equal to $\beta$.   
    \item Every non-empty relatively weakly open subset of $B_X$ has Kuratowski measure greater than or equal to $\beta$. 
\end{enumerate}
\end{thm}

\begin{proof}
(2)$\Rightarrow$(1) is immediate, so let us prove (1)$\Rightarrow$(2). To this end, fix $\beta\in(0,2]$ and assume that every non-empty relatively weakly open subset of $B_X$ has diameter greater than or equal to $\beta$. Then pick $\varepsilon>0$, and let us prove by induction on $n$ that for every non-empty relatively weakly open subset $W$ of $B_X$ and for every finite collection $C_1,\ldots, C_n$ of subsets of $X$ with $\diam(C_i)\leq \beta-\varepsilon$ for every $i$, we have that $W\not\subset \bigcup\limits_{i=1}^n C_i$. 

For $n=1$, it is clear since by assumption $\diam(W) \geq \beta > \beta-\eps $ for every non-empty relatively weakly open subset $W$ of $B_X$. 

So assume that the result is true for every non-empty relatively weakly open subset $W$ of $B_X$ and for every collection of $n$ sets, and let us prove the result for collections of $n+1$ sets.
To this end, consider $C_1,\ldots, C_n, C_{n+1}$ be subsets of $X$ with $\diam(C_i)\leq \beta-\varepsilon$ for every $i$. Observe that $\diam(C_i)=\diam(\overline{C_i}^w)\leq \beta-\varepsilon$ by $w$-lower semicontinuity of the norm of $X$, so that we may and do assume that $C_i$ is weakly closed for every $i$.

Observe that by the case $n=1$ we have that $W\not\subset C_{n+1}$, which means that $W\setminus C_{n+1}$ is non-empty. Moreover, it is a weakly open subset of $B_X$ since $C_{n+1}$ is assumed to be weakly closed, and by induction hypothesis we conclude that $W\backslash C_{n+1} \not\subset \bigcup\limits_{i=1}^n C_i$. In particular $W \not\subset \bigcup\limits_{i=1}^{n+1} C_i$ and the theorem is proved.
\end{proof}

Next, let us establish the analogue Theorem \ref{thm:alpha_size_slices} for convex combinations of slices. To this end, observe that by Bourgain lemma (see Lemma~\ref{lem:bourgain_lemma}) every convex combination of non-empty relatively weakly open subsets of $B_X$ contains a convex combination of slices of $B_X$. This assertion makes valid the following lemma which allows us to focus our attention in convex combination of weakly open subsets.

\begin{lemma}
    Let $X$ be a Banach space and $r>0$.
    \begin{enumerate}
        \item The following are equivalent:
        \begin{itemize}
            \item [(a)] Every convex combination of slices of $B_X$ has diameter greater than or equal to $r$.
            \item [(b)] Every convex combination of non-empty relatively weakly open subsets of $B_X$ has diameter greater than or equal to $r$.
        \end{itemize}
        \item The following are equivalent:
        \begin{itemize}
        \item [(a)] $\alpha(C)\geq r$ holds for every convex combination $C$ of slices of $B_X$.
            \item [(b)] $\alpha(D)\geq r$ holds for every convex combination $D$ of non-empty relatively weakly open subsets of $B_X$.
            \end{itemize}
    \end{enumerate}
\end{lemma}

Now we are able to give the following result.

\begin{thm}\label{thm:alpha_size_ccw}
Let $X$ be a Banach space and let $\beta\in(0,2]$. The following assertions are equivalent.
\begin{enumerate}
    \item Every convex combination of slices of $B_X$ has diameter greater than or equal to $\beta$.    
    \item Every convex combination of slices of $B_X$ has Kuratowski measure greater than or equal to $\beta$. 
\end{enumerate}
\end{thm}

\begin{proof}
    (2)$\Rightarrow$(1) is immediate, so let us prove (1)$\Rightarrow$(2). To this end, fix $\beta\in(0,2]$ and assume that every convex combination of non-empty relatively weakly open subsets of $B_X$ has diameter greater than or equal to $\beta$. Then pick $\varepsilon>0$, and let us prove by induction on $n$ that for every $D$ convex combination of non-empty relatively weakly open subsets of $B_X$ and for every finite collection $C_1,\ldots, C_n$ of subsets of $X$ with $\diam(C_i)\leq \beta-\varepsilon$ for every $i$, we have that $D\not\subset \bigcup\limits_{i=1}^n C_i$.

    For $n=1$ it is clear since by assumption $\diam(D) \geq \beta > \beta-\eps $ for every  convex combination of non-empty relatively weakly open subsets of $D$ of $B_X$. 

    Assume by inductive step that the result stands for $n$.

    Now pick $D$ convex combination of non-empty relatively weakly open subsets of $B_X$ and a finite collection $C_1,\ldots, C_{n+1}$ of subsets of $X$ with $\diam(C_i)\leq \beta-\varepsilon$ for every $i$. We can assume as in the proof of Theorem \ref{thm:alpha_size_weakly_open_sets} that every $C_i$ is weakly closed. Write $D=\sum_{i=1}^k \lambda_i W_i$. Observe that by the case $n=1$ we have that $D\not\subseteq C_{n+1}$, so there exists $z\in D\setminus C_{n+1}$. Since $z\in D$ we can write $z=\sum_{i=1}^k \lambda_i x_i$ where $x_i\in W_i$ holds for every $1\leq i\leq k$. Moreover, since $z=\sum_{i=1}^k \lambda_i x_i\notin C_{n+1}$, this means that $z=\sum_{i=1}^k\lambda_i x_i\in X\setminus C_{n+1}$, and the latter is a weakly open set. By a weak-continuity argument of the sum we can find weakly open subsets $V_i$ of $B_X$, with $x_i\in V_i$ for every $1\leq i\leq k$, satisfying that $z=\sum_{i=1}^k \lambda_i x_i\in \sum_{i=1}^n \lambda_i V_i\subseteq X\setminus C_{n+1}$. Up to taking smaller $V_i$, we can assume $V_i\subseteq W_i$ for every $i$. Now call $\tilde{D}:=\sum_{i=1}^k \lambda_i V_i$, which is a convex combination of weakly open subsets of $B_X$. By the inductive step we get that $\tilde{D}\not\subseteq \bigcup\limits_{j=1}^n C_j$, so there exists $y\in \tilde{D}$ with $y\notin C_j$ for $1\leq j\leq n$. Observe that the condition $V_i\subseteq W_i$ implies $\tilde{D}\subseteq D$, so $y\in D$ indeed. Moreover, $\tilde{D}\subseteq X\setminus C_{n+1}$ implies in particular $y\notin C_{n+1}$. This implies that $y\in D\setminus \bigcup\limits_{i=1}^{n+1} C_i$, which is precisely what we wanted to prove.
\end{proof}

\subsection{Kuratowski measure and $\Delta$-notions}

We now prove an analogue to Theorem \ref{thm:alpha_size_delta_points} for super $\Delta$-points.

\begin{theorem}\label{thm:alpha_size_superdelta_points}
Let $X$ be a Banach space and let $x\in S_X$ be a super $\Delta$-point. Then every non-empty relatively weakly open subset $W$ of $B_X$ containing $x$ satisfies that $\alpha(W)=2$.
\end{theorem}

The proof will be an obvious consequence of the following result.

\begin{proposition}
Let $X$ be a Banach space, $x\in S_X$ be a super $\Delta$ point, and $W$ be a weakly open subset of $B_X$ such that $x\in W$. 
Then, for every $\varepsilon>0$, there exists a sequence $\{x_n\}\subseteq W$ such that $\Vert x_i-x_j\Vert>2-\varepsilon$ holds for every $i\neq j$.   
\end{proposition}

\begin{proof} Set $\varepsilon>0$. Let us construct by induction a sequence $\{x_n\}$ satisfying that $\Vert x-x_i\Vert>2-\frac{\varepsilon}{2}$ and such that $\Vert x_i-x_j\Vert>2-\varepsilon$ for $i\neq j$.

Using that $x$ is a super $\Delta$ point select, by the definition of super $\Delta$, a point $x_1\in W$ with $\Vert x-x_1\Vert>2-\frac{\varepsilon}{2}$. 

Now, assume that $x_1,\ldots, x_n$ have been constructed and let us construct $x_{n+1}$. By the properties defining the sequence observe that, given $1\leq i\leq n$, we have $\Vert x-x_i\Vert>2-\frac{\varepsilon}{2}$, so we can find $g_i\in S_{X^*}$ with $\re g_i(x-x_i)>2-\frac{\varepsilon}{2}$, which implies $\re g_i(x)>1-\frac{\varepsilon}{2}$ and $\re g_i(x_i)<-1+\frac{\varepsilon}{2}$. Consequently
$$x\in V:=W\cap \bigcap\limits_{i=1}^n S\left(g_i,\frac{\varepsilon}{2}\right),$$
which is a weakly open set. Since $x$ is super $\Delta$ we can find $x_{n+1}\in V$ such that $\Vert x-x_{n+1}\Vert>2-\frac{\varepsilon}{2}$.  In order to finish the construction we only must to prove that $\Vert x_i-x_{n+1}\Vert>2-\varepsilon$ holds for every $1\leq i\leq n$. But this is clear because, given $1\leq i\leq n$, the condition $x_{n+1}\in V$ implies that $\re g_i(x_{n+1})>1-\frac{\varepsilon}{2}$, so 
$$\Vert x_{n+1}-x_i\Vert\geq \re g_i(x_{n+1}-x_i)>1-\frac{\varepsilon}{2}+1-\frac{\varepsilon}{2}=2-\varepsilon,$$
and the proof is finished.
\end{proof}

Note that a similar statement than Theorem~\ref{thm:alpha_size_superdelta_points} can be established for ccw $\Delta$ points.

\begin{theorem}
Let $X$ be a Banach space and let $x\in S_X$ be a ccw  $\Delta$-point. Then every non-empty convex combination $D$ of relatively weakly open subsets of $B_X$ containing $x$ satisfies that $\alpha(D)=2$.
\end{theorem}

As in the previous case, the proof follows directly from the next result.

\begin{proposition}
Let $X$ be a Banach space, $x\in S_X$ be a ccw $\Delta$ point, and $D$ a ccw of $B_X$ such that $x\in D$. 
Then, for every $\varepsilon>0$, there exists a sequence $\{x_n\}\subseteq D$ such that $\Vert x_i-x_j\Vert>2-\varepsilon$ holds for every $i\neq j$.
\end{proposition}

\begin{proof} Set $\varepsilon>0$. Write $D:=\sum_{i=1}^k \lambda_i W_i$ with $\lambda_i\neq 0$ for every $i$. Set $\delta:=\frac{\varepsilon}{2\min_{1\leq i\leq k}\lambda_i}$. Let us construct by induction a sequence $\{x_n\}\subseteq D$ satisfying that $\Vert x-x_i\Vert>2-\delta$ and such that $\Vert x_i-x_j\Vert>2-\varepsilon$ for $i\neq j$. Using that $x$ is a ccw $\Delta$ point select, by the definition of ccw $\Delta$, a point $x_1\in D$ with $\Vert x-x_1\Vert>2-\delta$. 

Now assume that $x_1,\ldots, x_n$ have been constructed and let us construct $x_{n+1}$. We can write $x=\sum_{j=1}^k \lambda_j x_j$ and $x_i:=\sum_{j=1}^k\lambda_j x_j^i$ as being elements of $D$.

By the properties defining the sequence, observe that, given $1\leq i\leq n$ we have $\Vert x-x_i\Vert>2-\delta$, so we can find $g_i\in S_{X^*}$ with $$\re g_i(x-x_i)=\sum_{j=1}^k \lambda_j \re g_i(x_j-x_j^i)>2-\delta=2-\frac{\varepsilon}{2\min_{1\leq j\leq n}\lambda_j}.$$ A convexity argument implies that $\re g_i(x_j-x_j^i)>2-\frac{\varepsilon}{2}$ holds for every $1\leq j\leq k$, which implies that  
$$
\re g_i(x_j)>1-\frac{\varepsilon}{2}\ \text{ and } \ \re g_i(x_j^i)<-1+\frac{\varepsilon}{2}.
$$
Observe that
$$x_j\in V_i:=W_i\cap \bigcap\limits_{i=1}^n S\left(g_i,\frac{\varepsilon}{2}\right),$$
which is a weakly open set. Since $x$ is a ccw $\Delta$-point and $x\in \sum_{j=1}^k \lambda_j V_j$, we can find a point  $x_{n+1}=\sum_{j=1}^k \lambda_j z_j\in \sum_{j=1}^k \lambda_j V_j\subseteq D$ such that $\Vert x-x_{n+1}\Vert>2-\delta$.  In order to finish the construction we only must to prove that $\Vert x_i-x_{n+1}\Vert>2-\varepsilon$ holds for every $1\leq i\leq n$. Given $1\leq j\leq k$, the condition $z_j\in V_j$ implies $\re g_i(z_j)>1-\frac{\varepsilon}{2}$. On the other hand, $\re g_i(x_j^i)<-1+\frac{\varepsilon}{2}$, so 
\[\begin{split}\Vert x_{n+1}-x_i\Vert\geq \re g_i(x-x_i)=\sum_{j=1}^k \lambda_j \re g_i(z_j-x_j^i)>(2-\varepsilon)\sum_{j=1}^k\lambda_j=2-\varepsilon,
\end{split}\]
and the proof is finished.
\end{proof}


\section{Commented open questions}\label{section:conclusionopenquestions}

The only implications between properties which is not known to hold or not is the following one (see Figure~\ref{figure02} in page~\pageref{figure02}).

\begin{quest}\label{problem:ccs-delta-implies-super-delta}
Let $X$ be a Banach space and let $x\in S_X$ be a ccs $\Delta$-point. Is $x$ a super $\Delta$-point?
\end{quest}

Let us give some comments on this question. On the one hand, it may look that the answer is positive by Bourgain's lemma (Lemma~\ref{lem:bourgain_lemma}), but this lemma \emph{does not say} that, in general, given an element $x$ of a relative weak open subset $W$ of $B_X$, there is a convex combination of slices of $B_X$ contained in $W$ and \emph{containing $x$}. The later happens when $x\in \co(\preext{B_X})$ (see Remark~\ref{rem:bourgain}) so, the answer to Question~\ref{problem:ccs-delta-implies-super-delta} is positive in this case.
On the other hand, a possible counterexample to this problem could be the molecules in Examples \ref{example:veefree} or \ref{example:Lipfree-ANPP}, which are known to be ccs $\Delta$-points and are extreme points but not preserved extreme points (hence they do not belong to the convex hull of the set of preserved extreme points).
A way to show that these molecules are not super $\Delta$-points would be to investigate whether RNP spaces may contains super $\Delta$-points.

Theorem~\ref{thm:one-unconditional_no_ccs_nor_super_Delta} states that \emph{real} Banach spaces with a one-unconditional basis do neither contain super $\Delta$-points nor ccs $\Delta$-points. It is likely that such result also holds true in the complex setting since we do believe  that the preliminary results from \cite{almt21} are also valid for complex scalars, provided that one works with the suitable notion of one-uncondional bases (for which \cite[Proposition~2.3]{almt21} holds). Also, we also expect that the results there can be easily extended to one-uncondional FDDs. Yet, since a sharper version of this result was obtained in Proposition~\ref{prop:operators_and_super_Delta_points} for super $\Delta$-points in a very general setting, it is natural to ask whether improved results could be simultaneously obtained in both directions for ccs $\Delta$-points by proving an analogue to Proposition~\ref{prop:operators_and_super_Delta_points}. So let us ask the following.

\begin{quest}
Let $X$ be a Banach space, and let us assume that there exists a subset $\mathcal{A}\subseteq \mathcal{F}(X,X)$ satisfying that $\sup\bigl\{\norm{\Id-T}\colon T\in \mathcal{A}\bigr\}<2$ and that for every $\eps>0$ and every $x\in X$, there exists $T\in\mathcal{A}$ such that $\norm{x-T x}<\eps$. Can $X$ contain a ccs $\Delta$-point?
\end{quest}

A negative answer to this question would be interesting, since it would provide an example of a ccs $\Delta$-point that is not a super $\Delta$-point, hence a negative answer to Question~\ref{problem:ccs-delta-implies-super-delta}.

Another interesting question could be if a point of continuity could be a ccs $\Delta$-point. Let us formalize the questions.

\begin{quest}
Let $X$ be a Banach space.
\begin{enumerate}
  \item Does $X$ fail the RNP (or even the CPCP) if contains a super $\Delta$-point or a super Daugavet point?
  \item Is it possible for a point of continuity being a ccs $\Delta$-point?
\end{enumerate}
\end{quest}

The surprising examples given in Section~\ref{sec:examples_and_counterexamples} shows that the mere existence of some diametral notions (but ccs Daugavet points) on a Banach space does not imply that the whole space has any diameter two property nor the Daugavet property. Our question here is how many diametral points has to contain a Banach space to have any diameter two property or the Daugavet property or fails to have the RNP or one-unconditional basis. 

\begin{quest}
How big can be the set of Daugavet points, super Daugavet points, $\Delta$-points, super $\Delta$-points, or ccs $\Delta$-points in a Banach space with the Radon-Nikod\'{y}m property, or with the CPCP, or being strongly regular, or having one-unconditional basis?
\end{quest}

Concerning isometric consequences of the existence of diametral points, there are some recent results showing that a Banach space containing a $\Delta$-point cannot be uniformly non-square \cite{almp22} or even locally uniformly non-square \cite{KaLeeTag2022}, or asymptotic uniformly smooth \cite{almp22,veelipfunc}. Also, a Banach space having an unconditional basis with suppression-unconditional constant less that $2$ cannot contains super $\Delta$-points and a Banach space containing a ccs Daugavet point has the SD2P. Taking into account that it is not known if there exists an stictly convex Banach space with the Daugavet property (see \cite[Section 5]{Kadets96}), the following question makes sense. Recall that Paragraph~\ref{subsubsect:MLUR} shows an example of an strictly convex Banach space in which every norm-one element is a ccs $\Delta$-point and a super $\Delta$-point, but it does not contain any Daugavet point by the way in which it is constructed.

\begin{quest}
Is there an strictly convex Banach space containing a Daugavet point?
\end{quest}

In view of Proposition~\ref{prop:extreme_diametral_delta} and of Theorem~\ref{thrm:extreme_Delta_molecules}, the following question makes sense.

\begin{quest}
Let $X$ be a Banach space. Suppose that $x\in \ext{B_X}$ is a $\Delta$-point, does this imply that $x$ is a ccs $\Delta$-point or a super $\Delta$-point?
\end{quest}

By now, the only \emph{isomorphic} restriction which is known for a Banach space to contain $\Delta$-points or even Daugavet points is that it cannot be finite-dimensional. It would be interesting to find some more.

\begin{quest}
Find isomorphic restrictions for a Banach space to contain $\Delta$-points or any of the other diametral notions. In particular, is it possible for a reflexive or even super-reflexive Banach space to contain $\Delta$-, super $\Delta$-, ccs $\Delta$-, Daugavet or super Daugavet points?
\end{quest}

The results about absolute sums in Subsection~\ref{subsec:absolutesums} are not complete in the case of super Daugavet points and they are even less clear in the case of ccs notions. Here are two possible questions.

\begin{quest}
Let $X$, $Y$ be Banach spaces and let $N$ be an absolute sum.
\begin{enumerate}
  \item If $N$ is $A$-octahedral, $x\in S_X$ and $y\in S_Y$ are super Daugavet points, is $(ax,by)$ a super Daugavet point in $X\oplus_N Y$ when $a,b$ satisfy the conditions in the definition of $A$-octahedrality?
  \item If $N$ is the $\ell_\infty$-sum, $x\in S_X$ and $y\in S_Y$ are ccs $\Delta$-points, are the elements of the form $(a x,b y)$ ccs $\Delta$-points in $X\oplus_\infty Y$ for $a,b\in [0,1]$ with $\max\{a,b\}=1$?
\end{enumerate}
\end{quest}

It would be also desirable to study the reversed results to those in Subsection~\ref{subsec:absolutesums} as it is done in \cite{pirkthesis} for $\Delta$-points and Daugavet points (see the tables in pages 86 and 87 of 
\cite{pirkthesis}).

\begin{quest}
Let $X$, $Y$ be Banach spaces, let $N$ be an absolute sum, $x\in S_X$, $y\in S_Y$, and $a,b\geq 0$ such that $N(a,b)=1$. Discuss what happens with $x$ and $y$ supposing that $(ax,by)$ satisfies any of the six diametral notions.
\end{quest}

It maybe the case that some of the arguments given in Subsections \ref{subsec:L_1_preduals} and \ref{subsec:L_1_spaces} 
can be adapted to other classes of Banach spaces. We propose some possibilities.

\begin{quest}
Characterize the six diametral notions in uniform algebras, in Lorentz spaces and their isometric preduals, and in some vector-valued function spaces as $C(K,X)$ or $L_\infty(\mu,X)$ spaces.
\end{quest}

The relations between the weak-star versions of the diametral points (see Remark~\ref{rem:weak*_diametral_points}) are not yet clear. For instance, the following questions arises.

\begin{quest}
Let $X$ be a Banach space and $x\in S_X$.
\begin{enumerate}
  \item Is $J_X(x)$ a ccs $\Delta$-point in $X^{**}$ if $x$ is a ccs $\Delta$-point?
  \item Is there any relationship between the DD2P in $X$ and the weak-star super $\Delta$-points in $S_{X^*}$?
\end{enumerate}
\end{quest}

As commented in Remark~\ref{remark:super-delta-closed-to-strongly-exposed}, a Banach space $X$ containing a sequence $(y_n)$ of super $\Delta$-points such that the distance of $y_n$ to the set of strongly exposed points of $B_X$ is going to zero. But the following question remains open.

\begin{quest}
Can a super $\Delta$-point (or even a $\Delta$-point) belong to the closure of the set of denting points?
\end{quest}

The answer to the next question on the behaviour of $\Delta$- and super $\Delta$-points in rays is still unknown, as we commented in Section~\ref{sec:inside_of_the_unit_ball}.

\begin{quest}
Let $X$ be a Banach space and let $x\in S_X$.
\begin{enumerate}
    \item If $rx$ is a $\Delta$-point for some $0<r<1$, does this imply that $x$ is a $\Delta$-point?
    \item If $rx$ is a super $\Delta$-point for some $0<r<1$, does this imply that $x$ is a auper $\Delta$-point?
    \item If $x$ is a super $\Delta$-point, does this imply that $rx$ is a super $\Delta$-point for all $0<r<1$?
\end{enumerate}
\end{quest}

As we proved in Section~\ref{section:Kuratowski_measure}, every relative weakly open subset which contains a super $\Delta$-point (respectively, a ccw $\Delta$-point) has Kuratowski measure $2$. Our proofs do not seem to work for convex combination of slices, so let us ask the following.  

\begin{quest}
If a $ccs$ of the unit ball contains a ccs $\Delta$-point, does it necessarily have maximal Kuratowski measure? 
\end{quest}

\section*{Acknowledgments}

Part of this work was done during the visit of the second named author at the University of Granada in September 2022. He wishes to thank his colleagues for the warm welcome he received, and to thank all the people that made his visit possible.

The authors thank Gin\'es L\'opez-P\'erez for fruitful conversations on the topic of the paper, specially for providing enlightening ideas in connection with the example of Subsection \ref{subsect:superdauganotccsdelta}.

The authors also thank Trond A. Abrahamsen, Andr\'e Martiny, and Vegard Lima for valuable discussions on the topic of the paper, and in particular for pointing out the ccs version of \cite[Proposition~2.12]{almt21} that is presented in Subsection~\ref{subsec:one-unconditional_bases}.

\end{document}